\documentclass[11pt, notitlepage]{article}

\usepackage{import}
\usepackage{expdlist}
\usepackage{circuitikz}
\usepackage{MathPreambleArt}

\makeatletter
\def\thm@space@setup{%
  \thm@preskip=\parskip \thm@postskip=0pt
}
\makeatother

\newcommand\numberthis{\addtocounter{equation}{1}\tag{\theequation}}

\title{The Equivariant Fried Conjecture for Suspension Flow of an Equivariant Isometry}
\author{Peter Hochs\footnote{Institute for Mathematics, Astrophysics and Particle Physics, Radboud University: p.hochs@math.ru.nl}\; and Christopher Pirie\footnote{Institute for Mathematics, Astrophysics and Particle Physics, Radboud University: c.pirie@math.ru.nl}}

\date{}

\begin{document}

\maketitle

\begin{abstract}
The Fried conjecture states that the Ruelle dynamical $\zeta$-function of a flow on a compact maniofold has a well-defined value at $0$, whose absolute value equals the Ray--Singer analytic torsion invariant.  The first author and Saratchandran proposed an equivariant version of the Fried conjecture for locally compact unimodular groups acting properly, isometrically, and cocompactly on  Riemannian manifolds. In this paper we prove the equivariant Fried conjecture for the suspension flow of an equivariant isometry of a Riemannian manifold in several cases. These include the case where the group is compact, the case where the group element in question has compact centraliser and closed conjugacy class, and the case of the identity element of a non-compact discrete group. 
\end{abstract}

\tableofcontents

\section{Introduction}

The Ruelle (dynamical) $\zeta$-function is a $\zeta$-function encoding information about the periodic orbits of a flow on a compact manifold. The Fried conjecture states that, for a large class of flows, the value of the Ruelle $\zeta$-function at 0 equals the Ray--Singer analytic torsion of the manifold. In \cite{hochs2023ruelle}, Hochs and Saratchandran propose an equivariant version of this conjecture for proper cocompact actions on possibly non-compact manifolds. In this paper, we prove the equivariant Fried conjecture for suspension flow of an equivariant isometry under natural assumptions.

\subsection*{Motivation}

Let $M$ be a compact manifold and $\vphi$ a flow on $M$. Suppose that for every flow curve $\gamma$ of $\vphi$ satisfying $\gamma(l) = \gamma(0)$ for some $l \in (0,\infty)$, the map
\[
    P_\gamma : T_{\gamma(0)}M/\R \gamma'(0) \to T_{\gamma(0)}M/\R \gamma'(0)
\]
induced by $D_{\gamma(0)}\vphi_t$ does not have 1 as an eigenvalue. Then $\sgn\big(\det(1 - P_\gamma)\big)$ is well-defined and we say that $\vphi$ is nondegenerate. An example of such a flow is geodesic flow on the unit sphere bundle of the tangent bundle of a compact Riemannian manifold with negative sectional curvature. Let $\nabla^E$ be a flat connection on a vector bundle $E \to M$ over $M$. For every flow curve as above, set $\rho(\gamma)$ to be parallel transport along $\gamma$ from $t = 0$ to $t=-l$. Also set
\[
    T_\gamma^\# = \min\{t > 0 : \gamma(t) = \gamma(0)\}
\]
to be the primitive period of the flow curve $\gamma$. Finally suppose that the set of periodic flow curves is countable. Formally, write
\begin{equation}\label{Eq: Intro Ruelle Classical}
    R_{\vphi,\nabla^E}(\sigma) = \exp\Big(\sum_l\frac{e^{-l\sigma}}{l}\sum_\gamma \sgn\big(\det(1 - P_\gamma)\big)\tr\big(\rho(\gamma)\big)T_\gamma^\#\Big)
\end{equation}
where the first sum is over all possible periods $l$ and the second sum is over all possible flow curves with period $l$. If \eqref{Eq: Intro Ruelle Classical} converges to a holomorphic function for $\sigma \in \C$ with $\re(\sigma) > 0$ large enough, then we say that $R_{\vphi,\nabla^E}$ is well-defined, and is the \textbf{Ruelle dynamical $\zeta$-function} for $\vphi$. 

Let $\Delta_E = \nabla^E(\nabla^E)^* + (\nabla^E)^*\nabla^E$ be the Laplacian on differential forms with values in the vector bundle $E$. Let $F$ be the number operator on such forms, which equals multiplication by $p$ on degree $p$-forms. Finally let $P$ denote the orthogonal projection onto the kernel of $\Delta_E$. Then the \textbf{Ray--Singer analytic torsion} of $\nabla^E$ is
\begin{equation}\label{Eq: Intro Ray-Singer}
    T(\nabla^E) = \exp\Big(-\frac{1}{2}\left.\frac{d}{ds}\right|_{s=0}\frac{1}{\Gamma(s)}\int_0^\infty t^{s-1}\Tr\big((-1)^FF(e^{-t\Delta_E} - P)\big)dt\Big).
\end{equation}
The Fried conjecture \cite{Lefschetz.Flows.Fried} states that if $\ker(\Delta_E) = 0$, then $R_{\vphi,\nabla^E}$ has a meromorphic extension to $\C$, is holomorphic at 0, and 
\[
    \abs{R_{\vphi,\nabla^E}(0)} = T(\nabla^E).
\]
While the conjecture is known to false in general, see \cite{Analytic.Torsion.Fried.Survey, Schweitzer.SeifertConj, Wilson.MinimalSet}; it has been proven in several cases, see for example \cite{HypoellipticLaplacianOrbitalIntegrals, Dang2020, Fried1988, Moscovici1991, Müller2021, Morgado.FriedAnosov, Analytic.Torsion.Dynamics, Shen.MorseSmaleFlow, Shen.TorsionOrbifolds, Wotzke.Torsion, Yamaguchi+2022+155+176}. For a more detailed overview of the Fried conjecture, we refer to \cite{Analytic.Torsion.Fried.Survey}.

In \cite{hochs2023ruelle, hochs2022equivariant}, Hochs and Saratchandran introduce equivariant versions of the Ruelle $\zeta$-function and analytic torsion, respectively, for proper actions by locally compact unimodular groups on (possibly non-compact) manifolds $M$ such that $M/G$ is compact. 

For the Ruelle dynamical $\zeta$-function, the main idea is that rather than consider periodic flow curves satisfying $\gamma(l) = \gamma(0)$, they consider $g$-periodic flow curves satisfying
\begin{equation}\label{Eq: Intro g-Periodic}
    \gamma(l) = g\gamma(0)
\end{equation}
for some $l \in \R \setminus \{0\}$ and fixed $g \in G$. More explicitly, suppose the $G$-action on $M$ lifts to $E$ and preserves $\nabla^E$. Further suppose that the flow $\vphi$ is $G$-equivariant. The condition that $\vphi$ is nondegenerate is replaced with the condition that $\vphi$ is $g$-nondegenerate. This means that for all flow curves satisfying \eqref{Eq: Intro g-Periodic}, the eigenspace for eigenvalue 1 of the map
\[
    D_{\gamma(0)}\vphi_l \circ g^{-1} : T_{\gamma(0)}M \to T_{\gamma(0)}M
\]
is one-dimensional, spanned by $\gamma'(0)$. Then, restricting to the case that $G$ is compact for simplicity, the equivariant Ruelle $\zeta$-function is given by
\begin{equation}\label{Eq: Intro Equivariant Ruelle}
\begin{aligned}
    R^g_{\vphi,\nabla^E}&(\sigma) \\
    &= \exp\Big(\sum_l \frac{e^{-\abs{l}\sigma}}{\abs{l}}\sum_\gamma \sgn\big(\det(1 - D_{\gamma(0)}(\vphi_l \circ g^{-1})|_{T_{\gamma(0)}M/\R \gamma'(0)})\big)\tr\big(\rho_g(\gamma)\big)T_\gamma^\#\Big).
\end{aligned}
\end{equation}
Here the first sum is over all $l$ such that \eqref{Eq: Intro g-Periodic} holds for some flow curve $\gamma$, the second sum is over all flow curves such that \eqref{Eq: Intro g-Periodic} holds for $l$, and the operator $\rho_g(\gamma)$ is parallel transport along $\gamma$ from $t=0$ to $t=-l$ composed with $g : E_{\gamma(-l)} \to E_{\gamma(0)}$. 

For analytic torsion, Hochs and Saratchandran replace the trace of the heat kernel in \eqref{Eq: Intro Ray-Singer} with the $g$-trace from equivariant index theory. In the case that $G$ is a compact group, the equivariant analytic torsion is given by the formula
\begin{equation}\label{Eq: Intro Equivariant Torsion}
    T_g(\nabla^E) = \exp\Big(-\frac{1}{2}\left.\frac{d}{ds}\right|_{s=0}\frac{1}{\Gamma(s)}\int_0^\infty t^{s-1}\Tr\big(g(-1)^FF(e^{-t\Delta_E} - P)\big)dt\Big).    
\end{equation}
We note that in the case that $G$ is a non-compact group convergence of $T_g(\nabla^E)$ becomes a subtle problem, see Section 3 of \cite{hochs2022equivariant} for more details.

Using \eqref{Eq: Intro Equivariant Ruelle} and \eqref{Eq: Intro Equivariant Torsion}, Hochs and Saratchandran propose an equivariant version of the Fried conjecture. This is the problem to find conditions, in addition to the vanishing of the $L^2$-kernel of $\Delta_E$, such that $R^g_{\vphi,\nabla^E}$ extends to 0, with 
\[
    R^g_{\vphi,\nabla^E}(0) = T_g(\nabla^E)^2.
\]
In the case that $G$ is compact and $g=e$, this reduces to the classical Fried conjecture. Moreover, in \cite{hochs2023ruelle} they show that the equivariant Fried conjecture is true in the cases of the real line and circle acting on themselves, and for certain discrete subgroups of the Euclidean motion group acting on $\R^n$.

However, for instance, it is unknown whether the equivariant Fried conjecture holds in cases where the classical Fried conjecture holds, $G$ is a compact, and $g \neq e$ is not the identity element of the group. Our results provide some work in this direction.

\subsection*{Results}

As mentioned previously, in this paper we prove that, under some additional assumptions, the equivariant Fried conjecture holds in the case of suspension flow of an equivariant isometry. 

Let $Y$ be a smooth oriented Riemannian manifold. Let $G$ be a locally compact unimodular group acting properly and isometrically on $Y$ such that $Y/G$ is compact, and let $T\colon Y \to Y$ be an $G$-equivariant isometry. Let $M$ be the quotient of $Y \times \R$ by the equivalence relation generated by setting $(Ty, s - 1) \sim (y,s)$. Then we say that $M$ is the \textbf{suspension} of $Y$ (with respect to $T$). It follows that $M$ is a fibre bundle over $S^1$ with fibre diffeomorphic to $Y$, and in fact all fibre bundles over $S^1$ can be obtained as the suspension of some space $Y$. 

On $M$ there is a simple flow $\vphi$ defined by
\[
    \vphi_t[y,s] = [y, t + s],
\]
called the suspension flow. In the case that $Y$ is compact, the classical Fried conjecture for suspension flow follows as a consequence of a fibration formula for analytic torsion, and a fixed point formula. We refer to Section 3 of \cite{Analytic.Torsion.Fried.Survey} for more details. 

In the equivariant setting, we note that $M$ inherits a $G$-action from the $G$-action on $Y$, and that suspension flow is $G$-equivariant for this action. Fix $g \in G$, and suppose that $g$ lies in a compact subgroup of $G$. Further suppose that for all $n \in \Z\setminus \{0\}$ the fixed point set $Y^{g^{-1}T^n}$ is either empty or discrete, and contain only nondegenerate fixed points, i.e.\ $D_y(g^{-1}T^n)$ does not have 1 as an eigenvalue for $y \in Y^{g^{-1}T^n}$. Then it follows that suspension flow is $g$-nondegenerate by \thref{Prop: Suspension Flow NonDegen}.

Then \thref{Thm: Most General Fried Suspension}, which is the most general result of this paper, states that under the compatibility of certain so-called cutoff functions, that the equivariant Fried conjecture for suspension flow holds if the equivariant analytic torsion for the suspension is well-defined. 

As a corollary, in the case that $G$ is a compact group, the equivariant analytic torsion is well-defined, and so the equivariant Fried conjecture holds for suspension flow by \thref{Thm: Most General Fried Suspension}. Moreover, in this case the equivariant analytic torsion is explicitly computable. To be more precise, if $H^\bullet(Y,E_1)$ denotes the de Rham cohomology of differential forms with values in a $G$-equivariant flat vector bundle $E_1$ over $Y$, then we have
\begin{equation}
    T_g(\nabla^E) = \exp\Big(\sum_{n \in \Z \setminus \{0\}} \frac{1}{2\abs{n}}\sum_{j=0}^{\dim Y}\Tr\big((-1)^jg^*(T^*)^n|_{H^j(Y,E_1)}\big)\Big),
\end{equation}
where $g^*$ and $T^*$ are induced maps on cohomology, and $\nabla^E$ is a flat connection on a vector bundle $E$ over $M$ which is induced from $E_1$. Moreover, on the Ruelle side we obtain
\begin{equation}\label{Eq: Intro Ruelle Suspension}
    R^g_{\vphi,\nabla^E}(\sigma) = \exp\Big(\sum_{n \in \Z \setminus \{0\}} \frac{e^{-\abs{n}\sigma}}{\abs{n}}\sum_{j=0}^{\dim Y}\Tr\big((-1)^jg^*(T^*)^n|_{H^j(Y,E_1)}\big)\Big),
\end{equation}
from which we see that setting $\sigma = 0$ in \eqref{Eq: Intro Ruelle Suspension} gives $R^g_{\vphi,\nabla^E}(0) = T_g(\nabla^E)^2$, and the equivariant Fried conjecture holds. 

More generally, in that case that $G$ is non-compact, we obtain positive results for the equivariant Fried conjecture for suspension flow in the cases where
\begin{enumerate}[label=(\Roman*)]
    \item the conjugacy class of $g$ in $G$ is closed and the centraliser of $g$ in $G$ is compact; and 
    \item $g=e$ is the identity element of a discrete group. 
\end{enumerate}
The proof of case (I) involves a large majority of the work in this paper, and thus may be viewed as our main result. 

While suspension flow is a very basic case of the classical Fried conjecture, we will see that there are some new ingredients involved in proving the equivariant case for noncompact groups. We view these 
 first results on the equivariant Fried conjecture for a general class of flows as a sign that there is hope that this conjecture may be true in other cases as well. 

\subsection*{Outline of this paper}

In Section \ref{Sub: Prelims} we recall the definitions of equivariant analytic torsion, the equivariant Ruelle dynamical $\zeta$-function, the equivariant Fried conjecture, and finally state the main results of this paper: \thref{Thm: Most General Fried Suspension} and Corollaries \ref{Cor: Equivariant Fried Compact}, \ref{Cor: Equivariant Fried Identity}, and \ref{Cor: Equivariant Fried Compact Z}. In Section \ref{Sub: Proper & Euler} we prove some results on proper group actions, define a equivariant generalisation of the Euler characteristic, and prove a generalisation of the Atiyah--Bott Lefschetz fixed-point theorem for non-compact manifolds: \thref{Prop: Euler Char as Fixed Points}. In Section \ref{Sec: Fibration Formula} we prove a fibration formula for equivariant analytic torsion which we use in Section \ref{Sec: Suspension Flow} to calculate the equivariant analytic torsion for the suspension. In Section \ref{Sec: Ruelle Suspension} we calculate the the equivariant Ruelle dynamical $\zeta$-function for suspension flow. Finally in Section \ref{Sec: Suspension Flow} we prove the results stated in Subsection \ref{Subsec: Results}. In Appendix \ref{App: Relation Classical} we explicitly show how our results recover the classical Fried conjecture for suspension flow in the case that $G$ is compact and $g = e$ is the identity element of the group.

\subsection*{Acknowledgements}

We thank Bingxiao Liu and Shu Shen for helpful conversations. 
This project was supported by the Netherlands Organisation for Scientific Research NWO, through ENW-M grant OCENW.M.21.176.

\section{Preliminaries and Results}\label{Sub: Prelims}

Here we recall the definitions and properties of equivariant analytic torsion and the equivariant Ruelle dynamical $\zeta$-function as defined in \cite{hochs2023ruelle} and \cite{hochs2022equivariant}. We also set up notation that will be used throughout this paper. Finally, at the end of this section we present the main results of this paper.

\subsection{Equivariant Analytic Torsion}\label{Subsec: Torsion Prelims}

Let $M$ be a complete, oriented Riemannian manifold, and $G$ a unimodular locally compact group acting properly and isometrically on $M$, and preserving the orientation. Suppose that $G$ acts cocompactly, i.e.\ $M/G$ is compact. Denote the Riemannian density on $M$ by $dm$. 

Fix a Haar measure $dx$ on $G$. Let $g \in G$, and let $Z \coloneqq Z_G(g) < G$ be its centraliser. Suppose that $Z$ is unimodular, or equivalently, that $G/Z$ has a $G$-invariant measure $d(hZ)$. Let $\psi \in C^\infty_c(M)$ be a cutoff function for the action, i.e.\ $\psi$ is non-negative and for all $m \in M$,
\[
    \int_G \psi(xm)\;dx = 1.
\]
Such a function exists as the action is proper and cocompact.

Let $W \to M$ be a $G$-equivariant Hermitian vector bundle. 

\begin{definition}[\cite{FPF.and.Character.Formula, 10.1093/imrn/rnab324}]
Let $T$ be a $G$-equivariant operator from $\Gamma^\infty_c(W)$ to $\Gamma^\infty(W)$. If $T$ has a smooth kernel $\kappa$, and the integral
\begin{equation}\label{Eq: g-trace}
    \Tr_g(T) \coloneqq \int_{G/Z}\int_M \psi(hgh^{-1}m)\tr\big(hgh^{-1}\kappa(hg^{-1}h^{-1}m,m)\big)\; dm \; d(hZ)
\end{equation}
converges absolutely, then $T$ is \textbf{$g$-trace class}, and the value of the integral is the \textbf{$g$-trace} of $T$. 
\end{definition}

\begin{remark}
From the $G$-invariance of $\kappa$, and the trace property of the fibrewise trace, we see that using substitution $m' = hgh^{-1}m$ in the integral on the right hand side of \eqref{Eq: g-trace} that
\[
    \Tr_g(T) = \int_{G/Z}\int_M \psi(m')\tr\big(hgh^{-1}\kappa(hg^{-1}h^{-1}m',m')\big)\; dm' \; d(hZ).
\]
We will use both expressions for the $g$-trace interchangeably.
\end{remark}

We note that by Lemma 3.1 in \cite{FPF.and.Character.Formula} the $g$-trace is independent of the choice of cutoff function. Moreover if $e$ is the identity element of the group $G$, and $\psi$ is the square of a smooth function, then $\Tr_e(T) = \Tr(\sqrt{\psi}T\sqrt{\psi})$ which is the von Neumann trace in \cite{AST_1976__32-33__43_0}.  

Let $E \to M$ be a Hermitian $G$-vector bundle. Let $\nabla^E$ be a Hermitian $G$-invariant, flat connection on $E$, assuming this exists. We also denote by $\nabla^E$ the induced operator on $E$-valued differential forms. Consider the Laplacian $\Delta_E \coloneqq (\nabla^E)^*\nabla^E + \nabla^E(\nabla^E)^*$ acting on $E$-valued differential forms. We denote its restriction to $p$-forms by $\Delta_E^p$.

From now on, the kernel of a Laplace-type operator like $\Delta_E$ will always mean its kernel on the space of \textbf{square-integrable} $E$-valued differential forms. We denote the kernel on such forms by $\ker(\Delta_E)$. Let $P^E$ be the orthogonal projection of square-integrable forms onto $\ker(\Delta_E)$, and $P_p^E$ its restriction to $p$-forms.

Let $F$ be the number operator on $\extp \bullet T^*M \otimes E$, which is equal to $p$ on $\extp p T^*M \otimes E$. If $e^{-t\Delta_E} - P^E$ is $g$-trace class, we write
\[
    \mathcal T_g(t) \coloneqq \Tr_g\big((-1)^FF(e^{-t\Delta_E} - P^E)\big).
\]

\begin{definition}[\cite{hochs2022equivariant}]\thlabel{Def: Torsion}
Suppose $e^{-t\Delta_E} - P^E$ is $g$-trace class for all $t > 0$. For $\sigma \in \C$ with $\re(\sigma) > 0$ such that the following converges, define $T_g(\nabla^E,\sigma)$ by the expression
\begin{equation}\label{Eq: Torsion Def}
    T_g(\nabla^E,\sigma) \coloneqq \exp\left(-\frac{1}{2} \left.\frac{d}{ds}\right|_{s = 0}\frac{1}{\Gamma(s)}\int_0^\infty t^{s-1}e^{-\sigma t}\mathcal T_g(t)\; dt\right).
\end{equation}
If $T_g(\nabla^E,\sigma)$ is holomorphic for $\re(\sigma)$ large enough, and defines a meromorphic extension which is holomorphic at $\sigma = 0$, then we say that the equivariant analytic torsion is \textbf{well-defined}. In this case the value of \eqref{Eq: Torsion Def} at $\sigma = 0$ is the \textbf{equivariant analytic torsion} at $g$, and we write
\[
    T_g(\nabla^E) \coloneqq T_g(\nabla^E,0).
\]
\end{definition}

When $G$ is a compact group and $e$ is the identity element of $G$, then $T_e(\nabla^E)$ is the Ray--Singer torsion $T_\rho(M)$ where $\rho$ is the representation of the fundamental group corresponding to the vector bundle $E$, see \cite{Analytic.Torsion.RS}. When $G$ is an arbitrary group and $g=e$ then $T_g(\nabla^E)$ is the $L^2$-analytic torsion of Lott and Mathai \cite{10.4310/jdg/1214448084,LOTT19991, MATHAI1992369, 10.1215/ijm/1403534491}. For more on equivariant analytic torsion as defined, and its links to other version of torsion with group actions, we refer back to \cite{hochs2022equivariant}.

The following lemma, \thref{Lem: Alternative Convergence Torsion}, provides some conditions for the convergence of analytic torsion and will be used throughout this paper. However, before we state the lemma we need the definition of a \v Svarc--Milnor function.

\begin{definition}
A \textbf{\v Svarc--Milnor function} for the action of $G$ on $M$, with respect to $g \in G$, is a function $l : G \to [0,\infty)$ such that:
\begin{enumerate}[label=(\Roman*)]
    \item for all $c > 0$,
    \begin{equation}\label{Eq: SM 1}
        \int_{G/Z}e^{-cl(hgh^{-1})^2}d(hZ)
    \end{equation}
    converges;
    \item the set
    \begin{equation}\label{Eq: SM 2}
        X_r \coloneqq \{hZ \in G/Z : l(hgh^{-1}) \le r\}
    \end{equation}
    is compact for all $r \ge 0$; 
    \item for all compact subsets $K \subseteq M$, there exist $a,b > 0$ such that for all $m \in K$ and $x \in G$,
    \begin{equation}\label{Eq: SM 3}
        d(xm,m) \ge al(x) - b;
    \end{equation}
    \item for all $x,y \in G$, we have
    \begin{equation}
        l(xy) \le l(x) + l(y);
    \end{equation}
    \item for all $x \in G$, we have $l(x) = l(x^{-1})$.
\end{enumerate}
\end{definition}

\begin{example}
If $G/Z$ is compact, then the zero function is a \v Svarc--Milnor function.
\end{example}

\begin{example}
If $G = \Gamma$ is discrete and finitely generated, then the word length metric for a finite generating set is a \v Svarc--Milnor function by the \v Svarc--Milnor lemma, see Lemma 2 in \cite{milnor1968note} or Proposition 8.19 in \cite{bridson2011metric}.
\end{example}

\begin{example}
Let $G$ is a connected real semisimple Lie group, and $g$ a semisimple element. Then the distance to the identity element of $G$ with respect to any left-invariant Riemannian metric is a \v Svarc--Milnor function for the action of left-multiplication of $G$ on itself. See Section 3 of \cite{hochs2022equivariant} or Lemma 5.13 in \cite{Piazza2025} for more details. 
\end{example}

\begin{remark}
These examples show that the existence of a \v Svarc--Milnor function is a mild condition. However, the existence of \v Svarc--Milnor function of a specific form, such as in \thref{Thm: Most General Fried Suspension}, is a possibly non-trivial assumption.
\end{remark}

\begin{lemma}\thlabel{Lem: Alternative Convergence Torsion}
Suppose that $P^E$ is $g$-trace class and that
\begin{enumerate}[label=\emph{(\Roman*)}]
    \item there exists a \v Svarc--Milnor function for the action with respect to $g$; and
    \item there is an $\alpha_g > 0$ such that $\mathcal T_g(t) = \mathcal O(t^{-\alpha_g})$ as $t \to \infty$,
\end{enumerate}
then the equivariant analytic torsion at $g$ is well-defined.
\end{lemma}

\begin{proof}
As there exists a \v Scarc--Milnor function for the action with respect to $g$, Proposition 3.15 in \cite{hochs2022equivariant} states that $e^{-t\Delta_{E}} - P^E$ is $g$-trace class for all $t > 0$, the integral
\[
    \frac{1}{\Gamma(s)}\int_0^1 t^{s-1}\mathcal T_g(t)\;dt
\]
converges for $s \in \C$ with $\re(s) > 0$ large enough, and the expression has a meromorphic continuation to $\C$ which is holomorphic at $s = 0$. The result now follows from Lemma 5.10 in \cite{hochs2022equivariant}.
\end{proof}

While the existence of a \v Svarc--Milnor function is a relatively mild assumption, it is known that condition (II) in \thref{Lem: Alternative Convergence Torsion} does not hold in general. A stronger condition is that the \textbf{delocalised Novikov--Shubin number}
\[
    \alpha_g^p \coloneqq \sup\{\alpha > 0 : \Tr_g(e^{-t\Delta_E^p} - P_p^E) = \mathcal O(t^{-\alpha}) \text{ as $t \to \infty$}\}
\]
is positive for all $p$. This is a non-trivial condition, as \cite{Grabowski2015} shows that the Novikov--Shubin numbers can be zero. However, these numbers are known in some cases. For example, Section 1 of \cite{MATHAI1992369} and \cite{shen2025heatkernellargetimebehavior} provide some results on positivity of these invariants. Moreover, \cite{hochs2022equivariant} Section 3 provides more results on delocalised Novikov--Shubin numbers.

\subsection{Ruelle Dynamical \texorpdfstring{$\zeta$}{zeta}-Function}\label{Subs: Ruelle Prelim}

In this subsection we recall the definition and basic properties of the equivariant Ruelle dynamical $\zeta$-function as stated in \cite{hochs2023ruelle}. 

Let $G$ and $M$ be as in Subsection \ref{Subsec: Torsion Prelims}, however we do not require $TM$ to have a metric and hence do not require $G$ to act by isometries. Let $\vphi$ be the flow of a smooth vector field $u$ on $M$. Suppose that $u$ has no zeroes, and that its flow $\vphi$ is complete. Suppose that for all $t \in \R$, the flow map $\vphi_t \colon M \to M$ is $G$-equivariant. 

Further, suppose $E$ and $\nabla^E$ are as in subsection \ref{Subsec: Torsion Prelims} except we do not require a metric on $E$. Recall that $\psi \in C^\infty_c(M)$ is a cutoff function for the action of $G$ on $M$, and $g \in G$ is fixed.

\begin{definition}[\cite{hochs2023ruelle}]
\mbox{}
\begin{enumerate}[label=(\Roman*)]
    \item The \textbf{$g$-delocalised length spectrum} of $\vphi$ is
    \begin{equation}\label{Eq: Delocal Length}
        L_g(\vphi) \coloneqq \{l \in \R \setminus \{0\} : \exists m \in M \text{ such that } \vphi_l(m) = gm\}.
    \end{equation}
\end{enumerate}
Let $l \in L_g(\vphi)$.
\begin{enumerate}[resume,label=(\Roman*)]
    \item A flow curve $\gamma$ of $\vphi$ is said to be \textbf{$(g,l)$-periodic} if $\gamma(l) = g\gamma(0)$. There is an equivalence relation on such flow curves by saying that two flow curves are equivalent if they are related by a time shift. We set $\Gamma^g_l(\vphi)$ to be the set of equivalence classes of $(g,l)$-periodic flow curves. We often choose representatives of classes in $\Gamma^g_l(\vphi)$ and implicitly claim that what follows is independent of the choice of representative.
\end{enumerate}
Let $\gamma \in \Gamma^g_l(\vphi)$.
\begin{enumerate}[resume, label=(\Roman*)]
    \item The \textbf{linearised $g$-delocalised Poincar\'e map} of $\gamma$ is the map 
    \begin{equation}
        P^g_\gamma \colon T_{\gamma(0)}M/\R u\big(\gamma(0)\big) \to T_{\gamma(0)}M/\R u\big(\gamma(0)\big)
    \end{equation}
    induced by $D_{\gamma(0)}\big(\vphi_l \circ g^{-1})$.
    
    \item The map
    \begin{equation}\label{Eq: g-Parallel Transport}
        \rho_g(\gamma) \colon E_{\gamma(0)} \to E_{\gamma(0)}
    \end{equation}
    is the parallel transport map from $E_{\gamma(0)}$ to $E_{\gamma(-l)}$ along $\gamma$ with respect to $\nabla^E$, composed with $g \colon E_{\gamma(-l)} \to E_{\gamma(0)}$.

    \item The \textbf{$\psi$-primitive period} of $\gamma$ is
    \begin{equation}\label{Eq: Cutoff Primitive Period}
        T_\gamma \coloneqq \int_{I_\gamma}\psi\big(\gamma(s)\big) \; ds,
    \end{equation}
    where $I_\gamma$ is an interval such that $\gamma|_{I_\gamma}$ is a bijection onto the image of $\gamma$, modulo sets of measure zero. 
\end{enumerate}
\end{definition}

\begin{remark}
It is possible that the $g$-delocalised length spectrum \eqref{Eq: Delocal Length} is empty for some $g \in G$, but not for others. For instance, let $M = G = \R$ with $G$ acting on $M$ by addition, with the flow $\vphi_t(m) = m+t$. Then for non-trivial $g \in G$, it follows that $L_g(\vphi) = \{g\}$ whereas $L_0(\vphi) = \emptyset$.
\end{remark}

\begin{remark}
As the domain of a flow curve is taken to be the real line, a single flow curve may be in more than one of the sets $\Gamma^g_l(\vphi)$, if $\gamma(l) = g\gamma(0)$ for more than one combination of $l \in \R \setminus \{0\}$ and $g \in G$. For a flow curve $\gamma$, the maps $P^g_\gamma$ and $\rho_g(\gamma)$ depend on $l$, though this is not stated in the notation. 
\end{remark}

\begin{definition}
The flow $\vphi$ is $g$-nondegenerate if the map $P^g_\gamma - 1_{T_{\gamma(0)}M/\R u(\gamma(0))}$ is invertible for all $l \in L_g(\vphi)$ and $\gamma \in \Gamma^g_l(\vphi)$.
\end{definition}

The following result from \cite{hochs2025equivariantguillemintraceformula} generalises a well-know criterion for fixed points of group actions, or flow curves of a given flow, to be isolated. 

\begin{lemma}[Lemma 2.25 \cite{hochs2025equivariantguillemintraceformula}]\thlabel{Lem: NonDeg Flows Images}
Suppose that $\vphi$ is $g$-nondegenerate and let $l \in L_g(\vphi)$. Then the set $\Gamma^g_l(\vphi)$ is discrete, in the sense that it is countable, and the images of the curves in $\Gamma^g_l(\vphi)$ together form the closed subset $M^{g^{-1} \circ \vphi} \subseteq M$ of points fixed by $g^{-1} \circ \vphi$. 
\end{lemma}

\begin{lemma}[Lemma 2.26 \cite{hochs2025equivariantguillemintraceformula}]\thlabel{Lem: Useful Results Flow Curves}
Suppose that $\vphi$ is $g$-nondegenerate. Then for $l \in L_g(\vphi)$,
\begin{enumerate}[label=\emph{(\Roman*)}]
    \item $T_\gamma$ converges for all $\gamma \in \Gamma^g_l(\vphi)$; 
    \item every $\gamma \in \Gamma^g_l(\vphi)$ is periodic if, and only if, $g$ lies in a compact subgroup of $G$.
\end{enumerate}
\end{lemma}

\begin{proof}
Only the implication $\gamma \in \Gamma^g_l(\vphi)$ being periodic implies $g$ is in a compact subgroup of $G$ is not proven in \cite{hochs2025equivariantguillemintraceformula}. So suppose that $\gamma \in \Gamma^g_l(\vphi)$ is periodic. Then $\im \gamma$ is a compact subset of $M$. Moreover, from the equivariance of $\vphi$, for all $t \in \R$ we have $g\gamma(t) = \gamma(t + l)$ and so $g \im\gamma \subseteq \im \gamma$. Furthermore, this implies that $g^n \im\gamma \subseteq \im\gamma$ for all $n \in \Z$. Thus, we see that for all $n \in \Z$ it follows that $g^n \in A \coloneqq \{x \in G \mid x\im \gamma \cap \im\gamma \neq \emptyset\}$ which is a compact subset of $G$ by properness of the action. Hence the cyclic subgroup generated by $g$ is contained in $A$. As $A$ is closed, the closure of the subgroup generated by $g$, which is a subgroup containing $g$, is contained in $A$ and therefore is compact. 
\end{proof}

\begin{example}\thlabel{Ex: Compact Primitive Period}
If $G$ is compact (for example, trivial), and the flow $\vphi$ is $g$-nondegenerate, then for all $\gamma \in \Gamma^g_l(\vphi)$ we have $T_\gamma$ is the primitive period 
\[
    T_\gamma^\# = \min\{t > 0 : \gamma(t) = \gamma(0)\}
\]
of $\gamma$. See Example 2.21 of \cite{hochs2025equivariantguillemintraceformula} for more details.
\end{example}

It follows from the definitions that for all $h \in G$,
\begin{equation}\label{Eq: Conjugate Length}
    L_{hgh^{-1}}(\vphi) = L_g(\vphi)
\end{equation}
and for $l$ in this set,
\begin{equation}\label{Eq: Conjuagte Flow Curves}
    \Gamma^{hgh^{-1}}_l(\vphi) = h \cdot \Gamma^g_l(\vphi).
\end{equation}
So by \thref{Lem: NonDeg Flows Images}, the set $\Gamma^{hgh^{-1}}_l(\vphi)$ is discrete for all $l \in L_g(\vphi)$ if the flow is $g$-nondegenerate.

As in the situation for equivariant torsion, assume the centraliser $Z$ of $g$ in $G$ is unimodular.

\begin{definition}
Suppose that $L_g(\vphi)$ is countable, and that $\vphi$ is $g$-nondegenerate. If the following expression \eqref{Eq: Equivariant Ruelle Conjugate} converges for $\sigma \in \C$ with large real part, then for such $\sigma$, the \textbf{equivariant Ruelle dynamical $\zeta$-function} at $g$ is defined by
\begin{equation}\label{Eq: Equivariant Ruelle Conjugate}
\begin{aligned}
    &\log R^g_{\vphi, \nabla^F}(\sigma) = \\
    & \lim_{r \to \infty}\int_{G/Z}\sum_{l \in L_g(\vphi) \cap [-r,r]}\frac{e^{-\abs{l}\sigma}}{\abs{l}}\sum_{\gamma \in \Gamma^{hgh^{-1}}_l(\vphi)}\sgn\big(\det(1 - P^{hgh^{-1}}_\gamma)\big)\tr\big(\rho_{hgh^{-1}}(\vphi)\big)T_\gamma\; d(hZ).
\end{aligned}
\end{equation}
\end{definition}

Note that from \eqref{Eq: Conjugate Length}, \eqref{Eq: Conjuagte Flow Curves}, and the properties of $P^{hgh^{-1}}_\gamma$ and $\rho_{hgh^{-1}}$, it follows that $R^g_{\vphi, \nabla^E}$ only depends on the conjugacy class of $g$. See Remarks 2.8 and 2.9 in \cite{hochs2023ruelle} for more details.

\begin{remark}
The numbers $\det(1 - P^{hgh^{-1}}_\gamma)$ and $\tr\big(\rho_{hgh^{-1}}(\gamma)\big)$ in \eqref{Eq: Equivariant Ruelle Conjugate} are independent of the choice of representative $\gamma$ of a class in $\Gamma^{hgh^{-1}}_l(\vphi)$. Furthermore, it follows from Proposition 3.6 in \cite{hochs2023ruelle} that \eqref{Eq: Equivariant Ruelle Conjugate} also does not depend on the choices made in the definition of $T_\gamma$, and in particular, is independent of the choice of cutoff function $\psi$ for the action. 
\end{remark}

\begin{remark}\thlabel{Rmk: Relation Classical and Equiv Ruelle}
In the case that $G$ is a compact group, $M$ is odd-dimensional, and $E$ is a Hermitian bundle with $\nabla^E$ a metric-preserving connection, then we have a relationship to the classical Ruelle dynamical $\zeta$-function. More precisely, if we suppose that all the operators \eqref{Eq: g-Parallel Transport} are unitary, and $\sigma \in \R$, then $R_{\vphi, \nabla^E}^e(\sigma)$ is the absolute value squared of the classical Ruelle dynamical $\zeta$-function $R_{\vphi, \rho}(\sigma)$ defined in \cite{Analytic.Torsion.Fried.Survey}. See Lemma 3.17 in \cite{hochs2023ruelle} for a proof. 
\end{remark}

The integrand in \eqref{Eq: Equivariant Ruelle Conjugate} is invariant under right multiplication of $h$ by an element in $Z$, so that it is a well-defined function on the quotient $G/Z$. However, we can reformulate \eqref{Eq: Equivariant Ruelle Conjugate} in such a way that makes it easier to evaluate, with the consequence of making this $Z$-invariance much less obvious.

\begin{lemma}[Lemma 3.7 \cite{hochs2023ruelle}]\thlabel{Lem: Ruelle Zeta No Conjugate}
The following expression \eqref{Eq: Equivariant Ruelle} converges if, and only if, \eqref{Eq: Equivariant Ruelle Conjugate} converges, and
\begin{equation}\label{Eq: Equivariant Ruelle}
\begin{aligned}
    &\log R^g_{\vphi, \nabla^F}(\sigma) = \\
    &\lim_{r \to \infty}\int_{G/Z}\sum_{l \in L_g(\vphi) \cap [-r,r]}\frac{e^{-\abs{l}\sigma}}{\abs{l}}\sum_{\gamma \in \Gamma^{g}_l(\vphi)}\sgn\big(\det(1 - P^{g}_\gamma)\big)\tr\big(\rho_{g}(\vphi)\big)\int_{I_\gamma}\psi\big(h\gamma(s)\big)\;ds\; d(hZ)
\end{aligned}
\end{equation}
\end{lemma}

\begin{corollary}
Suppose \eqref{Eq: Equivariant Ruelle} converges absolutely for $\sigma \in \C$ with $\re(\sigma) > 0$ large enough. Then
\begin{equation}\label{Eq: Absolute Ruelle Zeta}
\begin{aligned}
    &R^g_{\vphi,\nabla^E}(\sigma) \\
    &= \exp\Big(\sum_{l \in L_g(\vphi)}\frac{e^{-\abs{l}\sigma}}{\abs{l}}\sum_{\gamma \in \Gamma^g_l(\vphi)}\sgn\big(\det(1 - P^g_l)\big)\tr\big(\rho_g(\gamma)\big)\int_{G/Z}\int_{I_\gamma}\psi\big(h\gamma(s)\big)\;ds\;d(hZ)\Big),
\end{aligned}
\end{equation}
where we have implicitly chosen a Borel section of $G \to G/Z$.
\end{corollary}

\begin{remark}\thlabel{Rmk: Z Compact Replace Integral}
If $Z$ is compact, we can replace the integral over $G/Z$ in \eqref{Eq: Equivariant Ruelle} with an integral over $G$. Hence if $R^g_{\vphi,\nabla^E}$ also converges absolutely when $\re(\sigma) > 0$ is large enough, then as $\psi$ is a cutoff function for the action of $G$ on $M$, \eqref{Eq: Absolute Ruelle Zeta} becomes 
\begin{equation}
    R^g_{\vphi,\nabla^E}(\sigma) = \exp\Big(\sum_{l \in L_g(\vphi)}\frac{e^{-\abs{l}\sigma}}{\abs{l}}\sum_{\gamma \in \Gamma^g_l(\vphi)}\sgn\big(\det(1 - P^g_l)\big)\tr\big(\rho_g(\gamma)\big)T_\gamma^\#\Big).
\end{equation}
\end{remark}

\subsection{The Fried Conjecture}\label{Subsec: Fried Conjecture}

Recall that the Fried conjecture provides a link between analytic torsion and the Ruelle dynamical $\zeta$-function. In \cite{hochs2023ruelle} Hochs and Saratchandran propose an equivariant version of this conjecture, using the equivariant versions of analytic torsion and Ruelle dynamical $\zeta$-function defined earlier. This led to the following question.

\begin{question}[Equivariant Fried Conjecture]\thlabel{Q: Equivariant Fried}
Suppose that $M$ is odd-dimensional, and $\ker(\Delta_E) = \{0\}$. Under what further conditions does
\begin{enumerate}[label=(\Roman*)]
    \item the expression \eqref{Eq: Equivariant Ruelle Conjugate} for $R^g_{\vphi,\nabla^E}(\sigma)$ converge when $\sigma$ has large real part;
    \item $R^g_{\vphi,\nabla^E}$ have a meromorphic extension to some open neighbourhood of 0 in the complex plane that is regular at 0; and
    \item the equality
    \begin{equation}
        R^g_{\vphi,\nabla^E}(0) = T_g(\nabla^E)^2
    \end{equation}
    hold?
\end{enumerate}
\end{question}

\begin{remark}
In the case that $G$ is a compact group and $g = e$ is the identity element then, as mentioned in \thref{Rmk: Relation Classical and Equiv Ruelle}, $R^e_{\vphi, \nabla^E}(0) = \abs{R_{\vphi, \rho}(0)}^2$ when defined. Thus \thref{Q: Equivariant Fried} reduces to the classical Fried conjecture in this case. Therefore, the equivariant Fried conjecture is true in the cases where the classical Fried conjecture is true. 
\end{remark}

In \cite{hochs2023ruelle} the equivariant Fried conjecture was shown to hold in the cases of the real line and the circle acting on themselves, respectively. They also show that it is true for the case of geodesic flow on $\R^n$ with the action of a discrete subgroup of the oriented Euclidean motion group $\R^n \rtimes \SO(n)$. 

However, it is known that the equivariant Fried conjecture is false in general. Consider the action of $G = \pi_1(X)$ on the universal cover $\widetilde X$ of a compact manifold $X$, and take $g = e$ the identity element. Then for a suitable choice of manifold $X$ and flow $\vphi$, the $L^2$-analytic torsion $T_e(\nabla^E)$ may be non-trivial while $R^e_{\vphi, \nabla^E}$ is trivial. See \cite{Fried1986, 10.4310/jdg/1214448084} for more details. 

More generally, it is unknown whether the equivariant Fried conjecture holds in the cases where the classical Fried conjecture holds, $G$ is a compact group, and $g \neq e$ is not the identity element. Our results in the next section present some progress in this situation.

\subsection{Results}\label{Subsec: Results}

In this subsection we present the main results of this paper, that the equivariant Fried conjecture for suspension flow of an isometry is true under some additional assumptions. As far as the authors are aware, this is the first general equivariant Fried conjecture result, i.e.\ not for a specific group action on a specific manifold.  

Let $Y$ be an oriented Riemannian manifold, and let $T$ be an isometry of $Y$. Then $T$ induces a $\Z$-action on the product manifold $Y \times \R$ given by
\begin{equation}
    n \cdot (y, t) = (T^ny, t - n)
\end{equation}
for $n \in \Z$, and $(y,t) \in Y \times \R$. Let $M$ be the orbit space of this action, i.e.\ $M = (Y \times \R)/\Z$. Then $M$ is called the \textbf{suspension} of $Y$ with respect to $T$. It follows that $M$ is a fibration over the circle $S^1$ with fibres diffeomorphic to $Y$, as can be seen by taking the map $M \to S^1$ induced by the projection map $Y \times \R \to \R$.

\begin{remark}
In the literature $T$ is often only assumed to be a diffeomorphism. Here we require $T$ to be an isometry in order for the equivariant analytic torsion associated to suspension to be well-defined. Moreover, the suspension of $Y$ also appears in the literature as the mapping torus of $T$.
\end{remark}

We will say that $M$ is the suspension of $Y$ when it is clear what isometry we are speaking of, and we set $\pi \colon Y \times \R \to M$ to be the canonical projection map. 

Let $G$ be a locally compact unimodular group acting properly and cocompactly on $Y$, preserving the orientation. Further suppose that the action by $G$ on $Y$ commutes with $T$. As we will see from \thref{Lem: Twisted Product Action Proper}, this induces a proper action of $G \times \Z$ on $Y \times \R$ given by
\begin{equation}\label{Eq: Equivariant Suspension Action}
    (x,n) \cdot (y,t) = (xT^ny, t-n)
\end{equation}
for $(x,n) \in G \times \Z$, and $(y,t) \in Y \times \R$. This in turn induces a proper, cocompact action of $G$ on $M$, see \thref{Lem: Proper Action on Quotient}. 

Suspension flow $\vphi$ on $M$ is defined by
\[
    \vphi_s[y,t] = [y, t+ s]
\]
for all $s \in \R$, and $[y,t] \in M$. We see that for all time $s \in \R$, the flow map $\vphi_s \colon M \to M$ is $G$-equivariant. 

Let $E_1 \to Y$ be a flat, Hermitian, $G$-equivariant vector bundle with $\nabla^{E_1}$ a Hermitian, $G$-invariant flat connection. Further take $E_2 \to \R$ be the trivial line bundle with the flat connection $\nabla^{E_2} = d$. Let $\Z$ act on the fibres of $E_2$ by subtraction, so that $E_2$ is a $\Z$-equivariant vector bundle. 

Assume that $T$, the isometry defining the suspension, lifts to an equivariant bundle isometry of $E_1$, i.e.\ so there exists an equivariant fibrewise linear isometry $T_{E_1} \colon E_1 \to E_1$ such that following diagram commutes:
\[
\begin{tikzcd}
E_1 \arrow[r, "T_{E_1}"] \arrow[d] & E_1 \arrow[d] \\
Y \arrow[r, "T"']                  & Y.           
\end{tikzcd}
\]
The action of $T_{E_1}$ defines a $\Z$-action on $E_1$ so that $E_1$ is also $\Z$-equivariant vector bundle over $Y$. Suppose that $\nabla^{E_1}$ is $\Z$-invariant with respect to this action. Then $\widetilde E = E_1 \boxtimes E_2 = E_1 \times \R$ is a flat $G \times \Z$-equivariant vector bundle over $Y \times \R$, and $E = (E_1 \boxtimes \R)/\Z$ is the suspension of $E_1$, which is a flat vector bundle over $M$. Then we have a flat $G$-invariant connection $\nabla^E$ on $E$ which is obtained by restricting the product connection on $\widetilde E$ to the $\Z$-invariant sections. We compute the equivariant analytic torsion $T_g(\nabla^{E})$ for the suspension of $Y$ in \thref{Prop: Torsion Suspension}. 

Let $\psi_{G} \in C^\infty(G)$ be a cutoff function for the action of $Z$ on $G$ by right multiplication, i.e. 
\begin{equation}\label{Eq: Cutoff Z for G}
    \int_{Z}\psi_{G}(xz^{-1})dz = 1
\end{equation}
for all $x \in G$. Define $\psi^{g} \in C^\infty(Y)$ by
\begin{equation}\label{Eq: Def cutoff^g}
    \psi^{g}(y) \coloneqq \int_{G} \psi_{G}(x)\psi_Y(xgy)\;dx,
\end{equation}
where $\psi_Y$ is cutoff function for the action of $G$ on $Y$. Note that $\psi^g$ is a cutoff function for the action of $Z$ on $Y^{g}$. If $e^{-t\Delta_{E_1}}$ denotes the heat kernel of the Hodge Laplacian associated to $\nabla^{E_1}$, and $F_1$ is the number operator on $\extp \bullet T^*Y \otimes E_1$, we set 
\[
    \chi_{(g,0)}(\nabla^{E_1}) = \Tr_g\big((-1)^{F_1}e^{-t\Delta_{E_1}}\big).
\]
The fact that $\chi_{(g,0)}(\nabla^{E_1})$ is independent of $t$ follows from Proposition 3.6 in \cite{FPF.and.Character.Formula}, see also \thref{Lem: Euler Independent of t}. The number $\chi_{(g,0)}(\nabla^{E_1})$ represents an equivariant version of the Euler characteristic and is studied more generally in Section 5 of \cite{hochs2022equivariant} and Subsection \ref{Subsec: Euler Char} of this paper. Finally if $T^ly = y$ for some $l \in \Z \setminus \{0\}$, then let $p(y) \coloneqq \min\{k \in \N : T^ky = y\}$, which is primitive period of the orbit through $y$ generated by $T$.

For the rest of the section, we assume the following:
\begin{enumerate}[label=(\Roman*)]
    \item $P^{E_1}$ is $g$-trace class,
    \item $\ker(\Delta_E) = 0$; and
    \item for all $n \in \Z \setminus \{0\}$ the fixed point set $Y^{g^{-1}T^n}$ is either discrete or empty.
\end{enumerate}

For example if $G$ is compact group, then $M$ and $Y$ are compact, and so (I) holds. More generally, we see that (I) holds if either
\begin{itemize}
    \item $G/Z$ is compact, see Subsection 2.2 in \cite{hochs2022equivariant}; or
    \item 0 is isolated in the spectrum of the operator $\nabla^E + (\nabla^E)^*$, see Lemma 7.7 in \cite{piazza2023heatkernelsperturbedoperators}.
\end{itemize}

The following are results on the equivariant Fried conjecture for suspension flow. We first formulate a general result involving conditions on compatibility of various cutoff functions. We then deduce more concrete results in special cases.

\begin{theorem}\thlabel{Thm: Most General Fried Suspension}
Suppose $g$ lies in a compact subgroup of $G$. Suppose that one of the following sets of conditions holds.
\begin{enumerate}[label=\emph{(\Roman*)}]
    \item The space $G/Z$ is compact, and the Novikov--Shubin numbers $(\alpha_{e}^{p})_{E_1}$ for $\Delta_{E_1}$ are positive.
    \item We have that $\Tr_{g}(T_{E_1}^nP^{E_1}) = 0$ for all $n \in \Z$. There exists a \v Svarc--Milnor function $l_G$ for the action of $G$ on $Y$ with respect to $g$, and a \v Svarc--Milnor function for the action of $G \times \Z$ on $Y$ with respect to $(g,n)$ is given by the sum of $l_G$ and the absolute value function on $\Z$ for all $n \in \Z$. Finally suppose that
    \[
        \vol\big\{hZ \in G/Z : l_G(hgh^{-1}) \le r\big\}
    \]
    has at most exponential growth in $r$.
\end{enumerate}

If the Ruelle dynamical $\zeta$-function for suspension flow converges absolutely for $\sigma \in \C$ with $\re(\sigma) >0$ large enough; and there exists a Borel section of $G \to G/Z$ such that
\begin{equation}\label{Eq: Most General Fried Condition}
    \psi^g(y) = \frac{1}{p(y)}\int_{G/Z}\int_0^{p(y)}\psi_M[hy,s]\; ds\; d(hZ)
\end{equation}
holds for all $y \in Y^{g^{-1}T^n}$ and $n \in L_g(\vphi)$ with respect to some choice of cutoff functions $\psi_G,$ $\psi_Y$, and $\psi_M$, then for such $\sigma$, we have
\[
    R^g_{\vphi,\nabla^E}(\sigma) = \exp\big(\sigma \chi_{(g,0)}(\nabla^{E_1})\big)T_g(\nabla^E,\sigma^2)^2.
\]
So if $T_g(\nabla^E)$ is well-defined, then $R^g_{\vphi,\nabla^E}$ has a meromorphic extension which is holomorphic at $\sigma = 0$. In particular, in that case, $R^g_{\vphi,\nabla^E}(0) = T_g(\nabla^E)^2$ and the equivariant Fried conjecture holds. 
\end{theorem}

The proof of \thref{Thm: Most General Fried Suspension} can be found in Subsection \ref{Subsec: Equivariant Fried Conjecture}

In the following corollaries, the assumption that $g$ lies in a compact subgroup of $G$ holds immediately.

In the case that $G$ is a compact group, the equivariant analytic torsion for the suspension is well-defined, and so we obtain the following corollary which will be proven in Subsection \ref{Subsec: Compact Group Fried}.

\begin{corollary}\thlabel{Cor: Equivariant Fried Compact}
Suppose $G$ is a compact group. Then equivariant analytic torsion for suspension is well-defined, and the equivariant Fried conjecture holds. 
\end{corollary}

As stated in Subsection \ref{Subsec: Fried Conjecture}, the equivariant Fried conjecture is false in general when $g=e$ is the identity element of a non-compact group. However, the following result, \thref{Cor: Equivariant Fried Identity}, is a positive result in the case that $g = e$.

For \thref{Cor: Equivariant Fried Identity} and \thref{Ex: Fried Identity Example}, we assume that $G$ is a discrete group. This is not a strict assumption, as $G = Z_G(e)$ acts properly on $Y^{T^n}$ which is discrete. Hence there exists a compact subgroup $H \subseteq G$ such that $G/H$ is discrete. In particular if the action of $G$ on $Y^{T^n}$ is free, then we can choose $H = \{e\}$, and $G$ is discrete.

\begin{corollary}\thlabel{Cor: Equivariant Fried Identity}
Let $e$ be the identity element of a discrete group $G$ and suppose that the Novikov--Shubin numbers $(\alpha_e^p)_{E_1}$ for $\Delta_{E_1}$ are positive for all $p$. Suppose that there exist $C,c \ge 0$ such that for all $n \in \Z \setminus \{0\}$, we have 
\begin{equation}\label{Eq: Cor Fried Identity 1}
    \abs{\{y \in Y^{T^n}: \emph{there exists $s \in \R$ with $\sum_{k \in \Z}\psi_Y(T^ky)\psi_\R(s-k) \neq 0$} \}} \le Ce^{c\abs{n}},
\end{equation}
for some choice of cutoff functions $\psi_Y$ and $\psi_\R$, with $\psi_\R$ a cutoff function for the action of $\Z$ on $\R$. Also suppose that the total number of distinct primitive periods $p(y)$ for $y \in \bigcup_{n \in \Z \setminus \{0\}}Y^{T^n}$ is finite. 

Then the Ruelle dynamical zeta function converges absolutely for $\sigma \in \C$ with $\re(\sigma) > c$, and for $\sigma$ with $\re(\sigma) > c$ large enough
\[
    R^e_{\vphi,\nabla^E}(\sigma) = \exp\big(\sigma \chi_{(e,0)}(\nabla^{E_1})\big)T_e(\nabla^E,\sigma^2)^2.
\]
So if the equivariant analytic torsion $T_g(\nabla^E)$ is well-defined, then the equivariant Fried conjecture for suspension flow is true.   
\end{corollary}

\begin{example}\thlabel{Ex: Fried Identity Example}
Suppose $\Gamma$ is a countable discrete group, and let $X$ be a compact manifold. Let $T$ be an isometry of $X$ and lift this to an isometry on $Y \coloneqq \Gamma \times X$ by $T(\gamma, x) = (\gamma, Tx)$. Then $T$ is $\Gamma$-equivariant for the action of $\Gamma$ on $Y$ by left translation in the first component. Suppose further that the fixed point $X^T$ is finite and $X^{T^n} = X^T$ for all $n \in \Z \setminus \{0\}$. Then the equivariant analytic torsion for the suspension of $Y$ is well-defined, and hence the equivariant Fried conjecture holds by \thref{Cor: Equivariant Fried Identity}. The proof of these statements can be found in Subsection \ref{Subsec: Identity Discrete Group}. 

For a more concrete example, take $X = S^n$, and let $T \colon S^n \to S^n$ be a rotation about fixed axis by an irrational angle.
\begin{center}
\begin{circuitikz}[scale=0.9]
\tikzstyle{every node}=[font=\large]
\draw  (3.75,13) circle (1.5cm);
\draw  (7.5,13) circle (1.5cm);
\draw  (11.25,13) circle (1.5cm);
\draw [->, >=Stealth] (3.75,10.75) -- (10.75,10.75)node[pos=0.5,below, fill=white]{$\Gamma$};
\draw [ dashed] (3.75,13) ellipse (1.5cm and 0.25cm);
\draw [ dashed] (7.5,13) ellipse (1.5cm and 0.25cm);
\draw [ dashed] (11.25,13) ellipse (1.5cm and 0.25cm);
\draw [->, >=Stealth] (6.75,15) .. controls (8.25,14.5) and (8.25,16.25) .. (6.75,15.5) node[pos=0.5,right, fill=white]{$T$};
\draw [fill] (.5,13) circle [radius=2pt];
\draw [fill] (1,13) circle [radius=2pt];
\draw [fill] (1.5,13) circle [radius=2pt];
\draw [fill] (13.5,13) circle [radius=2pt];
\draw [fill] (14,13) circle [radius=2pt];
\draw [fill] (14.5,13) circle [radius=2pt];
\end{circuitikz}
\label{fig:my_label}
\end{center}
\end{example}

While \thref{Ex: Fried Identity Example} is not entirely natural as $Y$ is disconnected, it is still of note being a positive result for the $g=e$ case.

The proofs of \thref{Cor: Equivariant Fried Identity} and \thref{Ex: Fried Identity Example} can be found in Subsection \ref{Subsec: Identity Discrete Group}. 

Turning now to the case where $G/Z$ is not compact but $Z$ is compact, along with the additional assumption that the conjugacy class of $g$ in $G$ is closed, we have the following corollary.

\begin{corollary}\thlabel{Cor: Equivariant Fried Compact Z}
Suppose that the conjugacy class of $g$ is closed, that $Z$ is compact, and that assumption (II) of \thref{Thm: Most General Fried Suspension} holds. Suppose further that there is a compact set $A \subseteq Y$ such that $Y^{g^{-1}T^n} \subseteq A$ for all $n \in \Z \setminus \{0\}$. If there exist $C,c \ge 0$ such that $\abs{Y^{g^{-1}T^n}} \le Ce^{c\abs{n}}$ for all $n \in \Z \setminus \{0\}$, then $R^{g}_{\vphi,\nabla^E}$ converges absolutely on the set of $\sigma$ with $\re(\sigma) > c$, and for $\sigma$ with $\re(\sigma) > c$ large enough,
\[
    R^g_{\vphi,\nabla^E}(\sigma) = \exp\big(\sigma \chi_{(g,0)}(\nabla^{E_1})\big)T_g(\nabla^E,\sigma^2)^2.
\]
Moreover, if the equivariant analytic torsion $T_g(\nabla^E)$ is well-defined then the equivariant Fried conjecture for suspension flow is true. 
\end{corollary}

\begin{example}\thlabel{Ex: Fried Compact Z}
Let $G$ be a real semisimple Lie group, $K < G$ a maximal subgroup and $S < K$ a maximal torus. Consider the fibred product $Y = G \times_S (K/S)$. Let $g,T \in S$ and $\widetilde T \colon Y \to Y$ be defined by $\widetilde T[x,n] = [x,Tn]$. Suppose that for all $l \in \Z$, $g$, $T$, $g^{-1}T^l$, and $g^{-1}x^{-1}T^lx$ for $x$ a representative of an element of the Weyl group $N_K(S)/S$, generate dense subgroups of $S$. Then $R^g_{\vphi,\nabla^E}$ converges absolutely on the set of $\sigma$ with $\re(\sigma) > 0$. Therefore if the equivariant analytic torsion of the suspension is well-defined, then by \thref{Cor: Equivariant Fried Compact Z} the equivariant Fried conjecture holds. 
\end{example}

The proofs of \thref{Cor: Equivariant Fried Compact Z} and \thref{Ex: Fried Compact Z} can be found in Subsection \ref{Subsec: Compact Z Fried}.

\begin{remark}
Both \thref{Cor: Equivariant Fried Identity} and \thref{Cor: Equivariant Fried Compact Z} state that if the equivariant analytic torsion is well-defined, then the equivariant Fried conjecture for suspension flow holds. However, using \thref{Lem: Alternative Convergence Torsion} we see that the equivariant analytic torsion is well-defined under the additional assumption that the Novikov--Shubin numbers for $\Delta_{E}$ are positive.
\end{remark}

\begin{remark}
Suppose that the assumptions of \thref{Thm: Most General Fried Suspension} hold, that the equivariant analytic torsion $T_g(\nabla^E)$ is well-defined, and that $R^g_{\vphi, \nabla^E}$ converges absolutely for $\re(\sigma) > 0$. (In \thref{Cor: Equivariant Fried Identity} and \thref{Cor: Equivariant Fried Compact Z} this is the case if the sizes of the sets $K_{g^{-1}T^n}$ and $Y^{g^{-1}T^n}$ are now bounded in $n$.) Then by Abel's limit theorem, see Section 2.5 of \cite{ComplexAnalysisAhlfors}, we have 
\[
    \lim_{\sigma \to 0^+} R^g_{\vphi,\nabla^E}(\sigma) = T_g(\nabla^E)
\]
where the limit is taken in a region in which $|1 - e^{-\sigma}|/(1 - e^{-\re(\sigma)})$ is bounded. 
\end{remark}

\section{Proper Actions and Euler Characteristics}\label{Sub: Proper & Euler}

The proof of the classical Fried conjecture for suspension flow combines a fibration formula for analytic torsion with a fixed point formula. The proof in the equivariant case will follow a similar structure. While in Section \ref{Sec: Fibration Formula} we state and prove a fibration formula for equivariant analytic torsion, in this section we provide the background needed to state the fibration formula, while also proving a relevant fixed point theorem. 

More explicitly, in Subsection \ref{Subsec: Proper Actions} we prove that the two actions relevant for the definition of equivariant suspension are proper. Hence we can define equivariant analytic torsion for the suspension.  

In the following subsection, \ref{Subsec: Euler Char}, we define and prove various properties of an equivariant version of the Euler characteristic for a certain ``twisted'' product action on a manifold. These Euler characteristics are then used in the fibration formula, \thref{Thm: New Torsion Formula}. 

Finally in Subsection \ref{Subsec: Fixed Point Theorem} we prove a generalisation of the Atiyah--Bott Lefschetz fixed-point theorem to non-compact manifolds which may be of independent interest. This allows us to rewrite the equivariant Euler characteristic of Subsection \ref{Subsec: Euler Char} as a sum over fixed point sets for the action.

\subsection{Proper Actions}\label{Subsec: Proper Actions}

Suppose $G_1, G_2$ are locally compact groups. Suppose that $N$ is a locally compact topological space, and that $G_1 \times G_2$ acts properly $N$. Consider the quotient space $N/G_2$. There is an induced action of $G_1$ on $N/G_2$, defined from the commuting actions of $G_1$ and $G_2$ on $N$. We claim that this induced action is proper. To prove this, we need the following definition.

\begin{definition}
Suppose $H$ is a topological group acting on a topological space $X$. For $K \subseteq X$ compact, set
\begin{equation}
    (H)_K \coloneqq \{h \in H : hK \cap K \neq \emptyset\}.
\end{equation}
\end{definition}

It is well known that the action of $H$ on $X$ is proper if, and only if, $(H)_K$ is compact for every compact subset $K$ of $X$. 

\begin{lemma}\thlabel{Lem: Proper Action on Quotient}
The induced action of $G_1$ on $N/G_2$ is proper.
\end{lemma}

\begin{proof}
Consider a $G_2$-invariant subset $X \subseteq N$ with $X/G_2$ compact, and let $g_1 \in (G_1)_{X/G_2}$. Then there exist $g_2 \in G_2$, and $x,x' \in X$ such that $g_1x = g_2 x'$. Let $K \subseteq N$ be a compact subset such that $X = G_2 K$. Write $x = h_2 k$ and $x' = h_2' k'$ for $h_2, h_2' \in G_2$ and $k,k' \in K$. Then $g_1h_2 k = g_2h_2' k'$, and we find that $g_1(h_2')^{-1}g_2^{-1}h_2 k = k'$, as the actions commute. Thus $\big(g_1, (h_2')^{-1}g_2^{-1}h_2\big) \in (G_1 \times G_2)_K$, which is compact because $G_1 \times G_2$ acts properly on $N$. Therefore $g_1$ lies in some compact subset of $G_1$, $L$, which is the projection of $(G_1 \times G_2)_K$ onto the first factor. Thus, $(G_1)_{X/G_2} \subseteq L$, and $(G_1)_{X/G_2}$ is closed by continuity of the action. Hence $(G_1)_{X/G_2}$ is compact, and the action is proper. 
\end{proof}

\begin{definition}
Suppose that $M_1$ and $M_2$ are locally compact topological spaces. Suppose that $G_j$ acts on $M_j$ properly for $j=1,2$, and that $G_2$ acts on $M_1$ (not necessarily properly) commuting with the action of $G_1$. Then the \textbf{twisted product action} of $G_1 \times G_2$ on $M_1 \times M_2$ is defined by
\begin{equation}\label{Eq: Twisted Product Action}
    (g_1,g_2) \cdot (m_1,m_2) = (g_1g_2m_1,g_2m_2).
\end{equation}
\end{definition}

Note that a twisted product action is the action that appears in the definition of equivariant suspension, see \eqref{Eq: Equivariant Suspension Action}.

\begin{lemma}\thlabel{Lem: Twisted Product Action Proper}
The twisted product action of $G_1 \times G_2$ on $M_1 \times M_2$ is proper.
\end{lemma}

\begin{proof}
Let $K \subseteq M_1 \times M_2$ be compact. Then $K \subseteq K_1 \times K_2$ where $K_j \subseteq M_j$ is compact. Take $(g_1,g_2) \in (G_1 \times G_2)_K $. Then there exists $k_j \in K_j$ such that $g_1g_2k_1 \in K_1$ and $g_2k_2 \in K_2$. Thus $g_2 \in (G_2)_{K_2}$, which is a compact set. On the other hand,
\[
    g_1k_1 \in g_2^{-1}K_1 \subseteq (G_2)_{K_2}^{-1}K_1 \eqqcolon \widetilde{K}_1.
\]
As the inversion operation is continuous, $(G_2)_{K_2}^{-1}$ is compact, and so $\widetilde K_1$ is compact being the image of $(G_2)_{K_2}^{-1}$ and $K_1$ under the action map which is continuous. Hence $g_1 \widetilde K_1 \cap \widetilde K_1 \neq \emptyset$ (it contains $k_1$), and $g_1 \in (G_1)_{\widetilde K_1}$ lies in a compact set.  

Therefore $(G_1 \times G_2)_K $ is a subset of the compact set $(G_1)_{\widetilde K_1} \times (G_2)_{K_2}$. It is closed as the action is continuous. Hence $(G_1 \times G_2)_K $ is a compact subset of $G_1 \times G_2$, and the action is proper.
\end{proof}

\subsection{Equivariant Euler Characteristic}\label{Subsec: Euler Char}

Let $G_1$ and $G_2$ be locally compact unimodular groups, and suppose that $M_1$ is a complete, oriented Riemannian manifold. Suppose that $G_1 \times G_2$ acts on $M_1$ isometrically, and preserving the orientation. Further suppose that the restricted action of $G_1$ on $M_1$ is proper and cocompact. Suppose $E_1$ is a flat $G_1 \times G_2$-equivariant bundle over $M_1$, with $\nabla^{E_1}$ a $G_1 \times G_2$-invariant flat connection. Let $\Delta_{E_1}$ denote the corresponding Hodge Laplacian.

Let $g_j \in G_j$ for $j=1,2$, and assume that the centraliser of $g_1$ in $G_1$ is unimodular. From now on, assume that a \v Svarc--Milnor function for the action of $G_1$ on $M_1$ with respect to $g_1$ exists.

\begin{lemma}\thlabel{Lem: Euler Characterstic is Trace Class}
For all $t > 0$ and $j \in \N_0$, the operator $g_2\big(\nabla^{E_1} + (\nabla^{E_1})^*\big)^je^{-t\Delta_{E_1}}$ is $g_1$-trace class.
\end{lemma}

\begin{proof}
Fix $t > 0$, and for ease of notation let $D = \nabla^{E_1} + (\nabla^{E_1})^*$. Then $\Delta_{E_1} = D^2$, and from Lemma 3.11 in \cite{hochs2022equivariant}, the result is true if we can prove there exist $C,a > 0$ such that
\[
    \norm{g_2D^je^{-tD^2}(m,m')} \le Ce^{-ad(m,m')^2}
\]
for all $m,m' \in M_1$. By Proposition 4.2(1) in \cite{carey2014index}, there exist such an $a$ and $C$ such that
\begin{equation}\label{Eq: Gaussian Off-Diag Decay}
    \norm{D^je^{-tD^2}(m,m')} \le Ce^{-ad(m,m')^2}.
\end{equation}
Thus
\begin{align*}
    \norm{g_2D^je^{-tD^2}(m,m')} &= \norm{g_2 \cdot D^je^{-tD^2}(g_2^{-1}m,m')}\\
    &\le C e^{-ad(g_2^{-1}m,m')^2/t},
\end{align*}
because $G_2$ acts on $M_1$ and $E_1$ by isometries. For any $s,u,r \ge 0$ such that $s \ge u - r$, we have $s^2 \ge u^2/2 - r^2$, and hence by the triangle inequality
\[
    \norm{g_2D^je^{-tD^2}(m,m')} \le C e^{ad(g_2^{-1}m,m)^2} e^{-ad(m,m')^2/2}.
\]
We claim that $e^{ad(g^{-1}_2m,m)^2}$ is bounded on $M_1$. As the action of $G_1$ on $M_1$ is cocompact, there is a compact subset $K \subseteq M_1$ such that $G_1 K = M_1$. Write $m = hk$ for some $h \in G_1$ and $k \in K$. Then as $G_1$ acts by isometries, and because the actions of $G_1$ and $G_2$ commute,
\[
    d(g^{-1}_2m,m) = d(hg_2^{-1}k,hk) = d(g_2^{-1}k, k).
\]
As $K$ is compact, $d(g_2^{-1}k,k)$ is bounded, and hence so is $e^{ad(g_2^{-1}k,k)^2}$. Thus there exists a $B > 0$ such that $e^{ad(g_2^{-1}m,m)^2} \le B$ for all $m \in M_1$. Therefore, we see that
\[
    \norm{g_2D^2e^{-tD^2}(m,m')} \le  BC e^{-a d(m,m')^2/2}
\]
and so $g_2D^je^{-tD^2}$ is $g_1$-trace class. 
\end{proof}

\begin{lemma}\thlabel{Lem: Euler Independent of t}
The number $\Tr_{g_1}\big((-1)^{F_1}g_2e^{-t\Delta_{E_1}}\big)$ is independent of $t$.
\end{lemma}

\begin{proof}
Again, set $D = \nabla^{E_1} + (\nabla^{E_1})^*$. From Lemma 2.10 in \cite{hochs2022equivariant}, 
\begin{equation}\label{Eq: Delocal Euler Independent Proof}
    \frac{d}{dt}\Tr_{g_1}\big((-1)^{F_1}g_2e^{-tD^2}\big) = \Tr_{g_1}\Big((-1)^{F_1}g_2\frac{d}{dt}e^{-tD^2}\Big) = -\Tr_{g_1}\big((-1)^{F_1}g_2D^2e^{-tD^2}\big).
\end{equation}
Note that \eqref{Eq: Delocal Euler Independent Proof} is well-defined as $g_2D^2e^{-tD^2}$ is $g_1$-trace class for all $t > 0$ by \thref{Lem: Euler Characterstic is Trace Class}. Now, as $D$ is an odd operator and $g_2$ commutes with $D$,
\[
    \Tr_{g_1}\big((-1)^{F_1}g_2D^2e^{-tD^2}\big) = \frac{1}{2}\Tr_{g_1}\big((-1)^{F_1}(g_2DDe^{-tD^2} + De^{-tD^2}(g_2D))\big).
\]
If $D_{\pm}$ denotes the restriction of $D$ to the even and odd degree differential forms, respectively, we write
\[
    g_2D = \begin{pmatrix}
        0 & g_2D_-\\
        g_2D_+ & 0
    \end{pmatrix}
\]
and
\[
    De^{-tD^2} = \begin{pmatrix}
        0 & D_-e^{-tD_+D_-}\\
        D_+e^{-tD_-D_+} & 0
    \end{pmatrix}.
\]
Then
\begin{equation}\label{Eq: Derivative 0 Zero TP}
\begin{aligned}
    \Tr_{g_1}\big((-1)^{F_1}(g_2DDe^{-tD^2} + D&e^{-tD^2}(g_2D))\big) \\
    &= \Tr_{g_1}\big(g_2D_-D_+e^{-tD_-D_+} + (D_-e^{-tD_+D_-})g_2D_+\big)\\
    &\qquad - \Tr_{g_1}\big(g_2D_+D_-e^{-tD_+D_-} + (D_+e^{-tD_-D_+})g_2D_-\big).
\end{aligned}
\end{equation}
Now, note that $g_2D$ has a distributional kernel, and $De^{-tD^2}$ has a smooth kernel. Moreover, $g_2D(De^{-tD^2}) = (De^{-tD^2})g_2D = g_2D^2e^{-tD^2}$ which is $g_1$-trace class by \thref{Lem: Euler Characterstic is Trace Class}. Hence the trace property of the $g_1$-trace, Lemma 3.2 in \cite{10.1093/imrn/rnab324}, implies that \eqref{Eq: Derivative 0 Zero TP} is equal to 0. Therefore
\[
    \frac{d}{dt}\Tr_{g_1}\big((-1)^{F_1}g_2e^{-tD^2}\big) = 0
\]
showing that $\Tr_{g_1}\big((-1)^{F_1}g_2e^{-tD^2}\big)$ is independent of $t$.
\end{proof}

By \thref{Lem: Euler Characterstic is Trace Class} and \thref{Lem: Euler Independent of t} the following definition makes sense.

\begin{definition}\thlabel{Def: New Euler Characteristic}
We define the \textbf{$(g_1,g_2)$-Euler characteristic} for the action of $G_1 \times G_2$ on $E_1 \to M_1$ by
\[
    \chi_{(g_1,g_2)}(\nabla^{E_1}) = \Tr_{g_1}\big((-1)^{F_1}g_2e^{-t\Delta_{E_1}}\big),
\]
for $t > 0$.
\end{definition}

The $(g_1,g_2)$-Euler characteristic appears in the fibration formula for equivariant analytic torsion, \thref{Thm: New Torsion Formula}, which is used to calculate the equivariant torsion for the suspension. 

\begin{remark}
In the special case that $G_2 = \{e\}$, Theorem 5.4 in \cite{hochs2022equivariant} shows that $\chi_{(g,e)}(\nabla^{E_1})$ has a purely topological description. 
\end{remark}

If $G_1$ is a compact group, the follow Lemma gives a precise formula for the $(g_1,g_2)$-Euler characteristic of $E_1$, as a Lefschetz number of the isometry $g_1g_2$ on $M_1$. First we need the following definitions.

\begin{definition}
Suppose $V^\bullet$ is a $\Z$-graded vector space, with $V^k$ finite dimensional for all $k \in \Z$, and that only finitely many $V^k$ are non-trivial. If $A$ is linear map acting on $V^\bullet$, which preserves the degree, then we write
\[
    \Tr^{V^\bullet}(A) = \sum_{k \in \Z}\Tr(A|_{V^k})
\]
\end{definition}

\begin{definition}
If $M_1$ is compact, we set $H^\bullet(M_1,E_1)$ to be the cohomology of the de Rham complex $\big(\Omega^\bullet(M_1,E_1), \nabla^{E_1}\big)$.
\end{definition}

Recall that by Hodge theory $H^\bullet(M_1,E_1)$ is isomorphic to the kernel of $\Delta_{E_1}$.

\begin{lemma}\thlabel{Lem: g-Trace Compact Group}
Suppose $G_1$ is compact. Then
\[
    \chi_{(g_1,g_2)}(\nabla^{E_1}) = \Tr^{H^\bullet(M_1,E_1)}\big((-1)^Fg_1^*g_2^*\big),
\]
where $g_j^*$ are the induced maps on cohomology.
\end{lemma}

\begin{proof}
If $G_1$ is compact, then $M_1$ is compact. Hence we may normalise so $G_1/Z_1$ has unit volume, and take $\psi \equiv 1$, to obtain
\begin{align*}
    \Tr_{g_1}\big((-1)^{F_1}g_2e^{-t\Delta_{E_1}}\big) &= \Tr\big((-1)^{F_1}g_1g_2e^{-t\Delta_{E_1}}\big).
\end{align*}
The result now follows from classical theory, as 
\[
    \Tr\big((-1)^{F_1}g_1g_2e^{-t\Delta_{E_1}}\big) = \Tr^{H^\bullet(M_1,E_1)}\big((-1)^{F_1}g_1^*g_2^*\big).
\]
See Lemma 6.5 in \cite{BGV} for more details.
\end{proof}

\subsection{A Fixed Point Theorem}\label{Subsec: Fixed Point Theorem}

In this subsection we provide a topological description of $\chi_{(g_1,g_2)}(\nabla^{E_1})$ in the case of discrete and isolated fixed points. This involves a generalisation of the Atiyah--Bott Lefschetz fixed-point formula to non-compact manifolds, and may be of independent interest. 

Let $M$ be a complete, oriented, Riemannian manifold. Suppose that $G$ is a locally compact unimodular group acting properly, cocompactly, and isometrically on $M$ preserving the orientation. Let $W \to M$ be a Hermitian $G$-vector bundle. 

Let $(\kappa_t)_{t>0}$ be a family of smooth sections of $W \boxtimes W^* \to M \times M$ such that there exist $a_1,a_2,a_3>0$ such that for all $t \in (0,1]$ and $m,m' \in M$, we have
\begin{equation}\label{Eq: Useful Kernel Bound}
    \norm{\kappa_t(m,m')} \le a_1t^{-a_2}e^{-a_3d(m,m')^2/t}.
\end{equation}

Let $g \in G$, and suppose that $Z = Z_G(g)$ is unimodular. Further suppose that there exists a \v Svarc--Milnor function $l$ with respect to $g$. Let $\psi$ denote a cutoff function for the action of $G$ on $M$. 

Let $T \colon M \to M$ be a $G$-equivariant isometry, covered by a $G$-equivariant, metric preserving vector bundle endomorphism $T_{W} \colon W \to W$.

\begin{lemma}\thlabel{Lem: Limit g-Trace to 0}
There exists a compact subset $X \subseteq G/Z$ such that
\begin{equation}\label{Eq: Limit g-trace to 0}
    \lim_{t\to0^+}\int_{(G/Z)\setminus X}\int_M \psi(hgh^{-1}m)\tr\big(hgh^{-1}T_{W}\kappa_t(T^{-1}hg^{-1}h^{-1}m,m)\big)\;dm\;d(hZ) = 0.
\end{equation}
\end{lemma}

\begin{proof}
Let $b_1,b_2 > 0$ be such that for all $h \in G$ and $m \in \supp\psi$ we have,
\[
    d(hgh^{-1}m,m) \ge b_1l(hgh^{-1}) - b_2.
\]
As $\supp\psi$ is compact, there is a $c > 0$ such that $d(T^{-1}m,m) \le c$ for all $m \in \supp\psi$. 
Then for all $m\in  \supp\psi$, 
\begin{equation}\label{Eq: Limit g-trace to 0 proof 1}
    d(T^{-1}hgh^{-1}m,m) \ge d(hgh^{-1}m,m) - d(T^{-1}m,m)\ge   b_1l(hgh^{-1}) - b_2 - c.
\end{equation}

Following the proofs of Lemmas 3.10 and 3.12 in \cite{hochs2022equivariant}, and using \eqref{Eq: Useful Kernel Bound} and \eqref{Eq: Limit g-trace to 0 proof 1}, we find that there exists a compact subset $X \subseteq G/Z$ such that 
\begin{equation}\label{Eq: Limit g-trace to 0 proof 2}
    \int_{(G/Z)\setminus X}\int_M \psi(hgh^{-1}m)\tr\big(hgh^{-1}T_{W}\kappa_t(T^{-1}hg^{-1}h^{-1}m,m)\big)\;dm\;d(hZ)
\end{equation}
is $\mathcal O(t^a)$ as $t \to 0^+$ for all $a > 0$. 
\end{proof}

Let $q : G \to G/Z$ be the quotient map, and let $X$ be a compact subset of $G/Z$ as in \thref{Lem: Limit g-Trace to 0}. The action of $Z$ on $q^{-1}(X)$ by right multiplication is proper and cocompact. Thus there exists $\psi_X \in C^\infty_c(G)$ such that for all $h \in q^{-1}(X)$ we have
\[
    \int_Z \psi_X(hz)\;dz = 1.
\]
Let $\psi_X^g$ be the function on $M$ defined by
\[
    \psi^g_X(m) = \int_{q^{-1}(X)}\psi_X(h)\psi(hgm)\;dh.
\]

\begin{lemma}\thlabel{Lem: g-Trace Compact Subset}
We have the equality
\begin{equation}
\begin{aligned}
    \int_X\int_M\psi(hgh^{-1}m)\tr\big(hgh^{-1}&T_{W}\kappa_t(T^{-1}hg^{-1}h^{-1}m,m)\big)\;dm\;d(hZ)\\
    &= \int_M \psi^g_X(m)\tr\big(gT_{W}\kappa_t(T^{-1}g^{-1}m,m)\big)\;dm.
\end{aligned}
\end{equation}
\end{lemma}

\begin{proof}
Analogous to the proof of Lemma 3.1(b) in \cite{FPF.and.Character.Formula} or Lemma 4.7 in \cite{hochs2022equivariant}.
\end{proof}

Suppose that $E_1 \to M$ is a flat $G$-equivariant Hermitian vector bundle, with $\nabla^{E_1}$ a flat $G$-invariant metric-preserving connection. Suppose further that $T_{E_1} \colon E_1 \to E_1$ is a $G$-equivariant metric-preserving vector bundle endomorphism covering $T \colon M \to M$. Set $W = \extp\bullet T^*M \otimes E_1$. Then $T_{E_1}$ and the derivative of $T$ induce a $G$-equivariant metric-preserving vector bundle endomorphism $T_W : W \to W$ covering $T$.

Let $\Delta_{E_1}$ denote the Hodge Laplacian acting on $E_1$-valued differential forms. For $t > 0$ set $\kappa_t$ to be the Schwarz kernel of $e^{-t\Delta_{E_1}}$, which satisfies \eqref{Eq: Useful Kernel Bound} by Proposition 4.2 in \cite{carey2014index}. 

Suppose that $F$ denotes the number operator on $W$. Then the proof of \thref{Lem: Euler Independent of t} shows that the number $\Tr_g\big((-1)^FT_{W} \circ \kappa_t \circ T\big)$ is independent of $t$. Hence combining Lemmas \ref{Lem: Limit g-Trace to 0} and \ref{Lem: g-Trace Compact Subset}, we obtain the following.

\begin{proposition}\thlabel{Prop: g-trace as Limit}
We have the equality
\begin{equation}\label{Eq: g-Trace as limit}
    \Tr_g\big((-1)^F T_{W} \circ \kappa_t \circ T\big) = \lim_{t\to 0^+}\int_M \psi^g_X(m)\tr\big((-1)^FgT_{W}\kappa_t(T^{-1}g^{-1}m,m)\big)\;dm.
\end{equation}
\end{proposition}

From now on, suppose that the fixed points of $gT$ are isolated and nondegenerate. For every $m \in M^{gT}$, choose $\chi_m \in C^\infty_c(M)$ such that $\chi_m \equiv 1$ on a neighbourhood of $m$, and $\chi_m\chi_{m'} = 0$ if $m \neq m'$. Then using \eqref{Eq: Useful Kernel Bound}, we find that the limit in \eqref{Eq: g-Trace as limit} equals
\begin{equation}
    \sum_{m \in M^{gT} \cap \supp \psi_X^g}\int_M \chi_m(m')\psi^g_X(m')\tr\big((-1)^FgT_{W}\kappa_t(T^{-1}g^{-1}m',m')\big)\;dm'.
\end{equation}
Now applying Lemma 6.4 in \cite{BGV}, which also applies if $M$ is non-compact as long as the function $\chi$ has compact support, we obtain the following result.

\begin{proposition}\thlabel{Prop: Intermediate Fixed Point}
Suppose that the fixed point set $M^{gT}$ is discrete and contains only nondegenerate fixed points. Then
\begin{equation}
    \Tr_g\big((-1)^FT_{W} \circ \kappa_t \circ T\big) = \sum_{m \in M^{gT}}\psi^g_X(m)\frac{\tr\big((-1)^FgT_{W}|_{W_m}\big)}{\abs{\det\big(1 - D_m(g^{-1}T^{-1})\big)}}.
\end{equation}
\end{proposition}

Special cases of the following fact were proved in Lemma 2.4 in \cite{FPF.and.Character.Formula} and Lemma 5.22 in \cite{Piazza2025}.
\begin{lemma}\thlabel{Lem: Fixed Point Cocompact}
If the conjugacy class of $g$ is closed, then the quotient $M^{gT}/Z$ is compact.
\end{lemma}

\begin{proof}
Set
\[
    \Stab_T(M) = \big\{(m,x) \in M \times G : gT(m) = m\}
\]
Then $G$ acts on $\Stab_T(M)$ by $h \cdot (m,x) = (hm, hxh^{-1})$ for $h \in G$ and $(m,x) \in \Stab_T(M)$. We note that the map $\Stab_T(M) \to M$ given by the projection onto the first coordinate is a $G$-equivariant, continuous, and surjective map. The fibres of this map are compact as the $G$-action on $M$ is proper. Moreover, as $\Stab_T(M) \subseteq M \times G$ is closed, and $M/G$ is compact, it then follows that $\Stab_T(M)/G$ is also compact. Now the subset $\Stab_T(M) \cap \big(M \times (g)\big)$ is closed as $(g) \subseteq G$ is closed, and as $(g)$ is $G$-invariant we find that
\[
    \big(\Stab_T(M) \cap (M \times (g))\big)/G 
\]
is also compact. 

We claim that the map $\eta : M^{gT}/Z \to \big(\Stab_T(M) \cap (M \times (g))\big)/G$ given by $\eta(Zm) = G(m,g)$ is a homeomorphism. The fact that $\eta$ is well-defined follows from the definitions. Let $m,m' \in M^{gT}$ and suppose that $\eta(Zm) = \eta(Zm')$. Then there exists $x \in G$ such that $m' = xm$ and $g = xgx^{-1}$, so that $x \in Z$ and $Zm=Zm'$ showing that $\eta$ is injective. For surjectivity, let $m \in M$ and $x \in G$ be such that $(m,xgx^{-1}) \in \Stab_T(M)$. Then we find that $x^{-1}m \in M^{gT}$, and 
\[
    G(m,xgx^{-1}) = G(x^{-1}m,g) = \eta(Zx^{-1}m).
\]
Thus $\eta$ is bijective. Note that the inclusion map $M^{gT} \times \{g\} \hookrightarrow \Stab_T(M)$ is continuous and closed, and hence so is the induced map
\begin{equation}\label{Eq: M^gT/Z Compact Proof 1}
    \big(M^{gT} \times \{g\}\big)/Z \to \Stab_T(M)/Z.
\end{equation}
The quotient map
\begin{equation}\label{Eq: M^gT/Z Compact Proof 2}
    \Stab_T(M)/Z \to \Stab_T(M)/G
\end{equation}
is continuous and closed as well, and therefore the composition of \eqref{Eq: M^gT/Z Compact Proof 1} and \eqref{Eq: M^gT/Z Compact Proof 2}
\begin{equation}\label{Eq: M^gT/Z Compact Proof 3}
    \big(M^{gT} \times \{g\}\big)/Z \to \Stab_T(M)/G
\end{equation}
is also a continuous closed map. Since $(g)$ is closed in $G$, restricting the codomain of the map \eqref{Eq: M^gT/Z Compact Proof 3} we obtain a continuous closed map
\begin{equation}\label{Eq: M^gT/Z Compact Proof 4}
    \big(M^{gT} \times \{g\}\big)/Z \to \big(\Stab_G(M) \cap (M \times (g))\big)/G.
\end{equation}
The canonical inclusion map $M^{gT} \hookrightarrow M^{gT} \times \{g\}$ is a homeomorphism, and hence so is the induced map 
\begin{equation}\label{Eq: M^gT/Z Compact Proof 5}
    M^{gT}/Z \to \big(M^{gT} \times \{g\}\big)/Z.
\end{equation}
As $\eta$ is the composition of \eqref{Eq: M^gT/Z Compact Proof 4} and \eqref{Eq: M^gT/Z Compact Proof 5}, it follows that $\eta$ is a continuous, closed, bijective map and hence a homeomorphism.
\end{proof}

Let $K \subseteq M^{gT}$ be a compact subset such that $ZK = M^{gT}$. 

\begin{lemma}\thlabel{Lem: Fixed Point Support}
There exists a compact subset $X' \subseteq G/Z$ such that for all $k \in K$ and $h \in G$ such that $hZ \notin X'$ we have $hgk \notin \supp \psi$.
\end{lemma}
\begin{proof}
As $l$ is a \v Svarc--Milnor function with respect to $g$, there are $a,b > 0$  such that \eqref{Eq: SM 3} holds for all $x \in G$ and $m \in gK$. Moreover, as $K$ is compact, there exists $R > 0$ such that $d(gk,k) \le R$ for all $k\in K$.

Let $h\in G$ and $k\in K$.
From the triangle inequality, we have 
\[
    d(hgk,k) \ge d(hgk,gk) - d(gk,k) \ge  al(h) - b - R.
 \]   
Moreover, as $l$ is a \v Svarc--Milnor function we have $l(hgh^{-1}) \le 2l(h) + l(g)$, and so it follows that
\begin{equation}\label{Eq: Fixed Point Support Lemma Proof 2}
    d(hgk,k) \ge \frac{a}{2}l(hgh^{-1}) - \frac{a}{2}l(g) - b - R.
\end{equation}
Let $r > 0$ be such that $\supp \psi \subseteq \{m \in M : d(m,K) \le r\}$, which exists since $\supp \psi$ and $K$ are compact subsets. If  $hgk \in \supp \psi$, then there is a $k' \in K$ such that $d(hgk, k') \le r + 1$. Hence 
\[
    d(hgk, k) \le r + 1 + \diam(K)
 \]   
by the triangle inequality, where $\diam(K)$ is the diameter of $K$. By \eqref{Eq: Fixed Point Support Lemma Proof 2}, this implies that
\[ 
    l(hgh^{-1}) \le \frac{al(g) + 2b + 2R + 2(r + 1 \diam(K))}{a}.
\]
We conclude that the set
\[
    X' = \Big\{hZ \in G/Z : l(hgh^{-1}) \le \frac{al(g) + 2b + 2R + 2(r + 1 \diam(K))}{a}\Big\},
\]
which is compact because $l$ is a \v Svarc--Milnor function, has the desired properties. 
\end{proof}

Let $X \subseteq G/Z$ be a compact set as in \thref{Lem: Limit g-Trace to 0}. Let $\psi_G$ be as in \eqref{Eq: Cutoff Z for G}, and such that $\psi_G|_{q^{-1}(X)} = \psi_X$. Define $\psi^g \in C^\infty(M)$ by
\[
    \psi^g(m) = \int_G \psi_G(x)\psi(xgm)\;dx 
\]
analogously to \eqref{Eq: Def cutoff^g}.

\begin{lemma}\thlabel{Lem: Equality for cutoff^g}
The compact set $X \subseteq G/Z$ as in \thref{Lem: Limit g-Trace to 0} can be chosen such that for all $m \in M^{gT}$ we have
\[
    \psi^g_X(m) = \psi^g(m),
\]
where $\psi^g \in C^\infty(M)$ is the function defined by \eqref{Eq: Def cutoff^g}.
\end{lemma}

\begin{proof}
Let $\tilde X \subseteq G/Z$ be a compact subset as in \thref{Lem: Limit g-Trace to 0}. Let $X'$ be as in \thref{Lem: Fixed Point Support}. Then the set $X \coloneqq \tilde X \cup X'$ still satisfies \eqref{Eq: Limit g-trace to 0}. Let $k \in K$ and $z \in Z$. If $h \in G \setminus q^{-1}(X)$, then $hz \in G \setminus q^{-1}(X)$, so $hgzk = hzgk \notin \supp \psi$ by \thref{Lem: Fixed Point Support}. Hence
\[
    \psi^g(zk) = \int_G \psi_G(h)\psi(hgzk)\;dh = \int_{q^{-1}(X)}\psi_G(h)\psi(hgzk)\; dh = \psi_X^g(zk).
\]
The result now follows as $ZK = M^{gT}$.
\end{proof}

Combining \thref{Prop: Intermediate Fixed Point} and \thref{Lem: Equality for cutoff^g} we obtain the following fixed point formula. 

\begin{theorem}\thlabel{Thm: Fixed Point Theorem}
Suppose that the fixed point set $M^{gT}$ is discrete and contains only nondegenerate fixed points. Then
\begin{equation}
    \Tr_g\big((-1)^FT_{W} \circ \kappa_t \circ T\big) = \sum_{m \in M^{gT}}\psi^g(m)\frac{\tr\big((-1)^FgT_{W}|_{W_m}\big)}{\abs{\det\big(1 - D_m(g^{-1}T^{-1})\big)}}.
\end{equation}
\end{theorem}

In the special case where $G = G_1$ and $G_2$ are two locally compact unimodular groups acting on $M$ with the assumptions stated at the start of Subsection \ref{Subsec: Euler Char}. Taking $T$ and $T_{E_1}$ to be the respective actions by some element $g_2 \in G_2$, and $g = g_1$, we obtain the following corollary.

\begin{corollary}\thlabel{Prop: Euler Char as Fixed Points}
Suppose that the fixed point set $M^{g_1g_2}$ is discrete and contains only nondegenerate fixed points. Then
\begin{equation}
    \chi_{(g_1,g_2)}(\nabla^{E_1})= \sum_{m \in M^{g_1g_2}}\psi^{g_1}(m)\frac{\tr\big((-1)^Fg_1g_2|_{\extp \bullet T^*_mM \otimes (E_1)_m}\big)}{\abs{\det\big(1 - D_m(g_1^{-1}g_2^{-1})\big)}}.
\end{equation}
\end{corollary}

\begin{remark}
Let $D$ be a  Dirac operator on $M$, defined in terms of a $G$-invariant Clifford action and Clifford connection on a Hermitian $G$-vector bundle $S \to M$. Suppose that $S$ has a $G$-invariant $\Z/2\Z$-grading, defined by a grading operator $\gamma$, that the Clifford action used to define $D$ is odd for this grading, and that the Clifford connection used to define $D$ is even. 

Let $\kappa_t$ be the Schwartz kernel of $e^{-tD^2}$. 
Analogously to Proposition \ref{Prop: g-trace as Limit}, we have
\begin{equation}\label{Eq: g-Trace as a limit D}
    \Tr_g(\gamma \kappa_t ) = \lim_{t\to 0^+}\int_M \psi^g_X(m)\tr\big(\gamma g\kappa_t(g^{-1}m,m)\big)\;dm.
\end{equation}
By compactness of the support of $\psi^g_X$, we may now use standard heat kernel asymptotics, see e.g.\ Theorem 6.11 in \cite{BGV}, to show that the right hand side of \eqref{Eq: g-Trace as a limit D} equals
\begin{equation} \label{eq g index thm}
\int_{M^g} \psi_X^g \operatorname{AS}_g(D),
\end{equation}
where $\operatorname{AS}_g(D)$ is the Atiyah--Segal--Singer integrand associated to $D$. By Lemma \ref{Lem: Equality for cutoff^g} we then obtain
\[
 \Tr_g(\gamma \kappa_t ) = \int_{M^g} \psi^g \operatorname{AS}_g(D).
\]
This implies the main results in \cite{FPF.and.Character.Formula, WW16}.

There were two issues with the original proof of the main result in \cite{FPF.and.Character.Formula}: in this paper, the decomposition (4.2) of a Dirac operator on a fibred product is missing an order zero term, and the proof of the inequality (4.6) is incomplete. The argument given here leads to a complete proof of the result in \cite{FPF.and.Character.Formula}. Another complete proof was given in Theorem 5.32 in \cite{Piazza2025}.
\end{remark}

\section{A Fibration Formula}\label{Sec: Fibration Formula}

In this section we state and prove a fibration formula for equivariant analytic torsion, \thref{Thm: New Torsion Formula}. In the process, we also prove quotient and product formulae for equivariant analytic torsion. \thref{Thm: New Torsion Formula} will be used in Section \ref{Sec: Suspension Flow} to calculate the equivariant analytic torsion for the suspension.

\subsection{Fibration Formula}\label{Subsec: Fibration Formula}

Let $M_1$ and $M_2$ be complete, oriented Riemannian manifolds and let $G$, $\Gamma$ be locally compact unimodular groups with $\Gamma$ discrete, abelian, and countable. Suppose $G$ acts on $M_1$ as stated in the first paragraph of Subsection \ref{Subsec: Torsion Prelims}, and also that $\Gamma$ acts on $M_2$ with the assumptions stated in the first paragraph of Subsection \ref{Subsec: Torsion Prelims}. Then $\Gamma$ is finitely generated by the \v Svarc--Milnor lemma.

Further suppose that $\Gamma$ acts on $M_1$ isometrically, and that the action of $\Gamma$ on $M_1$ commutes with action of $G$ on $M_1$. Note that we do not assume that $\Gamma$ acts on $M_1$ properly. Then we have an action of $G \times \Gamma$ on $M_1 \times M_2$, given by \eqref{Eq: Twisted Product Action}. This action is isometric, and preserves the orientation on $M_1 \times M_2$. It is cocompact, and proper by \thref{Lem: Twisted Product Action Proper}. Finally this induces commuting actions of $\Gamma$ and $G$ on $M_1 \times M_2$, and we assume that the $\Gamma$-action on $M_1 \times M_2$ is free.  

Suppose $E_1 \to M_1$ is a flat, $G \times \Gamma$-equivariant, Hermitian vector bundle with $\nabla^{E_1}$ a flat, $G \times \Gamma$-invariant, metric-preserving connection. Similarly suppose $E_2 \to M_2$ is a flat, $\Gamma$-equivariant, Hermitian vector bundle with $\nabla^{E_2}$ a flat, $\Gamma$-invariant, metric-preserving connection. Then $E \coloneqq E_1 \boxtimes E_2 \to M_1 \times M_2$ is a $G \times \Gamma$-equivariant flat vector bundle with flat connection 
\[
    \nabla^E = \nabla^{E_1} \otimes 1 + 1 \otimes \nabla^{E_2}. 
\]
We use these connections to form the Hodge Laplacians $\Delta_{E_1}$, $\Delta_{E_2}$, and $\Delta_E$, for the respective vector bundles. 

As the $\Gamma$-action is free and proper $M \coloneqq (M_1 \times M_2)/\Gamma$ is a smooth manifold. We note that the map $M \to M_2/\Gamma$ induced by the projection of $M_1 \times M_2$ onto $M_2$ gives $M$ the structure of a fibre bundle over $M_2/\Gamma$ with fibres diffeomorphic to $M_1$. Moreover, we have a flat $G$-equivariant vector bundle $E/\Gamma \to M$, with flat connection $\nabla^{E/\Gamma}$ obtained by restricting $\nabla^E$ to the $\Gamma$-invariant sections. Let $\Delta_{E/\Gamma}$ denote the corresponding Hodge Laplacian. Finally, there is an induced action of $G$ on $M$, which is proper by \thref{Lem: Proper Action on Quotient}.

Before we state the main result of this section, \thref{Thm: New Torsion Formula}, we also extend the definition of analytic torsion to the case of a product group action, which may not act properly on a manifold $N$.

\begin{definition}\thlabel{Def: Weird Torsion}
Let $G_1$ and $G_2$ be two locally compact unimodular groups, and suppose that $N$ is a complete, oriented, Riemannian manifold. Suppose that $G_1 \times G_2$ acts on $N$ isometrically and preserves the orientation. Assume that the restricted action of $G_1$ on $N$ is proper and cocompact, however we do not assume that the action of $G_2$ on $N$ is proper. If $g_2$ is a central element of $G_2$, and $g_2(e^{-t\Delta_E} - P^E)$ is $g_1$-trace class, we set
\begin{equation}\label{Eq: Product Trace}
    \mathcal T_{(g_1,g_2)}(t) \coloneqq \Tr_{g_1}\big(g_2(-1)^FF(e^{-t\Delta_E} - P^E)\big).
\end{equation}
The equivariant analytic torsion for the action of $G_1 \times G_2$ on $N$, $T_{(g_1,g_2)}(\nabla^E)$, is then defined in the same way \thref{Def: Torsion} but with $\mathcal T_{(g_1,g_2)}(t)$ replacing $\mathcal T_g(t)$ in \eqref{Eq: Torsion Def}.
\end{definition}

The following theorem, \thref{Thm: New Torsion Formula}, relates the equivariant analytic torsion of $M$ to the equivariant analytic torsion of $M_1$ and $M_2$. We call it a fibration formula as $M$ is a fibration over $M_2/\Gamma$ with fibre $M_1$.

\begin{theorem}\thlabel{Thm: New Torsion Formula}
Let $g \in G$. Suppose that $\ker \Delta_E = \ker \Delta_{E/\Gamma} = \{0\}$, and that $P^{E_1}$ is $g$-trace class. Suppose that one of the following sets of conditions holds.  
\begin{enumerate}[label=\emph{(\Roman*)}]
    \item The space $G/Z$ is compact and the Novikov--Shubin numbers $(\alpha_{e_j}^{p_j})_{E_j}$ for $\Delta_{E_j}$ are positive for $j=1,2$.
    \item There exists a \v Svarc--Milnor function $l_G$ for the action of $G$ on $M_1$ with respect to $g$, and a \v Svarc--Milnor function for the action of $G \times \Gamma$ on $M_1 \times M_2$ with respect to $(g,\gamma)$ is given by the sum of $l_G$ and a word length metric on $\Gamma$ for all $\gamma \in \Gamma$. Further suppose that the volume of
    \[
        \big\{hZ \in G/Z : l_G(hgh^{-1}) \le r\big\}
    \]
    has at most exponential growth in $r$. Finally suppose that $\Tr_{g}(\gamma P^{E_1}) = \Tr_{\gamma}(P^{E_2}) = 0$ for all $\gamma \in \Gamma$.
\end{enumerate}
Then for $\sigma \in \C$ with $\re(\sigma) > 0$ large enough, we have
\begin{equation}\label{Eq: Fibration Formula Sigma}
    T_g(\nabla^{E/\Gamma},\sigma) = \prod_{\gamma \in \Gamma}T_{(g,\gamma)}(\nabla^{E_1},\sigma)^{\chi_\gamma(\nabla^{E_2})}T_\gamma(\nabla^{E_2},\sigma)^{\chi_{(g,\gamma)}(\nabla^{E_1})}.
\end{equation}
Moreover, the equivariant analytic torsion $T_g(\nabla^{E/\Gamma})$ is well-defined if, and only if, the product on the right hand side of \eqref{Eq: Fibration Formula Sigma} has a meromorphic extension which is holomorphic at $0$. 
\end{theorem}

The proof can be found at the end of Subsection \ref{Subsec: Product Formula}. It combines a quotient theorem for analytic torsion, \thref{Thm: Quotient Formula}, with a product formula, \thref{Thm: Product Formula}.

\subsection{Quotient Formula}\label{Subsec: Quotient Formula}

Here we present a quotient formula for equivariant analytic torsion, which is a generalisation of Proposition 5.12 in \cite{hochs2022equivariant}. 

Suppose $G, \Gamma$ are locally compact, unimodular groups, with $\Gamma$ discrete, countable, and finitely generated. Suppose that $M$ is an oriented Riemannian manifold, and that $G \times \Gamma$ acts on $M$ properly, isometrically and preserving the orientation. Assume the restricted $G$-action on $M$ is proper and cocompact, and that the restricted $\Gamma$-action on $M$ is free. 

Suppose now that $W \to M$ is a $G \times \Gamma$-equivariant, flat, Hermitian vector bundle, and that $\nabla^W$ is a flat, $G \times \Gamma$-invariant, metric-preserving connection. This then induces a flat $G$-equivariant vector bundle $W/\Gamma \to M/\Gamma$, with flat connection $\nabla^{W/\Gamma}$ obtained by restricting $\nabla^W$ to the $\Gamma$-invariant sections. Let $\Delta_W$ and $\Delta_{W/\Gamma}$ denote the respective Hodge Laplacians of the connections $\nabla^W$ and $\nabla^{W/\Gamma}$.

\begin{theorem}\thlabel{Thm: Quotient Formula}
Let $g \in G$ and suppose there exists a \v Svarc--Milnor function, $l_G$, for the action of $G$ on $M$ with respect to $g$. Suppose that a word length metric, $l_{\Gamma}$, is a \v Svarc--Milnor function for the action of $\Gamma$ on $M$, and suppose that $l_G + l_{\Gamma}$ is a \v Svarc--Milnor function for the action of $G \times \Gamma$ on $M$ with respect to $(g,\gamma)$ for all $\gamma \in \Gamma$. Suppose the volume of
\begin{equation}\label{Eq: Quotient Formula Exponential Volume}
    \big\{hZ \in G/Z : l_G(hgh^{-1}) \le r\big\}
\end{equation}
has at most exponential growth in $r$. 

Finally suppose $\ker(\Delta_W) = \ker(\Delta_{W/\Gamma}) = \{0\}$. Then the expression
\[
    \prod_{(\gamma)}T_{(g,\gamma)}(\nabla^W, \sigma),
\]
where the product is over conjugacy classes of $\Gamma$, converges for $\sigma \in \C$ with $\re(\sigma) > 0$ large enough and at such a $\sigma$,
\begin{equation}\label{Eq: Torsion Quotient Formula}
    T_g(\nabla^{W/\Gamma}, \sigma) = \prod_{(\gamma)}T_{(g,\gamma)}(\nabla^W, \sigma).
\end{equation}
Hence $T_g(\nabla^{E/\Gamma})$ is well-defined if, and only if, the product on the right hand side of \eqref{Eq: Torsion Quotient Formula} has a meromorphic extension to $\C$ which is holomorphic at $\sigma = 0$.
\end{theorem}

\begin{remark}
The additional assumptions on the \v Svarc--Milnor functions in \thref{Thm: Quotient Formula} compared to Proposition 5.12 in \cite{hochs2022equivariant} are because without the assumption that the $\Gamma$-action on $M$ is cocompact we cannot use the \v Svarc--Milnor lemma to determine that a word length metric is a \v Svarc--Milnor function for the action of $\Gamma$ on $M$, and that $\Gamma$ is finitely generated. The fact that a word length metric is a \v Svarc--Milnor function is needed in \thref{Lem: Quotient Formula Small t}. Moreover, this assumption is satisfied under the setting of \thref{Thm: New Torsion Formula}.
\end{remark}

\begin{remark}
Recall that if $G/Z$ is compact, then $l_G \equiv 0$ is a \v Svarc--Milnor function for the action of $G$ on $M$. Thus the assumptions in \thref{Thm: Quotient Formula} on the \v Svarc--Milnor functions aid in controlling the analysis in the case that $G/Z$ is not compact. See \thref{Prop: Exponential Volume Growth Semisimple Lie} for an example where \eqref{Eq: Quotient Formula Exponential Volume} has at most exponential volume growth.
\end{remark}

The proof of \thref{Thm: Quotient Formula} is a generalisation of the proof of Proposition 5.12 in \cite{hochs2022equivariant}. Here we indicate how to modify the proof of Proposition 5.12 \cite{hochs2022equivariant}, so it provides a proof for \thref{Thm: Quotient Formula}.

From now on let $\psi$ be a cutoff function for the action of $G \times \Gamma$ on $M$, and $\kappa_t$ be the Schwartz kernel of $(-1)^FFe^{-t\Delta_W}$. 

\begin{lemma}\thlabel{Lem: Quotient Formula Large t}
Under the assumptions of \thref{Thm: Quotient Formula}, the integral 
\begin{equation}\label{Eq: Quotient Theorem Large t Integral}
    \sum_{x \in \Gamma}\int_1^\infty t^{-1}e^{-\sigma t}\int_{G/Z}\int_M \abs{\psi(m) \tr\big(hgh^{-1}x\kappa_t\big(hg^{-1}h^{-1}x^{-1}m,m\big)\big)} \; dm\; d(hZ) \; dt
\end{equation}
converges for $\sigma \in \C$ with $\re(\sigma)$ large enough.
\end{lemma}

\begin{proof}
Under the assumptions on the \v Svarc--Milnor functions, there exist $b_1,b_2 > 0$ such that
\[
    d\big(hgh^{-1}xm,m\big) \ge b_1\big(l_G(hgh^{-1}) + l_\Gamma(x)\big) - b_2
\]
for all $h \in G$, $x \in \Gamma$, and $m \in \supp(\psi)$. Moreover, for all $s,u,v > 0$ with $s \ge u - v$ implies
\[
    2(s^2 + v^2) \ge (s + v)^2 \ge u^2,
\]
and so
\begin{equation}\label{Eq: SM Bound Quotient}
    d\big(hgh^{-1}xm,m\big)^2 \ge \frac{b_1^2l_G(hgh^{-1})^2}{2} + \frac{b_1^2l_\Gamma(x)^2}{2} - b_2^2.
\end{equation}
From Proposition 3.18 in \cite{hochs2022equivariant} there exist $C, a > 0$ be such that
\[
    \norm{\kappa_t(m,m')} \le C e^{-ad(m,m')^2/t}
\]
for all $t \ge 1$, and $m,m' \in M$. Then using \eqref{Eq: SM Bound Quotient} we find that
\begin{equation}
\begin{aligned}
    &\int_M \abs{\psi(m) \tr\big(hgh^{-1}x\kappa_t(hg^{-1}h^{-1}x^{-1}m,m)\big)} \; dm\\
    &\qquad\qquad\qquad\qquad\le \tilde C \vol(\supp \psi) e^{ab_2^2}e^{-ab_1^2l_\Gamma(x)^2/2t}e^{-ab_1^2l_G(hgh^{-1})^2/2t},
\end{aligned}
\end{equation}
for some $\tilde C > 0$. Hence the absolute value of \eqref{Eq: Quotient Theorem Large t Integral} is bounded above by
\begin{equation}
\begin{aligned}
    & \tilde C \vol(\supp \psi) e^{ab_2^2}\sum_{x \in \Gamma}\int_1^\infty t^{-1}e^{-\re(\sigma) t}e^{-ab_1^2l_\Gamma(x)^2/2t} \int_{G/Z} e^{-ab_1^2l_G(hgh^{-1})^2/2t} \; d(hZ) \; dt.
\end{aligned}
\end{equation}
Define a function $f\colon (0,\infty) \to (0,\infty)$ by
\[
    f(t) \coloneqq \int_{G/Z} e^{-ab_1^2l_G(hgh^{-1})^2/2t} \; d(hZ).
\]
This converges because $l_G$ is a \v Svarc--Milnor function. We have proven the lemma if we can show that
\begin{equation}\label{Eq: Quotient Formula Large t proof}
    \sum_{x \in \Gamma}\int_1^\infty t^{-1}e^{-\re(\sigma) t}e^{-ab_1^2l_\Gamma(x)^2/2t}f(t) \; dt
\end{equation}
converges for $\re(\sigma)$ large enough. First note that as $\Gamma$ has at most exponential growth, 
\[
    \sum_{x \in \Gamma}\int_1^\infty t^{-1}e^{-\re(\sigma) t}e^{-ab_1^2l_\Gamma(x)^2/2t}dt
\]
converges for $\re(\sigma)$ large enough by Lemma 4.8 in \cite{FPF.and.Character.Formula}. To deal with $f(t)$, let $r > 0$ and set
\begin{align*}
    F_1(r) &= \vol\{hZ \in G/Z : l_G(hgh^{-1}) \le r\},\\
    F_2(t) &= \sum_{j = 0}^\infty F_1(j)e^{-ab_1^2j^2/2t}.
\end{align*}
Then $f(t) \le F_2(t)$ for all $t > 0$. The assumption on the volume growth in \thref{Thm: Quotient Formula} implies that $F_1(r) = \mathcal O(e^{br})$ for some $b \in \R$ as $r \to \infty$, and so $F_2(t) = \mathcal O(t^{1/2}e^{2b^2t/4ab_1^2})$ at $t \to \infty$ by Lemma 3.21 in \cite{hochs2022equivariant}. Therefore, we find that 
\[
    \sum_{x \in \Gamma}\int_1^\infty t^{-1}e^{-\re(\sigma) t}e^{-ab_1^2l_\Gamma(x)^2/2t}F_2(t) \; dt
\]
converges for $\re(\sigma)$ large enough, and hence so does \eqref{Eq: Quotient Formula Large t proof}, proving the result. 
\end{proof}

\begin{lemma}\thlabel{Lem: Quotient Formula Small t}
Under the assumptions of \thref{Thm: Quotient Formula}, the integral 
\[
    \sum_{x \in \Gamma}\int_0^1 t^{-1}e^{-\sigma t}\int_{G/Z}\int_M \abs{\psi(m) \tr\big(hgh^{-1}x\kappa_t\big(hg^{-1}h^{-1}x^{-1}m,m\big)\big)} \; dm\; d(hZ) \; dt
\]
converges for all $\sigma \in \C$.
\end{lemma}

\begin{proof}
Following the proof of Proposition 5.12 in \cite{hochs2022equivariant}, let $\vphi \colon \Gamma \to \R$ be defined by 
\[
    \vphi(x) \coloneqq \frac{b_1^2l_\Gamma(x)^2}{4}-b_2^2,
\]
with $b_1$ and $b_2$ as in \thref{Lem: Quotient Formula Large t}. Set $S \coloneqq \{x \in \Gamma : \vphi(x) \le 1\}$, then $S$ is finite. By Proposition 4.2(1) in \cite{carey2014index}, there exist $a_1,a_2,a_3 > 0$ such that
\[
    \norm{\kappa_t(m,m')} \le a_1t^{-a_2}e^{-a_3d(m,m')^2/t}
\]
for all $m,m' \in M$ and $t \in (0,1]$. Then in place of (5.19) in \cite{hochs2022equivariant}, using \eqref{Eq: SM Bound Quotient} we have
\begin{equation}
\begin{aligned}
    \sum_{x \in \Gamma\setminus S}&\int_0^1 t^{-1}e^{-\sigma t}\int_{G/Z}\int_M \abs{\psi(m) \tr\big(hgh^{-1}x\kappa_t\big(hg^{-1}h^{-1}x^{-1}m,m\big)\big)} \; dm\; d(hZ) \; dt\\
    &\le a_1\vol(\supp\psi)\Big(\sum_{x \in \Gamma \setminus S}e^{-a_3b_1^2l_\Gamma(x)^2/4}\Big)\Big(\int_0^1t^{-a_2-1}e^{-\sigma t}e^{-a_3/t}\;dt\Big)\\
    &\qquad\qquad\qquad\qquad\qquad\qquad \cdot \Big(\int_{G/Z}e^{-a_3b_1^2l_G(hgh^{-1})^2/2}\;d(hZ)\Big).
\end{aligned}
\end{equation}
The proof now proceeds in exactly the same way the is presented in the proof of Proposition 5.12 of \cite{hochs2022equivariant}. 
\end{proof}

\begin{lemma}\thlabel{Lem: Cutoff on Quotient}
Define $\psi_{M/\Gamma} \in C^\infty_c(M)$ by
\[
    \psi_{M/\Gamma}(\Gamma m) = \sum_{x \in \Gamma}\psi(xm).
\]
Then $\psi_{M/\Gamma}$ is a cutoff function for the action of $G$ on $M/\Gamma$.
\end{lemma}

\begin{proof}
As $\psi$ is a cutoff function for the $G \times \Gamma$ action on $M$, $\psi_{M/\Gamma}$ is a cutoff function for the $G$ action on $M/\Gamma$. It is compactly supported as the support of $\psi_{M/\Gamma}$ is contained in the image of the support of $\psi$ under the canonical projection $M \to M/\Gamma$.
\end{proof}

\begin{proof}[Proof of \thref{Thm: Quotient Formula}]
The proof proceeds in the same way as in the proof of Proposition 5.12 in \cite{hochs2022equivariant}, except now using \thref{Lem: Quotient Formula Large t} and \thref{Lem: Quotient Formula Small t} for large and small $t$ convergence of the integral, respectively. We replace (5.23) in \cite{hochs2022equivariant} with the following expression
\begin{equation}\label{Eq: Quotient Formula Proof 1}
    \sum_{x \in \Gamma}\int_{G/Z}\int_M \psi(m) \tr\big(hgh^{-1}x\kappa_t\big(hg^{-1}h^{-1}x^{-1}m,m\big)\big) \; dm\;d(hZ),
\end{equation}
which converges absolutely for every $t > 0$ by compactness of the support of $\psi$, estimates of the form \eqref{Eq: Gaussian Off-Diag Decay}, \eqref{Eq: SM Bound Quotient}, and the fact that $\Gamma$ is finitely generated. Hence \eqref{Eq: Quotient Formula Proof 1} equals
\begin{equation}\label{Eq: Quotient Formula Proof 4}
    \int_{G/Z}\int_M \psi(m)\sum_{x \in \Gamma} \tr\big(hgh^{-1}x\kappa_t\big(hg^{-1}h^{-1}x^{-1}m,m\big)\big) \; dm\;d(hZ).
\end{equation}
Now by the trace property of the fibrewise trace, the $\Gamma$-equivariance of $\kappa_t$, and the fact that the actions of $G$ and $\Gamma$ on $M$ commute, we have that \eqref{Eq: Quotient Formula Proof 4} equals
\begin{equation}\label{Eq: g-trace base and total}
\begin{aligned}
    \int_{M/\Gamma}&\sum_{y \in \Gamma}\psi(ym)\sum_{x \in \Gamma} \tr\big(hgh^{-1}x\kappa_t\big(hg^{-1}h^{-1}x^{-1}ym,ym\big)\big) \; d(\Gamma m)\\
    &= \int_{M/\Gamma}\sum_{y \in \Gamma}\psi(ym)\sum_{x \in \Gamma} \tr\big(hgh^{-1}\kappa_t\big(hg^{-1}h^{-1}m,y^{-1}xym\big)y^{-1}xy\big) \; d(\Gamma m)\\
    &= \int_{M/\Gamma}\sum_{y \in \Gamma}\psi(ym)\sum_{x' \in \Gamma} \tr\big(hgh^{-1}\kappa_t\big(hg^{-1}h^{-1}m,x'm\big)x'\big) \; d(\Gamma m)\\
    &= \int_{M/\Gamma}\sum_{y \in \Gamma}\psi(ym)\sum_{x' \in \Gamma} \tr\big(hgh^{-1}x'\kappa_t\big(hg^{-1}h^{-1}(x')^{-1}m,m\big)\big) \; d(\Gamma m).
\end{aligned}
\end{equation}
Furthermore, let $\tilde \kappa_t$ denote the kernel of $(-1)^FFe^{-t\Delta_{E/\Gamma}}$. Then for all $m \in M$, it follows that
\[
    \tilde \kappa_t(\Gamma m, \Gamma m) = \sum_{x \in \Gamma}x\kappa_t(x^{-1}m,m). 
\]
Choosing $\psi_{M/\Gamma}$ to be the cutoff function from \thref{Lem: Cutoff on Quotient}, we can write the right hand side of \eqref{Eq: g-trace base and total} as
\[
    \int_{M/\Gamma}\psi_{M/\Gamma}(\Gamma m)\tr\big(hgh^{-1}\tilde\kappa_t(hg^{-1}h^{-1}\Gamma m, \Gamma m)\big) \; d(\Gamma m).
\] 
Therefore, as the $g$-trace is independent of the choice of cutoff function, we find that \eqref{Eq: Quotient Formula Proof 4} equals $\Tr_g\big((-1)^FFe^{-t\Delta_{E/\Gamma}}\big)$.
This implies that, when $\re(\sigma) > 0$ is large enough, using Lemmas \ref{Lem: Quotient Formula Large t} and \ref{Lem: Quotient Formula Small t} together with the Fubini-Tonelli theorem, we obtain that
\begin{equation}\label{Eq: Quotient Formula Proof 3}
\begin{aligned}
    \sum_{(\gamma)}\left.\frac{d}{ds}\right|_{s=0}\frac{1}{\Gamma(s)}\int_0^\infty &t^{s-1}e^{-\sigma t}\Tr_{(g,\gamma)}\big((-1)^FFe^{-t\Delta_W}\big)\;dt\\
    &= \left.\frac{d}{ds}\right|_{s=0}\frac{1}{\Gamma(s)}\int_0^\infty t^{s-1}e^{-\sigma t}\Tr_g\big((-1)^FFe^{-t\Delta_{W/\Gamma}}\big)\;dt.
\end{aligned}
\end{equation}
Finally, note that the right hand side of \eqref{Eq: Quotient Formula Proof 3} extends meromorphically to $\C$ and is holomorphic at $\sigma = 0$ if, and only if, the right hand side does as well, proving the result. 
\end{proof}

The following result provides an example where the condition that \eqref{Eq: Quotient Formula Exponential Volume} has at most exponential volume growth is satisfied.

\begin{proposition}\thlabel{Prop: Exponential Volume Growth Semisimple Lie}
Let $G$ be a connected real semisimple Lie group and $g \in G$ a semisimple element. Then there exists a \v Svarc--Milnor function $l$ with respect to $g$ for the action of $G$ on itself via left-translations such that for all $r \ge 0$,
\begin{equation}
    X_r \coloneqq \{hZ \in G/Z : l(hgh^{-1}) \le r\}
\end{equation}
has at most exponential volume growth in $r$, i.e.\ there exist $C,c > 0$ such that $\vol(X_r) \le Ce^{cr}$.
\end{proposition}

\begin{proof}
Let $K$ be a maximal compact subgroup of $G$, and let $d$ be a left $G$-invariant and right $K$-invariant Riemannian metric on $G$. Set $l(x) = d(e,x)$, which is a \v Svarc--Milnor function for the action of $G$ on itself via left-translations. Note that as $d$ is $K$-invariant, $l_G$ is $K$-invariant. Then from (3.4.45) and (4.2.9) in \cite{HypoellipticLaplacianOrbitalIntegrals}
\begin{equation}\label{Eq: Exponential Volume Growth Proof}
    \vol(X_r) = \int_{G/Z}1_{X_r}(hZ)d(hZ) = \int_V 1_{\{Y \in V: l_G(e^Yge^{-Y})\le r\}}(X)f(X)dX,
\end{equation}
where $1_{S}$ denotes the indicator function on a set $S$, $V$ is a vector subspace of $\Lie(G)$ with $dX$ the Lebesgue measure, and finally $f(X)$ is a positive function on $V$ which is exponentially bounded in $X$, i.e.\ there exist $C_1, c_1>0$ such that $f(X) \le C_1 e^{ c_1 \norm{X}}$. Moreover, from Theorem 3.4.1 in \cite{HypoellipticLaplacianOrbitalIntegrals} there exist $a(g) \in \Lie(G)$ and $C_g > 0$ such that for all $X \in V$ with $\norm{X} \ge 1$, we have
\[
    l(e^Xge^{-X}) \ge \norm{a(g)} + C_g\norm{X}.
\]
Thus if $\norm{X} \ge 1$, then $l(e^Xge^{-X}) \le r$ implies that $\norm{X} \le (r-\norm{a(g)})/C_g$. Let $B$ denote the closed ball of radius $(r - \norm{a(g)})/C_g$ in $V$. Then there exists $C_2 > 0$ such that the right hand side of \eqref{Eq: Exponential Volume Growth Proof} is less than or equal to
\[
    \vol(B)C_1C_2e^{c_1(r - \norm{a(g)})/C_g}.
\]
As $\vol(B)$ is proportional to $\big((r - \norm{a(g)})/C_g\big)^{\dim V}$, it follows that there exist $C, c > 0$ such that
\[
    \vol(B)C_1C_2e^{c_1(r - \norm{a(g)})/C_g} \le Ce^{cr},
\]
proving the result.
\end{proof}

\subsection{Product Formula}\label{Subsec: Product Formula}

Here we generalise the product formula from \cite{hochs2022equivariant} for equivariant analytic torsion to twisted product actions. 

Let $G_1$, $G_2$ be locally compact unimodular groups and let $M_1$, $M_2$ be complete, oriented Riemannian manifolds. Suppose that $G_1 \times G_2$ acts on $M_1 \times M_2$ in a twisted product action, and $W_1$, $W_2$ are equivariant Hermitian bundles over $M_1,M_2$, respectively with the assumptions stated as in the start of Subsection \ref{Subsec: Fibration Formula} where $G_1 = G$ and $G_2 = \Gamma$. However, here we drop the assumptions that $G_2$ is countable, discrete, and finitely generated.

The twisted product action on $M_1 \times M_2$ is proper by \thref{Lem: Twisted Product Action Proper}. Thus we can consider the equivariant analytic torsion for the twisted product action of $G_1 \times G_2$ on $M_1 \times M_2$. The product formula for analytic torsion, see Proposition 5.8 in \cite{hochs2022equivariant}, does not apply in this case for two reasons. First is the fact that $G_2$ now also acts on $M_1$, and the second is that the action of $G_1 \times G_2$ on $M_1$ may not be proper. 

However, if $g_2 \in G_2$ is a central element, we obtain a relevant product formula generalising Proposition 5.8 in \cite{hochs2022equivariant}. Therefore, from now on, for $j = 1,2$ let $g_j \in G_j$ and denote by $Z_1$ the centraliser of $g_1$ in $G_1$. Finally suppose that $g_2 \in G_2$ is a central element. 

\begin{lemma}\thlabel{Lem: Cutoff Function Product}
For $j=1,2$ let $\psi_j$ be a cutoff function for the action of $G_j$ on $M_j$. Then $\psi_1 \otimes \psi_2$ is cutoff function for the twisted product action of $G_1 \times G_2$ on $M_1 \times M_2$.
\end{lemma}

\begin{proof}
As $\psi_j$ are cutoff functions, it follows that  
\begin{align*}
    \int_{G_1}\int_{G_2}&(\psi_1 \otimes \psi_2)\big((x_1,x_2) \cdot (m_1,m_2)\big)\; dx_2 \; dx_1\\
    &= \int_{G_2}\Big(\int_{G_1}\psi_1(x_1x_2m_1)dx_1\Big)\psi_2(x_2m_2)\;dx_2\\
    &= 1.
\end{align*}
So $\psi_1 \otimes \psi_2$ is a cutoff function for the twisted product action of $G_1 \times G_2$ on $M_1 \times M_2$. 
\end{proof}

\begin{lemma}\thlabel{Lem: Product Formula Lemma}
Let $\kappa_1 \in \Gamma^\infty(W_1 \boxtimes W_1^*)$ be a $G_1 \times G_2$-invariant kernel, and $\kappa_2 \in \Gamma^\infty(W_2 \boxtimes W_2^*)$ be a $G_2$-invariant kernel. If $g_2\kappa_1$ is a $g_1$-trace class operator, then $\kappa_1 \otimes \kappa_2$ is $(g_1,g_2)$-trace class and
\[
    \Tr_{(g_1,g_2)}(\kappa_1 \otimes \kappa_2) = \Tr_{g_1}(g_2\kappa_1)\Tr_{g_2}(\kappa_2).
\]
\end{lemma}

\begin{proof}
For $j=1,2$, let $\psi_j$ be a cutoff function for the action of $G_j$ on $M_j$. Then using \thref{Lem: Cutoff Function Product} we see that
\begin{equation}\label{Eq: Product Lemma Proof}
\begin{aligned}
    &\Tr_{(g_1,g_2)}(\kappa_1 \otimes \kappa_2) \\
    &= \int_{G_1/Z_1}\int_{M_1} \tilde \psi_1(h_1g_1h_1^{-1}m_1)\tr\big(h_1g_1h^{-1}g_2\kappa_1(h_1g_1^{-1}h_1^{-1}g_2m_1,m_1)\big)dm_1\; d(h_1Z_1)\\
    &\qquad \cdot \int_{M_2} \psi_2(g_2m_2) \tr\big(g_2 \kappa_2(g_2^{-1}m_2,m_2)\big)dm_2,
\end{aligned}
\end{equation}
where $\tilde \psi_1$ is the function defined by $\tilde \psi_1(m) = \psi_1(g_2m)$ which is a cut-off function for the action of $G_1$ on $M_1$. Here we have used the assumption that $g_2$ is central so that the conjugacy class of $g_2$ is a point, and hence so is $G_2/Z_2$. Thus, \eqref{Eq: Product Lemma Proof} equals $\Tr_{g_1}(g_2\kappa_1)\Tr_{g_2}(\kappa_2)$.
\end{proof}

The following theorem is a generalisation of Proposition 5.8 in \cite{hochs2022equivariant}, to twisted product actions. 

\begin{theorem}\thlabel{Thm: Product Formula}
Suppose $g_2$ is central in $G_2$, and that the orthogonal projection onto the $L^2$-kernel of $\Delta_{W_1}$, $P^{W_1}$, is $g_1$-trace class. If either
\begin{enumerate}[label=\emph{(\Roman*)}]
    \item $G_1/Z_1$ is compact and the Novikov--Shubin numbers $(\alpha_{e_j}^{p_j})_{W_j}$ for $\Delta_{W_j}$ are positive, or
    \item $\Tr_{g_1}(g_2P^{W_1}) = \Tr_{g_2}(P^{W_2}) = 0$,
\end{enumerate}
then for all $\sigma \in \C$ with $\re(\sigma) > 0$ such that \eqref{Eq: Torsion Def} converges, we have
\begin{equation}\label{Eq: Product Formula}
    T_{(g_1,g_2)}(\nabla^{W_1 \boxtimes W_2},\sigma) = T_{(g_1,g_2)}(\nabla^{W_1},\sigma)^{\chi_{g_2}(\nabla^{W_2})}T_{g_2}(\nabla^{W_2},\sigma)^{\chi_{(g_1,g_2)}(\nabla^{W_1})}.
\end{equation}
Further suppose that for $j=1,2$, condition (II) of \thref{Lem: Alternative Convergence Torsion} holds for $\nabla^{W_j}$, and that there exists \v Svarc--Milnor functions for the actions with respect to $g_j$. Suppose there exists a \v Svarc--Milnor function for the twisted product action of $G_1 \times G_2$ on $M_1 \times M_2$ with respect to $(g_1,g_2)$. Then $T_{(g_1,g_2)}(\nabla^{W_1 \boxtimes W_2})$, $T_{(g_1,g_2)}(\nabla^{W_1})$, and $T_{g_2}(\nabla^{W_2})$ are all well-defined, and
\begin{equation}\label{Eq: Product Formula 2}
    T_{(g_1,g_2)}(\nabla^{W_1 \boxtimes W_2}) = T_{(g_1,g_2)}(\nabla^{W_1})^{\chi_{g_2}(\nabla^{W_2})}T_{g_2}(\nabla^{W_2})^{\chi_{(g_1,g_2)}(\nabla^{W_1})}.
\end{equation}
\end{theorem}

\begin{proof}
This proof is a generalisation of the proof of Proposition 5.8 in \cite{hochs2022equivariant}.

As for all $t > 0$,
\[
    e^{-t\Delta^p_{W_1 \boxtimes W_2}} = \bigoplus_{p_1 + p_2 = p}e^{-t\Delta^{p_1}_{W_1}} \otimes e^{-t\Delta^{p_2}_{W_2}},
\]  
and denoting the relevant number operators by $F$ and $F_j$, we find
\begin{align*}
    \Tr_{(g_1,g_2)}&\big((-1)^FFe^{-t\Delta_{W_1 \boxtimes W_2}}\big) \\
    &= \Tr_{g_1}\big(g_2(-1)^{F_1}F_1e^{-t\Delta_{W_1}}\big)\chi_{g_2}(\nabla^{W_2}) + \Tr_{g_2}\big((-1)^{F_2}F_2e^{-t\Delta_{W_2}}\big) \chi_{(g_1,g_2)}(\nabla^{W_1})\numberthis \label{Eq: Product Formula Proof},
\end{align*}
following (5.9) in \cite{hochs2022equivariant}, and using \thref{Lem: Product Formula Lemma}.  

In case (I), Proposition 2.8 of \cite{hochs2022equivariant}, which allows us to estimate the $g_1$-trace in terms of the $e_1$-trace, implies
\begin{align*}
    \abs{\Tr_{g_1}\big((-1)^{F_1}F_1g_2(e^{-t\Delta_{W_1}} - P^{W_1})\big)} &\le \vol(G_1/Z_1)\norm{(-1)^{F_1}F_1g_2}\abs{\Tr_{e_1}(e^{-t\Delta_{W_1}} - P^{W_1})}\\
    &\le \vol(G_1/Z_1)\dim(M_1)\abs{\Tr_{e_1}(e^{-t\Delta_{W_1}} - P^{W_1})}
\end{align*}
and similarly,
\[
    \abs{\Tr_{g_2}\big((-1)^{F_2}F_2(e^{-t\Delta_{W_2}} - P^{W_2})\big)} \le \dim(M_2)\abs{\Tr_{e_2}(e^{-t\Delta_{W_2}} - P^{W_2})}.
\]
In both cases, the right hand side goes to zero as $t \to \infty$, from the assumption that $(\alpha^{p_j}_{e_j})_{W_j} > 0$ for all $p_j$. Using the analogous estimate on $M_1 \times M_2$, we see that taking the limit $t \to \infty$ in \eqref{Eq: Product Formula Proof} gives
\begin{equation}\label{Eq: Product Formula Proof Limit}
\begin{aligned}
    \Tr_{(g_1,g_2)}\big((-1)^FF&P^{W_1 \boxtimes W_2}\big)\\
    &= \Tr_{g_1}\big(g_2(-1)^{F_1}F_1P^{W_1}\big)\chi_{g_2}(\nabla^{W_2}) + \Tr_{g_2}\big((-1)^{F_2}F_2P^{W_2}\big) \chi_{(g_1,g_2)}(\nabla^{W_1}).
\end{aligned}
\end{equation}
Subtracting \eqref{Eq: Product Formula Proof Limit} from \eqref{Eq: Product Formula Proof} gives
\begin{align*}
    \Tr_{(g_1,g_2)}\big((-1)^FF(e^{-t\Delta_{W_1 \boxtimes W_2}} - &P^{W_1 \boxtimes W_2})\big)\\
    &= \Tr_{g_1}\big(g_2(-1)^{F_1}F_1(e^{-t\Delta_{W_1}} - P^{W_1})\big)\chi_{g_2}(\nabla^{W_2}) \\
    &\qquad + \Tr_{g_2}\big((-1)^{F_2}F_2(e^{-t\Delta_{W_2}} - P^{W_2})\big) \chi_{(g_1,g_2)}(\nabla^{W_1}).
\end{align*}
Therefore, it follows that for all $\sigma \in \C$ such that \eqref{Eq: Torsion Def} converges we have
\begin{equation}\label{Eq: Product Formula Proof 2}
\begin{aligned}
    -2\log T_{(g_1,g_2)}(\nabla^{W_1 \boxtimes W_2},\sigma) &= -2\chi_{g_2}(\nabla^{W_2})\log T_{(g_1,g_2)}(\nabla^{W_1},\sigma) \\
    &\qquad\qquad\qquad-2\chi_{(g_1,g_2)}(\nabla^{W_1})\log T_{g_2}(\nabla^{W_2},\sigma\big).
\end{aligned}
\end{equation}
Exponentiating \eqref{Eq: Product Formula Proof 2} gives \eqref{Eq: Product Formula}.

For case (II), as the Laplacians are non-negative,
\[
    \ker(\Delta_{W_1 \boxtimes W_2}) = \ker(\Delta_{W_1}) \otimes \ker(\Delta_{W_2}).
\]
Hence, by \thref{Lem: Product Formula Lemma}
\[
    \Tr_{(g_1,g_2)}(P^{W_1 \boxtimes W_2}_p) = \sum_{p_1 + p_2 = p}\Tr_{g_1}(g_2P^{W_1}_{p_1})\Tr_{g_2}(P^{W_2}_{p_2}) = 0.
\]
Thus \eqref{Eq: Product Formula Proof} implies \eqref{Eq: Product Formula Proof 2}, and therefore \eqref{Eq: Product Formula}.

If condition (II) of \thref{Lem: Alternative Convergence Torsion} holds for $\nabla^{W_j}$, and there are \v Svarc--Milnor functions for the respective actions, it follows by \thref{Lem: Alternative Convergence Torsion} that the torsions $T_{(g_1,g_2)}(\nabla^{W_1})$ and $T_{g_2}(\nabla^{W_2})$ are well-defined. From \eqref{Eq: Product Formula Proof}, we see that the second condition of \thref{Lem: Alternative Convergence Torsion} holds for $\nabla^{W_1 \boxtimes W_2}$. Again, as there exists a \v Svarc--Milnor function for the twisted product action of $G_1 \times G_2$ on $M_1 \times M_2$, \thref{Lem: Alternative Convergence Torsion} implies that $T_{(g_1,g_2)}(\nabla^{W_1 \boxtimes W_2})$ is well-defined. Thus taking meromorphic extensions and setting $\sigma = 0$ in \eqref{Eq: Product Formula Proof 2}, and then exponentiating gives \eqref{Eq: Product Formula 2}.
\end{proof}

\begin{proof}[Proof of \thref{Thm: New Torsion Formula}]
As $\Gamma$ is countable and discrete and acts properly and cocompactly on $M_2$, by the \v Svarc--Milnor lemma $\Gamma$ is finitely generated, and a word length metric, $l_\Gamma$, is a \v Svarc--Milnor function for the action of $\Gamma$ on $M_2$. As for all $\gamma \in \Gamma$ and $(m_1,m_2) \in M_1 \times M_2$
\begin{equation}\label{Eq: Fibration Formula Proof}
    d_{Y \times Z}\big((\gamma m_1, \gamma m_2), (m_1,m_2)\big) \ge d_Z(\gamma m_2, m_2),
\end{equation}
we see that $l_\Gamma$ is a \v Svarc--Milnor function for the action of $\Gamma$ on $M_1 \times M_2$. 

In case (I), as $G/Z$ is compact we can take $l_G \equiv 0$, the zero function, as a \v Svarc--Milnor function for the action of $G$ on $M_1$. Thus a similar inequality to \eqref{Eq: Fibration Formula Proof} shows that $l_\Gamma = l_G + l_\Gamma$ is \v Svarc--Milnor function for the action $G \times \Gamma$-action on $M_1 \times M_2$. Hence the assumptions on the \v Svarc--Milnor functions in \thref{Thm: Quotient Formula} hold. In case (II), the assumptions on the \v Svarc--Milnor functions are precisely the assumptions stated in \thref{Thm: Quotient Formula}.

Thus, in either case, using \thref{Thm: Quotient Formula} we have for $\sigma \in \C$ with $\re(\sigma) > 0$ large enough
\[
    T_g(\nabla^{E/\Gamma},\sigma) = \prod_{\gamma \in \Gamma}T_{(g,\gamma)}(\nabla^E,\sigma),
\]
where we used the assumption that $\Gamma$ is abelian to write the product over conjugacy classes as a product over the whole group. Further, as every $\gamma \in \Gamma$ is central, using the respective case in  \thref{Thm: Product Formula}, we find
\[
    T_{(g,\gamma)}(\nabla^E,\sigma) = T_{(g,\gamma)}(\nabla^{E_1},\sigma)^{\chi_{\gamma}(\nabla^{E_2})}T_{\gamma}(\nabla^{E_2},\sigma)^{\chi_{(g,\gamma)}(\nabla^{E_1})}.
\]
Finally, we see that $T_g(\nabla^{E/\Gamma})$ is well-defined if, and only if, the product
\[
    \prod_{\gamma \in \Gamma}T_{(g,\gamma)}(\nabla^{E_1},\sigma)^{\chi_{\gamma}(\nabla^{E_2})}T_{\gamma}(\nabla^{E_2},\sigma)^{\chi_{(g,\gamma)}(\nabla^{E_1})}
\]
converges to a holomorphic function, and has a meromorphic extension which is holomorphic at $\sigma = 0$.
\end{proof}

\section{The Ruelle \texorpdfstring{$\zeta$}{zeta}-Function for Suspension Flow}\label{Sec: Ruelle Suspension}

In this section we calculate the equivariant Ruelle dynamical $\zeta$-function for suspension flow. In Subsection \ref{Subsec: Suspension Ruelle} we calculate the equivariant $\zeta$-function under the assumption that certain maps have finite fibres, which is equivalent to $g$ lying in a compact subgroup of $G$ (see \thref{Cor: Compact Subgroup iff Finite Fibres}). This is done so that the equivariant Ruelle $\zeta$-function for suspension flow can be calculated from data on $Y$. Then in Subsection \ref{Subsec: Absolute Convergence}, we provide some conditions under which the equivariant Ruelle $\zeta$-function converges absolutely. Finally, in Subsection \ref{Subsec: Compact Group Suspension} we show the equivariant Ruelle $\zeta$-function for a compact group action can be computed from data on the cohomology of $Y$. 

We recall and reuse the notation defined previously in Subsection \ref{Subsec: Results}.

\subsection{The Equivariant Ruelle \texorpdfstring{$\zeta$}{zeta}-Function}\label{Subsec: Suspension Ruelle}

In this subsection we obtain a general expression for the equivariant Ruelle dynamical $\zeta$-function for suspension flow, \thref{Thm: Equivariant Ruelle Suspension}, under the condition that for $g \in G$ we have $g$ lies inside a compact subgroup of $G$, and the fix points of $g^{-1}T^n$ are isolated and non-degenerate for all $n \in \Z\setminus \{0\}$. In this subsection we drop the assumption that $T$ is an isometry, and only assume that $T$ is a diffeomorphism. 

The $g$-delocalised length spectrum $L_g(\vphi)$ for suspension flow is a subset of $\Z \setminus \{0\}$. To see this, take $n \in L_g(\vphi)$ so that $\vphi_n[y,t] = g[y,t]$ for some $[y,t] \in M$. Then $[y,t+n] = [gy,t]$ and so there exists a $k \in \Z$ such that $(T^ky,t + n - k) = (gy,t)$ which implies that $n = k \in \Z$.

\begin{lemma}\thlabel{Lem: Fixed Point Set Flow}
For all $n \in \Z \setminus \{0\}$
\[
    M^{g^{-1}\vphi_n} = \pi( Y^{g^{-1}T^n} \times \R).
\]
\end{lemma}

\begin{proof}
Let $[y,t] \in M^{g^{-1}\vphi_n}$. Then
\[
    [y,t] = g^{-1}\vphi_n[y,t] = [g^{-1}y,t+n].
\]
So there exists $k \in \Z$ such that $(y,t) = (g^{-1}T^ky, t + n - k)$, which implies $n = k$. Thus $(y,t) \in Y^{g^{-1}T^n} \times \R$. 

Conversely, suppose $(y,t) \in Y^{g^{-1}T^n} \times \R$, then 
\[
    g^{-1}\vphi_n[y,t] = [g^{-1}y, t+n] = [g^{-1}T^ny,t+n-n] = [y,t]. \qedhere
\]
\end{proof}

\begin{lemma}\thlabel{Lem: Suspension Flow Curves}
For every $n \in \Z \setminus \{0\}$, we have $n \in L_g(\vphi)$ if, and only if, $Y^{g^{-1}T^n}$ is non-empty. Moreover, every flow curve $\gamma \in \Gamma^g_n(\vphi)$ has a representative of the form $\gamma(t) = [y,t]$ for $y \in Y^{g^{-1}T^n}$.
\end{lemma}

\begin{proof}
The first statement follows as $n \in L_g(\vphi)$ if, and only if, $M^{g^{-1}\vphi_n} \neq \emptyset$, from which the result follows by \thref{Lem: Fixed Point Set Flow}. Let $\gamma$ be a $(g,n)$-periodic flow curve, and choose $y \in Y$ and $s \in \R$ such that $\gamma(0) = [y,s]$. Then as $\gamma(0) \in M^{g^{-1}\vphi_n}$, it follows from \thref{Lem: Fixed Point Set Flow} that $y \in Y^{g^{-1}T^n}$. As $\gamma(t - s) = [y,t]$, we see that $\gamma(t)$ and $\tilde \gamma(t) \coloneqq [y,t]$ define the same equivalence class in $\Gamma^g_n(\vphi)$.
\end{proof}

We now show that under the assumption that the fixed points of $g^{-1}T^n$ are isolated and nondegenerate for all $n \in \Z \setminus\{0\}$, the suspension flow is $g$-nondegenerate.

\begin{proposition}\thlabel{Prop: Suspension Flow NonDegen}
Suppose that for every $n \in \Z \setminus \{0\}$ we have $Y^{g^{-1}T^n}$ is either discrete or empty. Then for all $y \in Y^{g^{-1}T^n}$ we have 
\[
    \det \big( (1 - D_y(g^{-1}T^n))|_{T_yY}\big) = \det\big((1 - D_{[y,t]}(g^{-1}\vphi_n))|_{T_{[y,t]}M/T_{[y,t]}M^{g^{-1}\vphi_n}}\big).
\]
In particular, suspension flow is $g$-nondegenerate if, and only if, $\det \big( (1 - D_y(g^{-1}T^n))|_{T_yY}\big) \neq 0$ for all $y \in Y^{g^{-1}T^n}$ and $n \in \Z \setminus \{0\}$.
\end{proposition}

\begin{proof}
First we note that as $\Z$ is discrete, $T_{[y,t]}M = D\pi_{(y,t)}(T_yY \times \R)$. Now $Y^{g^{-1}T^n}$ being discrete implies that $\dim(Y^{g^{-1}T^n} \times \R) = 1$, and so $\dim M^{g^{-1}\vphi_n} = \dim D\pi\big(Y^{g^{-1}T^n} \times \R) = 1$ as $\pi$ is a local diffeomorphism. Let $[y,t] = m \in M^{g^{-1}\vphi_n}$. If $u$ is the vector field generating the suspension flow, we know from equivariance of $\vphi_n$ that $\R u(m') \subseteq T_{m'} M^{g^{-1}\vphi_n}$ for all $m' \in M^{g^{-1}\vphi_n}$, and so $\R u(m) = T_mM^{g^{-1}\vphi_n}$ due to dimensional reasons. Moreover, using \thref{Lem: Fixed Point Set Flow}, we deduce that
\[
    TM^{g^{-1}\vphi_n} = D\pi\big(TY^{g^{-1}T^n} \times \R) = D\pi(\{0\} \times \R),
\]
as $Y^{g^{-1}T^n}$ is discrete. Therefore, as $\R u(m) = T_mM^{g^{-1}\vphi_n}$, identifying $T_mM/\R u(m)$ with the normal bundle of $M^{g^{-1}\vphi_n}$ in $M$ we obtain
\[
    T_mM/\R u(m) = D\pi_{(y,t)}(T_yY \times \R)/D\pi_{(y,t)}(\{0\} \times \R) \cong D\pi_{(y,t)}(T_yY \times \{0\}).
\]
As functions from $Y \times \R \to M$, we have the equality
\[
        \pi \circ (g^{-1}T^n \times \identity) = g^{-1}\vphi_n \circ \pi    
\]
and as $D\pi_{(y,t)} \colon T_yY \times \{0\} \to T_mM/\R u(m)$ is a linear isomorphism,
\[
    \det\big((1 - D_m(g^{-1}\vphi_n))|_{T_mM/\R u(m)}\big) = \det \big( (1 - D_y(g^{-1}T^n))|_{T_yY}\big). \qedhere
\]
\end{proof}

\begin{remark}
If $g$ and $T$ preserve some Riemannian metric on $Y$, and if $Y^{g^{-1}T^n}$ is either empty or discrete for all $n \in \Z \setminus \{0\}$, then $1 - D_y(g^{-1}T^n)$ is invertible and suspension flow is $g$-nondegenerate by \thref{Prop: Suspension Flow NonDegen}. This is because using the Riemannian exponential map, one can show that the $+1$ eigenspace of $D_y(g^{-1}T^n)$ is $T_yY^{g^{-1}T^n}$ (see page 187 of \cite{BGV} for details). However, $T_yY^{g^{-1}T^n} = \{0\}$ as $Y^{g^{-1}T^n}$ is discrete, and hence 1 cannot be an eigenvalue of $D_y(g^{-1}T^n)$.
\end{remark}

From now on assume that $Y^{g^{-1}T^n}$ is either discrete or empty for all $n \in \Z \setminus \{0\}$, and that for all $y \in Y^{g^{-1}T^n}$ we have $\det \big( (1 - D_y(g^{-1}T^n))|_{T_yY}\big) \neq 0$. Then by \thref{Prop: Suspension Flow NonDegen} the suspension flow is $g$-nondegenerate.

Now under the assumption $Y^{g^{-1}T^n}$ is either empty or discrete for all $n \in \Z \setminus \{0\}$, it follows from \thref{Lem: Fixed Point Set Flow} that the connected components of $M^{g^{-1}\vphi_n}$ are of the form
\[
    \pi\big( \{y \} \times \R\big)
\]
for $y \in Y^{g^{-1}T^n}$ when $n \in L_g(\vphi)$. However, multiple such $y$ may give rise to the same connected component. Therefore, we have a surjective map
\begin{equation}\label{Eq: Connected Component Map}
\begin{aligned}
    f_n \colon Y^{g^{-1}T^n} &\to \{\text{connected components of }M^{g^{-1}\vphi_n}\}\\
    y &\mapsto \pi\big(\{y\} \times \R\big)
\end{aligned}
\end{equation}
for each $n \in L_g(\vphi)$, and we wish to describe the fibres of $f_n$. In particular, we wish to know under what conditions the maps $f_n$ have finite fibres. This will be done using \thref{Lem: Fibres of fn} and \thref{Lem: Existence Prime Period}.

\begin{lemma}\thlabel{Lem: Fibres of fn}
Let $n \in L_g(\vphi)$ and take $y_1,y_2 \in Y^{g^{-1}T^n}$. Then $f_n(y_1) = f_n(y_2)$ if, and only if, $T^ky_1 = y_2$ for some $k \in \Z$. 
\end{lemma}

\begin{proof}
Suppose that $f_n(y_1) = f_n(y_2)$ for $y_1,y_2 \in Y^{g^{-1}T^n}$. Then there exist $t_1,t_2 \in \R$ such that $[y_1,t_1] = [y_2,t_2]$. This implies there exists some $k \in \Z$ such that $(T^ky_1, t_1 - k) = (y_2, t_2)$, and so $T^ky_1 = y_2$. 

Conversely, suppose that for $y_1,y_2 \in Y^{g^{-1}T^n}$ that there is some $k \in \Z$ such that $T^ky_1 = y_2$. Let $t \in \R$ and consider $[y_1,t] \in \pi\big(\{y_1\} \times \R\big)$. Then
\[
    [y_1,t] = [T^ky_1,t-k] = [y_2,t-k] \in \pi\big(\{y_2\} \times \R\big),
\]
so that $\pi\big(\{y_1\} \times \R\big) \subseteq \pi\big(\{y_2\} \times \R\big)$. Swapping $y_1$ and $y_2$ proves the reverse inclusion. Hence $\pi\big(\{y_1\} \times \R\big) = \pi\big(\{y_2\} \times \R\big)$ or, equivalently, $f_n(y_1) = f_n(y_2)$.
\end{proof}

Using \thref{Lem: Fibres of fn} we can say when the fibres of $f_n$ are finite. Recall that if $T^j(y) = y$ for some $j \in \Z$, then we set $p(y) = \min\{k \in \N: T^ky =y\}$ to be the primitive period of $y$.

\begin{lemma}\thlabel{Lem: Existence Prime Period}
For $n \in L_g(\vphi)$ the fibres of map $f_n$ are finite if, and only if, for every $y \in Y^{g^{-1}T^n}$ there exists $l \in \Z \setminus \{0\}$ such that $T^ly = y$. Moreover, in that case $\abs{f_n^{-1}\big(f_n(y)\big)} = p(y)$.
\end{lemma}

\begin{proof}
Suppose first that $f_n$ has finite fibres. Let $y \in Y^{g^{-1}T^n}$. Then by \thref{Lem: Fibres of fn} $T^ky$ lies in same fibre of $f_n$ as $y$ for all $k \in \Z$. As the fibres are finite, there must exists some $l \in \Z \setminus \{0\}$ such that $T^ly = y$. 

Conversely suppose there exists $l \in \Z \setminus \{0\}$ such that $T^ly = y$. Then, as every $d \in \Z$ can be written as $d = \ell p(y) + k$ for some $\ell \in \Z$ and $k \in \{0,1,\dots,p(y)-1\}$, we have
\begin{equation}\label{Eq: What is fibre of fn}
\begin{aligned}
    f_n^{-1}\big(f_n(y)\big) &= \{T^dy : d \in \Z\}\\
    &= \{T^{\ell p(y) + k}y : l \in \Z, k \in \{0,1,\dots,p(y)-1\}\}\\
    &= \{T^ky : k \in \{0,1,\dots,p(y)-1\}\},
\end{aligned}
\end{equation}
which is a finite set of $p(y)$ elements.
\end{proof}

\begin{proposition}\thlabel{Prop: Finite Fibres iff Periodic}
The fibres of $f_n$ are finite if, and only if, every flow curve $\gamma \in \Gamma^g_n(\vphi)$ is periodic. 
\end{proposition}

\begin{proof}
From \thref{Lem: Existence Prime Period} the fibres of $f_n$ are finite if, and only if, every $y \in Y^{g^{-1}T^n}$ is fixed by some power of $T$. Let $\gamma \in \Gamma^g_n(\vphi)$. By \thref{Lem: Suspension Flow Curves}, it is of the form $\gamma(t) = [y,t]$ for $y \in Y^{g^{-1}T^n}$.

If there exists $k \in \Z \setminus \{0\}$ such that $T^ky = y$, then
\[
    \gamma(k) = [y,k] = [T^ky,0] = [y,0] = \gamma(0) 
\]
and $\gamma$ is periodic. Conversely, if $\gamma$ is periodic, then there exists $s \in \R \setminus \{0\}$ such that $\gamma(s) = \gamma(0)$. This implies $[y,s] = [y,0]$, and so there is a $k \in \Z$ such that $(T^ky, s-k) = (y,0)$ which implies $T^ky = y$. 
\end{proof}

Combining \thref{Prop: Finite Fibres iff Periodic} and \thref{Lem: Useful Results Flow Curves}(II), we obtain a condition for the fibres to be finite, independent of both $Y$ and $M$. 

\begin{corollary}\thlabel{Cor: Compact Subgroup iff Finite Fibres}
The fibres of $f_n$ are finite if, and only if, $g$ lies inside a compact subgroup of $G$. 
\end{corollary}

Hence from now on we will assume that $g$ lies in a compact subgroup of $G$.

\begin{theorem}\thlabel{Thm: Equivariant Ruelle Suspension}
The Ruelle dynamical $\zeta$-function for equivariant suspension flow converges at $\sigma \in \C$ if, and only if, the limit on the right hand side of \eqref{Eq: Equivariant Ruelle Suspension} converges. Moreover, at such a $\sigma$, we have for all cutoff functions $\psi_M$ that
\begin{equation}\label{Eq: Equivariant Ruelle Suspension}
\begin{aligned}
      \log R^g_{\vphi, \nabla^E}(\sigma) &= \lim_{r \to \infty}\int_{G/Z}\sum_{n \in \Z \setminus \{0\} \cap [-r,r]}\frac{e^{-|n|\sigma}}{|n|}\sum_{y \in Y^{g^{-1}T^n}}\sgn\det\big(1 - D_y(g^{-1}T^n)\big)\\
     &\qquad \qquad \qquad \qquad \qquad\qquad\cdot \tr(gT_{E_1}^{-n}|_{(E_1)_y})\frac{1}{p(y)}\int_0^{p(y)}\psi_M[hy,s]ds\;d(hZ).
\end{aligned}
\end{equation} 
\end{theorem}

\begin{proof}
Consider $n \in L_g(\vphi) \subseteq \Z \setminus \{0\}$. Consider $\gamma \in \Gamma^g_n(\vphi)$, then $\gamma(0) = [y,0]$ with $y \in Y^{g^{-1}T^n}$ by \thref{Lem: Suspension Flow Curves}. As the suspension flow is $g$-nondegenerate it follows that $T_{[y,0]}M/\R\gamma'(0) = T_{[y,0]}M/T_{[y,0]}M^{g^{-1}\vphi_n}$, and so using \thref{Prop: Suspension Flow NonDegen} we have
\begin{equation}\label{Eq: Suspension Ruelle Proof 1}
\begin{aligned}
    \sgn\big(\det(1 - P^g_\gamma)\big) &= \sgn\big(\det(1 - D_{[y,0]}(g^{-1}\vphi_n)|_{T_{[y,0]}M/T_{[y,0]}M^{g^{-1}\vphi_n}}\big) \\
    &= \sgn\big(\det(1 - D_y(g^{-1}T^n))\big).
\end{aligned}
\end{equation}
To calculate the operator $\rho_g(\gamma)$, we first need to calculate the parallel transport $\tau_{\nabla^E}$ in $E$ along the curve $\gamma$ from $t=0$ to $t=-n$. As $\nabla^E$ is obtained by restricting the connection $\nabla^{E_1} \otimes 1 + 1 \otimes d$ to the $\Z$-invariant sections, parallel transport along the curve $\delta(t) = [y,t]$ over the interval $[0,s]$ is given by $\tau_{\nabla^E}\big([v,0]\big) = [v,s]$ for $v \in (E_1)_y$. Therefore the parallel transport along $\gamma$ is given by $[v,-n] = [T_{E_1}^{-n}v,0]$ for $v \in (E_1)_y$, and so we find that
\begin{equation}\label{Eq: Suspension Ruelle Proof 2}
    \tr\big(\rho_g(\gamma)\big) = \tr\big((gT_{E_1}^{-n})|_{({E_1})_y}\big).
\end{equation}
Finally note that $\gamma$ is periodic by \thref{Lem: Useful Results Flow Curves}(II), and that the primitive period of the flow curve is $p(y)$. Thus $I_\gamma$ can be chosen to be $[0,p(y)]$. This implies
\begin{equation}\label{Eq: Suspension Ruelle Proof 4}
   \int_{I_\gamma}\psi_M\big(h\gamma(s)\big)\; ds = \int_0^{p(y)}\psi_M\big([hy,s]\big)\;ds 
\end{equation}
for all $h \in G$. Moreover, notice that \eqref{Eq: Suspension Ruelle Proof 1}, \eqref{Eq: Suspension Ruelle Proof 2}, and \eqref{Eq: Suspension Ruelle Proof 4} all remain unchanged if we replace $y$ with $T^ky$ for $k \in \Z$, and hence do not depend the choice of $y$ in a fibre of $f_n$. These fibres are finite and of size $p(y)$ by \thref{Prop: Finite Fibres iff Periodic}, \thref{Lem: Existence Prime Period}, and \thref{Cor: Compact Subgroup iff Finite Fibres}. Hence it follows that
\begin{equation}\label{Eq: Suspension Ruelle Proof 3}
\begin{aligned}
    &\sum_{\gamma \in \Gamma^{g}_n(\vphi)}\sgn\big(\det(1 - P^{g}_\gamma)\big)\tr\big(\rho_{g}(\vphi)\big)\int_{I_\gamma}\psi_M\big(\gamma(s)\big)\;ds\\
    &\quad= \sum_{y \in Y^{g^{-1}T^n}}\frac{1}{p(y)}\sgn\big(\det(1 - D_y(g^{-1}T^n))\big)\tr\big((gT_{E_1}^{-n})|_{({E_1})_y}\big) \int_0^{p(y)}\psi_M\big([hy,s]\big)\;ds.\\
\end{aligned}
\end{equation}
Substituting \eqref{Eq: Suspension Ruelle Proof 3} into \eqref{Eq: Equivariant Ruelle}, we obtain \eqref{Eq: Equivariant Ruelle Suspension}.
\end{proof}

\subsection{Absolute Convergence of \texorpdfstring{$R^g_{\vphi,\nabla^E}$}{R g phi}}\label{Subsec: Absolute Convergence}

While \thref{Thm: Equivariant Ruelle Suspension} provides a formula for the Ruelle $\zeta$-function for suspension flow, it does not say when the Ruelle $\zeta$-function (or equivalently, the integral in \eqref{Eq: Equivariant Ruelle Suspension}) converges. If $G$ is compact, then $G/Z$ and $M$ are compact, the fixed point sets $Y^{g^{-1}T^n}$ are finite for all $n \in \Z \setminus \{0\}$, and so using the Atiyah--Bott Lefschetz fixed-point theorem $R^g_{\vphi, \nabla^E}$ converges to a holomorphic function on the set of $\sigma \in \C$ with $\re(\sigma) > 0$ large enough. See Section \ref{Subsec: Compact Group Suspension} for more details. 

For the case when $G$ is non-compact it becomes harder to say for which $\sigma \in \C$ the expression \eqref{Eq: Equivariant Ruelle Suspension} converges. However, as $g$ lies in a compact subgroup of $G$, we can use the following lemma and corollaries to provide some conditions on convergence. 

For the rest of this section, we choose the cutoff function $\psi_M$ to be the cutoff function for the action of $G$ on $M$ defined by 
\begin{equation}\label{Eq: Cutoff on M from Y and R}
    \psi_M[y,t] \coloneqq \sum_{k \in \Z}\psi_Y(T^ky)\psi_\R(t-k).
\end{equation}
where $\psi_Y$ and $\psi_\R$ are cutoff functions for the actions of $G$ on $Y$ and $\Z$ on $\R$, respectively. The fact that $\psi_M$ is a cutoff function for the action of $G$ on $M$ follows from \thref{Lem: Cutoff on Quotient} and \thref{Lem: Cutoff Function Product}.

\begin{lemma}\thlabel{Lem: Convergence Ruelle Cpt Subgroup}
Suppose that $A \subseteq G$ is a compact subset. Then for all $n \in \Z \setminus \{0\}$ the set
\begin{equation}\label{Eq: y in support finite}
    K_{g^{-1}T^n} \coloneqq \{y \in Y^{g^{-1}T^n} : \emph{there exist $a \in A$, $s \in \R$ such that $\psi_M[ay,s] \neq 0$}\}
\end{equation}
is finite.
\end{lemma}

\begin{proof}
Let $y \in Y^{g^{-1}T^n}$, $a \in A$, and $s \in \R$ be such that $\psi_M[ay,s] \neq 0$. We note that by replacing $A$ by $\tilde A = A \cup A^{-1}$, we may assume without a loss of generality that $A$ is closed under inversion. From \eqref{Eq: Cutoff on M from Y and R} we see that there exists $k \in \Z$ such that $\psi_Y(T^kay) \neq 0$. This is the case if, and only if, there exists $\ell \in \{0,1,\dots, |n|-1\}$ and $r \in \Z$ such that $\psi_Y(g^rT^\ell ay) \neq 0$. This implies that $g^ray \in \bigcup_{l=0}^{|n|-1} T^{-\ell}\big(\supp \psi_Y\big)$ which is a compact set. This further implies that $y \in A\overline{g^\Z}\bigcup_{l=0}^{|n|-1} T^{-\ell}\big(\supp \psi_Y\big)$. As $g$ lies in a compact subgroup, $\overline{g^\Z}$ is compact, and hence so is $A\overline{g^\Z}\bigcup_{l=0}^{|n|-1} T^{-\ell}\big(\supp \psi_Y\big)$. As $Y^{g^{-1}T^n}$ is closed and discrete, it follows that $Y^{g^{-1}T^n} \cap \overline{g^\Z}\bigcup_{l=0}^{|n|-1} T^{-\ell}\big(\supp \psi_Y\big)$ is a finite set. Moreover, it contains \eqref{Eq: y in support finite} proving the result.
\end{proof}

\begin{lemma}\thlabel{Lem: Convergence Ruelle Z Compact}
Suppose that there exists a compact subset $B \subseteq Y$ such that $Y^{g^{-1}T^n} \subseteq B$ for all $n \in \Z \setminus \{0\}$. Then $Z$ is compact, and 
\begin{equation}\label{Eq: Convergence Ruelle Z Compact}
    \{h \in G : \emph{there exist $n \in \Z \setminus \{0\}$, $y \in Y^{g^{-1}T^n}$, and $s \in \R$ such that $\psi_M[hy,s] \neq 0$}\}
\end{equation}
has compact closure.
\end{lemma}

\begin{proof}
As $Y^{g^{-1}T^n}$ is closed and contained in a compact set, it is compact. As $Z$ acts properly on $Y^{g^{-1}T^n}$, it follows that $Z$ itself must be compact. For the second statement, take $h \in G$. Let $n \in \Z \setminus \{0\}$, $y \in Y^{g^{-1}T^n}$, and $s \in \R$ be such that $\psi_M[hy,s] \neq 0$. Then we also have $y \in B$. Moreover, following the proof of \thref{Lem: Convergence Ruelle Cpt Subgroup}, there exists $\ell \in \{0,1,\dots,|n|-1\}$, $r \in \Z$ such that $g^rhT^\ell y \in \supp \psi_Y$. So $g^rhy \in K \coloneqq \bigcup_{\ell=0}^{|n|-1}T^{-\ell}\big(\supp \psi_Y\big)$, which is compact. Define $\tilde B \coloneqq \overline{g^\Z}B$, $\tilde K \coloneqq \overline{g^\Z}K$. Then $\tilde B$, $\tilde K$ are compact sets as $\overline{g^\Z}$ is compact. Hence
\[
    h \in \{x \in G : \tilde B \cap x\tilde K \neq \emptyset\},
\]
but this set is compact due to the properness of the action of $G$ on $Y$. Thus \eqref{Eq: Convergence Ruelle Z Compact} has compact closure.
\end{proof}

For the rest of this section, we assume that $Y$ has a Riemannian metric and $E_1$ has a Hermitian metric. Suppose the metric on $Y$ is preserved by the maps $g$ and $T$, and the Hermitian metric on $E_1$ is preserved by the bundle maps $g$ and $T_{E_1}$.

\begin{proposition}\thlabel{Prop: Convergence of Ruelle for G/Z Compact}
Suppose that $G/Z$ is compact. Suppose also that there exist $C,c > 0$ such that the set $K_{g^{-1}T^n}$ defined in \eqref{Eq: y in support finite} for $A \subseteq G$ a compact set of representatives of $G/Z$, satisfies $\abs{ K_{g^{-1}T^n}} \le Ce^{c\abs{n}}$ for all $n \in \Z \setminus \{0\}$. Then the equivariant Ruelle $\zeta$-function converges absolutely for $\sigma$ with $\re(\sigma) > c$, and is equal to
\begin{equation}\label{Eq: Ruelle Suspension Useful}
\begin{aligned}
    R^g_{\vphi, \nabla^E}(\sigma) &= \exp\Big(\sum_{n \in \Z \setminus \{0\}}\frac{e^{-|n|\sigma}}{|n|}\sum_{y \in Y^{g^{-1}T^n}}\sgn\det\big(1 - D_y(g^{-1}T^n\big)\tr(gT_{E_1}^{-n}|_{(E_1)_y})\\
     &\qquad \qquad \qquad \qquad \qquad\qquad\qquad\qquad\cdot \frac{1}{p(y)}\int_{G/Z}\int_0^{p(y)}\psi_M[hy,s]ds\;d(hZ)\Big).
\end{aligned}
\end{equation}
\end{proposition}

\begin{proof}
For ease of notation, let 
\begin{equation}\label{Eq: Absolute Ruelle Proof a(y,n)}
    a(y,n) = \sgn\big(\det(1 - D_y(g^{-1}T^n))\big)\tr\big((gT_{E_1}^{-n})|_{(E_1)_y}\big)
\end{equation}
for all $n \in \Z \setminus \{0\}$ and $y \in Y^{g^{-1}T^n}$. Note that $\abs{a(y,n)} \le \rank(E_1)$ as $g$ and $T$ preserve the metric on $E_1$. So
\begin{equation}\label{Eq: Ruelle convergence bounds}
\begin{aligned}
    \sum_{y \in Y^{g^{-1}T^n}}\abs{\frac{a(y,n)}{p(y)}\int_0^{p(y)}\psi_M\big([hy,s]\big)\;ds} &= \sum_{y \in K_{g^{-1}T^n}}\abs{\frac{a(y,n)}{p(y)}\int_0^{p(y)}\psi_M\big([hy,s]\big)\;ds}\\
    & \le \sum_{y \in K_{g^{-1}T^n}}\rank(E_1)\norm{\psi_M}_\infty\\
    & \le C\norm{\psi_M}_\infty \rank(E_1)e^{c\abs{n}}.
\end{aligned}
\end{equation}
Thus 
\begin{align*}
    &\int_{G/Z}\sum_{n \in \Z \setminus \{0\}}\sum_{y \in Y^{g^{-1}T^n}}\abs{\frac{e^{-|n|\sigma}}{|n|}\frac{a(y,n)}{p(y)}\int_0^{p(y)}\psi_M\big([hy,s]\big)\;ds}\;d(hZ)\\
    & \qquad \qquad \qquad \qquad \le C\norm{\psi_M}_\infty \rank(E_1)\vol(G/Z) \sum_{n \in \Z \setminus \{0\}}\frac{e^{-(\re\sigma - c)\abs{n}}}{\abs{n}},
\end{align*}
and the right hand side converges absolutely for $\sigma \in \C$ with $\re(\sigma) > c$. Thus on the set of $\sigma$ such that $\re(\sigma) > c$, using Fubini's theorem we have
\begin{align*}
    \log R^g_{\vphi, \nabla^E}(\sigma) &= \lim_{r \to \infty}\int_{G/Z}\sum_{n \in \Z \setminus \{0\} \cap [-r,r]}\frac{e^{-|n|\sigma}}{|n|}\sum_{y \in Y^{g^{-1}T^n}}\sgn\det\big(1 - D_y(g^{-1}T^n\big)\tr(gT_{E_1}^{-n}|_{(E_1)_y})\\
     &\qquad \qquad \qquad \qquad \qquad\qquad\qquad\qquad\cdot \frac{1}{p(y)}\int_0^{p(y)}\psi_M[hy,s]ds\;d(hZ)\\
     &= \sum_{n \in \Z \setminus \{0\}}\frac{e^{-|n|\sigma}}{|n|}\sum_{y \in Y^{g^{-1}T^n}}\sgn\det\big(1 - D_y(g^{-1}T^n\big)\tr(gT_{E_1}^{-n}|_{(E_1)_y})\\
     &\qquad \qquad \qquad \qquad \qquad\qquad\qquad\qquad\cdot \frac{1}{p(y)}\int_{G/Z}\int_0^{p(y)}\psi_M[hy,s]ds\;d(hZ)\qedhere
\end{align*}
\end{proof}

\begin{proposition}\thlabel{Prop: Convergence of Ruelle for Z Compact}
Suppose that there exists a compact subset $B \subseteq Y$ such that $Y^{g^{-1}T^n} \subseteq B$ for all $n \in \Z\setminus\{0\}$. Further suppose that there exist $C,c > 0$ such that $\abs{Y^{g^{-1}T^n}} \le Ce^{c\abs{n}}$ for all $n \in \Z \setminus \{0\}$. Then the equivariant Ruelle $\zeta$-function converges absolutely for $\sigma$ with $\re(\sigma) > c$, and is equal to
\begin{equation}\label{Eq: Ruelle Suspension Z Compact}
    R^g_{\vphi, \nabla^E}(\sigma) = \exp\Big(\sum_{n \in \Z \setminus \{0\}}\frac{e^{-|n|\sigma}}{|n|}\sum_{y \in Y^{g^{-1}T^n}}\sgn\det\big(1 - D_y(g^{-1}T^n\big)\tr(gT_{E_1}^{-n}|_{(E_1)_y})\Big)
\end{equation}
\end{proposition}

\begin{proof}
Again, let $a(y,n)$ be defined by the expression \eqref{Eq: Absolute Ruelle Proof a(y,n)} for all $y \in Y^{g^{-1}T^n}$ and $n \in \Z \setminus \{0\}$. We note that the inequality \eqref{Eq: Ruelle convergence bounds} still holds, thus letting $X$ be the set defined in \eqref{Eq: Convergence Ruelle Z Compact}, we find that by \thref{Lem: Convergence Ruelle Z Compact} and the fact that $Z$ is compact, that
\begin{align*}
    &\int_{G}\sum_{n \in \Z \setminus \{0\}}\sum_{y \in Y^{g^{-1}T^n}}\abs{\frac{e^{-|n|\sigma}}{|n|}\frac{a(y,n)}{p(y)}\int_0^{p(y)}\psi_M\big([hy,s]\big)\;ds}\;d(hZ)\\
    &\qquad\qquad \qquad  = \int_{\bar X}\sum_{n \in \Z \setminus \{0\}}\sum_{y \in Y^{g^{-1}T^n}}\abs{\frac{e^{-|n|\sigma}}{|n|}\frac{a(y,n)}{p(y)}\int_0^{p(y)}\psi_M\big([hy,s]\big)\;ds}\;d(hZ)\\
    & \qquad \qquad \qquad  \le C\norm{\psi_M}_\infty \rank(E_1)\vol(\bar X) \sum_{n \in \Z \setminus \{0\}}\frac{e^{-(\re\sigma - c)\abs{n}}}{\abs{n}}.
\end{align*}
We see this converges absolutely for $\sigma \in \C$ with $\re(\sigma) > c$. Recall that as $Z$ is compact, using \thref{Rmk: Z Compact Replace Integral} we can replace the integral over $G/Z$ with an integral over $G$ in the definition of the Ruelle $\zeta$-function. Thus, on the set of $\sigma \in \C$ with $\re(\sigma) > c$, using Fubini's theorem we find
\begin{align*}
    \log R^g_{\vphi, \nabla^E}(\sigma) &= \lim_{r \to \infty}\int_{G}\sum_{n \in \Z \setminus \{0\} \cap [-r,r]}\frac{e^{-|n|\sigma}}{|n|}\sum_{y \in Y^{g^{-1}T^n}}\sgn\det\big(1 - D_y(g^{-1}T^n\big)\tr(gT_{E_1}^{-n}|_{(E_1)_y})\\
     &\qquad \qquad \qquad \qquad \qquad\qquad\qquad\qquad\cdot \frac{1}{p(y)}\int_0^{p(y)}\psi_M[hy,s]ds\;dh\\
     &= \sum_{n \in \Z \setminus \{0\}}\frac{e^{-|n|\sigma}}{|n|}\sum_{y \in Y^{g^{-1}T^n}}\sgn\det\big(1 - D_y(g^{-1}T^n\big)\tr(gT_{E_1}^{-n}|_{(E_1)_y})\\
     &\qquad \qquad \qquad \qquad \qquad\qquad\qquad\qquad\cdot \frac{1}{p(y)}\int_0^{p(y)}\int_{G}\psi_M[hy,s]dh\;ds\\
     &= \sum_{n \in \Z \setminus \{0\}}\frac{e^{-|n|\sigma}}{|n|}\sum_{y \in Y^{g^{-1}T^n}}\sgn\det\big(1 - D_y(g^{-1}T^n\big)\tr(gT_{E_1}^{-n}|_{(E_1)_y}),
\end{align*}
with the last equality following from the fact that $\psi_M$ is a cutoff function for the action of $G$ on $M$. 
\end{proof}

The following result provides an example of \thref{Prop: Convergence of Ruelle for G/Z Compact}. To state it, we recall that a metric space $X$ is \textbf{uniformly discrete} if there exists an $\veps > 0$ such that $B_\veps(x) \cap X = \{x\}$ for all $x \in X$.

\begin{corollary}
Suppose that $Y$ has a $G$-invariant Riemannian metric. Further suppose that for all $n \in \Z \setminus \{0\}$, $Y^{g^{-1}T^n}$ is contained in some fixed uniformly discrete subset $X$. If $g$ lies in a compact subgroup of $G$ and $G/Z$ is compact, then $R^g_{\vphi,\nabla^E}$ converges absolutely on the set of $\sigma$ with $\re(\sigma) > 0$. 
\end{corollary}

\begin{proof}
Let $\veps > 0$ be such that $B_\veps(x) \cap X = \{x\}$ for all $x \in X$. By replacing $\veps$ with $\veps/2$, we may assume without loss of generality that the balls $B_\veps(x)$ are disjoint in $Y$. As $G$ acts cocompactly on $Y$, we can write $Y = GK$ for some compact subset $K$ of $Y$. As $K$ is compact, there exists some $a > 0$ such that $\vol(B_\veps(k)) > a$ for all $k \in K$, and as $G$ acts isometrically this implies that $\vol(B_\veps(y)) > a$ for all $y \in Y$. 

Let $B_\veps(\supp\psi_Y) = \{y' \in Y : d(y',\supp\psi_Y) < \veps\}$ be the points of $Y$ within $\veps$ distance to a point in $\supp \psi_Y$, and set $B = \overline{B_\veps(\supp\psi_Y)}$. Then $B$ is a compact subset containing $\supp \psi_Y$. Let $A \subseteq G$ be a compact subset such that $G/Z = \{aZ : a \in A\}$, and set 
\begin{align*}
    A_1 &\coloneqq A\overline{g^\Z} \cdot \bigcup_{l=0}^{\abs{n}-1}T^{-l}(\supp \psi_Y),\\
    A_2 &\coloneqq A\overline{g^\Z} \cdot \bigcup_{l=0}^{\abs{n}-1}T^{-l}(B).
\end{align*}
Then $A_1$, $A_2$ are compact subsets of $Y$ with the property $A_1 \subseteq A_2$, and for all $x \in X \cap A_1$, $B_\veps(x) \subseteq A_2$. Moreover, we saw previously in the proof of \thref{Lem: Convergence Ruelle Cpt Subgroup} that the set $ K_{g^{-1}T^n}$ defined by \eqref{Eq: y in support finite} is a subset of $A_1$. Now,
\[
    a \abs{K_{g^{-1}T^n}} \le \sum_{x \in Y^{g^{-1}T^n} \cap A_1}\vol\big(B_\veps(x)\big) \le \vol\left(\bigcup_{x\in X}B_\veps(x) \cap A_2\right) \le \vol(A_2).
\]
However, recall that $T$ is an isometry, and the action of $G$ on $Y$ commutes with $T$. Therefore
\begin{align*}
    \vol(A_2) = \vol\left(\bigcup_{l=0}^{\abs{n}-1}T^{-l}(A\overline{g^\Z}B)\right) \le \vol(A\overline{g^\Z}B) \abs{n}.
\end{align*}
Hence,
\[
    \abs{K_{g^{-1}T^n}} \le a^{-1}\vol(A\overline{g^\Z}B) \abs{n}.
\]
Now for all $\veps > 0$ there exists $C_\veps > 0$ such that
\[
    a^{-1}\vol(A\overline{g^\Z}B) \abs{n} \le C_\veps e^{\veps \abs{n}}
\]
So by \thref{Prop: Convergence of Ruelle for G/Z Compact}, $R^g_{\vphi,\nabla^E}$ convergences absolutely on the set of $\sigma \in \C$ with $\re(\sigma) > \veps$ for all $\veps > 0$.
\end{proof}

\subsection{Compact Groups}\label{Subsec: Compact Group Suspension}

If $G$ is a compact group we can use \thref{Thm: Equivariant Ruelle Suspension} to calculate the equivariant Ruelle $\zeta$-function. However, compactness of $G$ implies that the suspension $M$ is also compact. Hence we can use the Atiyah--Bott Lefschetz fixed-point formula to simplify \eqref{Eq: Equivariant Ruelle Suspension} further. 

\begin{proposition}\thlabel{Prop: Ruelle Suspension Compact}
Suppose that $G$ is a compact group. For all $\sigma \in \C$ with $\re(\sigma) > 0$ large enough, the Ruelle dynamical $\zeta$-function for equivariant suspension flow is given by
\begin{equation}\label{Eq: Equivariant Ruelle Suspension Compact}
    R^g_{\vphi,\nabla^E}(\sigma) = \exp\left(\sum_{n \in \Z \setminus \{0\}}\frac{e^{-|n|\sigma}}{|n|}\Tr^{H^\bullet(Y,E_1)}\big((-1)^Fg^*(T^*)^n\big)\right),
\end{equation}
where $g^*$, $T^*$ are the induced maps on cohomology.
\end{proposition}

\begin{proof}
Normalise the measure on $G/Z$ so that $\vol(G/Z) = 1$, and choose $\psi_M$ to be the constant function 1 on $M$. Then \eqref{Eq: Equivariant Ruelle Suspension} is
\begin{align*}
    R^g&_{\vphi,\nabla^E}(\sigma)
    = \exp\left(\sum_{n \in \Z \setminus\{0\}}\frac{e^{-|n|\sigma}}{|n|}\sum_{y \in Y^{g^{-1}T^{n}}}\sgn\big(\det(1 - D_y(g^{-1}T^n))\big)\tr\big((gT_{E_1}^{-n})|_{(E_1)_y}\big)\right).
\end{align*}
Now using the Atiyah--Bott Lefschetz fixed-point formula, Theorem 6.6 in \cite{BGV}, we obtain
\[
    R^g_{\vphi, \nabla^E}(\sigma) = \exp\left(\sum_{n \in \Z \setminus \{0\}}\frac{e^{-|n|\sigma}}{|n|}\Tr^{H^\bullet(Y,E_1)}\big((-1)^F g^*(T^*)^{-n}\big)\right).
\]
Using the substitution $n \mapsto -n$, we obtain \eqref{Eq: Equivariant Ruelle Suspension Compact}, and \eqref{Eq: Equivariant Ruelle Suspension Compact} converges absolutely when $\re(\sigma) > 0$ is large enough because $\Tr^{H^\bullet(Y,E_1)}\big((-1)^F g^*(T^*)^{-n}\big)$ grows at most exponentially in $n$.
\end{proof}

\begin{example}
Suppose that $Y$ is compact and that $g$, $T$ are commuting isometries of $Y$. As the isometry group of a compact manifold is compact, $G = \overline{g^\Z}$ is compact and \thref{Prop: Ruelle Suspension Compact} holds for the action of $G$ on $M$. 
\end{example}

\section{The Equivariant Fried Conjecture for Suspension Flow}\label{Sec: Suspension Flow}

In this section we finally present proofs of our results on the equivariant Fried conjecture for suspension flow, which were stated in Subsection \ref{Subsec: Results}. First, in Subsection \ref{Subsec: Torsion Suspension}, we calculate the equivariant analytic torsion for the suspension. In the remaining subsections, we prove the results stated in Subsection \ref{Subsec: Results}.

We will use the notation established in Subsection \ref{Subsec: Results} throughout this section.

\subsection{Equivariant Torsion}\label{Subsec: Torsion Suspension}

To calculate equivariant analytic torsion for the suspension, we apply the fibration formula \thref{Thm: New Torsion Formula}, with $M_2 = \R$, $\Gamma = \Z$, and $E_2$ the trivial line bundle. Then quite a few assumptions of \thref{Thm: New Torsion Formula} are already satisfied. Hence we have the following.

\begin{theorem}\thlabel{Prop: Torsion Suspension}
Suppose that $\ker(\Delta_{E}) = 0$, and that $P^{E_1}$ is $g$-trace class. Suppose that one of the following sets of conditions holds.
\begin{enumerate}[label=\emph{(\Roman*)}]
    \item The space $G/Z$ is compact and the Novikov--Shubin numbers $(\alpha_{e}^{p})_{E_1}$ for $\Delta_{E_1}$ are positive.
    \item We have $\Tr_{g}(T_{E_1}^nP^{E_1}) = 0$ for all $n \in \Z$. That there exists a \v Svarc--Milnor function $l_G$ for the action of $G$ on $Y$ with respect to $g$, and a \v Svarc--Milnor function for the action of $G \times \Z$ on $Y$ with respect to $(g,n)$ is given by the sum of $l_G$ and the absolute value function on $\Z$ for all $n \in \Z$. Finally suppose that
    \[
        \vol\big\{hZ \in G/Z : l_G(hgh^{-1}) \le r\big\}
    \]
    has at most exponential growth in $r$.
\end{enumerate}
Then for $\sigma \in \C$ with $\re(\sigma) > 0$ large enough
\begin{equation}\label{Eq: Equivariant Torsion Suspension}
    T_g(\nabla^{E},\sigma) = \exp\Big(-\frac{\chi_{(g,0)}(\nabla^{E_1})\sqrt\sigma}{2}\Big)\exp\Big(\sum_{n \in \Z \setminus \{0\}}\frac{e^{-\abs{n}\sqrt{\sigma}}\chi_{(g,n)}(\nabla^{E_1})}{2|n
    |}\Big).
\end{equation}
\end{theorem}

\begin{proof}
First note that the $L^2$-kernel of the Hodge Laplacian on $\R$ is trivial, and so $P^{E_2} = 0$ and trivially $n$-trace class for all $n \in \Z$. Moreover, using Lemma 6.4 in \cite{hochs2022equivariant} we obtain
\[
    \Tr_n(e^{-t\Delta_{E_2}}) = \frac{1}{\sqrt{4\pi t}}e^{-n^2/4t},
\]
which is $\mathcal O(t^{-1/2})$ as $t \to \infty$, and therefore the delocalised Novikov--Shubin numbers for $d$ are positive. Finally note that as the $\Z$-action on $\R$ is proper and cocompact, the absolute value function on $\Z$ is a \v Svarc--Milnor function for the action by the \v Svarc--Milnor lemma. Hence the assumptions in \thref{Thm: New Torsion Formula} concerning $M_2 = \R$ are satisfied. 

As $\Z$ is abelian, using \thref{Thm: New Torsion Formula} we find that for all $\sigma \in \C$ with $\re(\sigma) > 0$ large enough that
\[
    T_g(\nabla^{E},\sigma) = \prod_{n \in \Z}T_{(g,n)}(\nabla^{E_1},\sigma)^{\chi_n(\nabla^{E_2})}T_n(d,\sigma)^{\chi_{(g,n)}(\nabla^{E_1})}.
\]
As $\R$ is odd-dimensional, $\chi_n(E_2) = 0$ for all $n \in \Z$ by Example 5.3 in \cite{hochs2022equivariant}. Hence the above expression for $T_g(\nabla^E,\sigma)$ becomes
\begin{equation}\label{Eq: Equivariant Torsion Suspension Proof 1}
    T_g(\nabla^{E},\sigma) = \prod_{n \in \Z}T_n(d,\sigma)^{\chi_{(g,n)}(\nabla^{E_1})}.
\end{equation}
However, from Proposition 6.2 of \cite{hochs2022equivariant}, $T_n(d,\sigma) = \exp(e^{-\abs{n}\sqrt{\sigma}}/2\abs{n})$ for $n \neq 0$ and $T_0(d,\sigma) = \exp(-\sqrt{\sigma}/2)$, and so \eqref{Eq: Equivariant Torsion Suspension Proof 1} equals
\begin{align*}
    T_g(\nabla^{E},\sigma) &= \exp\Big(-\frac{\chi_{(g,0)}(\nabla^{E_1})\sqrt\sigma}{2}\Big)\prod_{n \in \Z \setminus \{0\}}\exp\Big(\frac{e^{-\abs{n}\sqrt{\sigma}}\chi_{(g,n)}(\nabla^{E_1})}{2|n|}\Big)\\
    &= \exp\Big(-\frac{\chi_{(g,0)}(\nabla^{E_1})\sqrt\sigma}{2}\Big)\exp\Big(\sum_{n \in \Z \setminus \{0\}}\frac{e^{-\abs{n}\sqrt{\sigma}}\chi_{(g,n)}(\nabla^{E_1})}{2|n|}\Big). \qedhere
\end{align*}
\end{proof}

From \thref{Prop: Torsion Suspension}, to calculate the equivariant torsion of the suspension we need to calculate the $(g,n)$-Euler characteristics $\chi_{(g,n)}(\nabla^{E_1})$ for $n \in \Z$. If the fixed point set $Y^{g^{-1}T^n}$ is discrete and contains only non-degenerate fixed points for all $n \in \Z \setminus \{0\}$, then using \thref{Prop: Euler Char as Fixed Points} we obtain

\begin{proposition}\thlabel{Prop: Torsion Suspension Fix Points}
Suppose that the assumptions of \thref{Prop: Torsion Suspension} hold, and that the sets $Y^{g^{-1}T^n}$ are either discrete or empty and contain only non-degenerate fixed points for all $n \in \Z \setminus\{0\}$. Then for $\sigma \in \C$ with $\re(\sigma) > 0 $ large enough
\begin{equation}
\begin{aligned}\label{Eq: Analytic Torsion Suspension}
    &\log T_g(\nabla^{E},\sigma) \\
    &= -\frac{\chi_{(g,0)}(\nabla^{E_1})\sqrt{\sigma}}{2}+\sum_{n \in \Z\setminus \{0\}} \frac{e^{-\abs{n}\sqrt{\sigma}}}{2|n|}\sum_{y \in Y^{g^{-1}T^n}} \psi^g(y)\sgn\det\big(1 - D_y(g^{-1}T^{n})\big)\tr\big(gT_{E_1}^{-n}|_{(E_1)_y}\big).
\end{aligned}
\end{equation}
\end{proposition}

\begin{proof}
Recall that there is a vector bundle endomorphism on the bundle $\extp \bullet T^*Y \otimes E_1$ covering $T \colon Y \to Y$, which is given by the tensor product of the operator on $\extp \bullet T^*Y$ induced by the derivative of $T$, and the map $T_{E_1}$ on $E_1$. We denote this operator on $\extp \bullet T^*Y \otimes E_1$ by $T_W$. Then using the substitution $n \mapsto -n$, and \thref{Prop: Euler Char as Fixed Points}, we can rewrite \eqref{Eq: Equivariant Torsion Suspension} as
\begin{equation}
\begin{aligned}
    &\log T_g(\nabla^{E},\sigma)\\
    &= -\frac{\chi_{(g,0)}(\nabla^{E_1})\sqrt{\sigma}}{2}+\sum_{n \in \Z\setminus \{0\}} \frac{e^{-\abs{n}\sqrt{\sigma}}}{2|n|}\sum_{y \in Y^{g^{-1}T^n}} \psi^g(m)\frac{\tr\big((-1)^FgT_{W}^{-n}|_{\extp \bullet T_y^*Y \otimes (E_1)_y}\big)}{\abs{\det\big(1 - D_y(g^{-1}T^{n})\big)}}.
\end{aligned}
\end{equation}
The result now follows as
\begin{align*}
    \tr\big((-1)^FgT_{W}^{-n}|_{\extp \bullet T_y^*Y \otimes (E_1)_y}\big) &= \sum_{j=0}^{\dim Y}(-1)^j \Tr\big(gT^{-n}|_{\extp j T^*_yY}\big) \tr(gT_{E_1}^{-n}|_{(E_1)_y}\big)\\
    &= \det\big(1 - D_y(g^{-1}T^{n})\big)\tr(gT_{E_1}^{-n}|_{(E_1)_y}\big).\qedhere
\end{align*}
\end{proof}

In the case that $G$ is a compact group, we use \thref{Lem: g-Trace Compact Group} to calculate the $(g,n)$-Euler characteristics. Doing so, we obtain the following result.

\begin{proposition}\thlabel{Prop: Suspension Torsion Compact}
Suppose that $G$ is compact and that $H^\bullet(M,E) = 0$. Then for all $\sigma$ with $\re(\sigma) > 0$ large enough, we have
\begin{equation}\label{Eq: Torsion Suspension Compact}
        T_g(\nabla^{E},\sigma) = \exp\Big(-\frac{\chi_{(g,0)}(\nabla^{E_1})\sqrt\sigma}{2}\Big)\exp\Big(\sum_{n \in \Z \setminus \{0\}} \frac{e^{-\abs{n}\sqrt{\sigma}}\Tr^{H^\bullet(Y,E_1)}\big((-1)^Fg^*(T^*)^n\big)}{2\abs{n}}\Big),
\end{equation}
where $g^*, T^*$ are the induced maps on cohomology. Moreover, the equivariant analytic torsion $T_g(\nabla^E)$ is well-defined and is given by
\begin{equation}\label{Eq: Torsion Suspension Compact 2}
    T_g(\nabla^{E}) = \exp\Big(\sum_{n \in \Z \setminus \{0\}} \frac{\Tr^{H^\bullet(Y,E_1)}\big((-1)^Fg^*(T^*)^n\big)}{2\abs{n}}\Big).
\end{equation}
\end{proposition}

\begin{proof}
As $G$ is compact, $G/Z$ and $Y$ are compact, and so the assumptions stated at the start of \thref{Prop: Torsion Suspension} hold. Further as $Y$ is compact, case (I) of \thref{Prop: Torsion Suspension} also holds. Thus by \thref{Prop: Torsion Suspension}, we have $T_g(\nabla^E,\sigma)$ is given by \eqref{Eq: Equivariant Torsion Suspension} for $\re(\sigma) > 0$ large enough. However, by \thref{Lem: g-Trace Compact Group},
\[
    \chi_{(g,n)}(\nabla^{E_1}) = \Tr^{H^\bullet(Y,E_1)}\big((-1)^Fg^*(T^*)^n\big),
\]
and \eqref{Eq: Torsion Suspension Compact} follows immediately from \thref{Prop: Torsion Suspension}. 

As $g^*$ and $T^*$ are commuting isometries they have a joint eigenspace decomposition. Let $\mu_{j,q}(g^*)$ and $\mu_{j,q}(T^*)$ denote the eigenvalues of the maps $g^*$ and $T^*$ on $H^q(Y,E_1)$, where $j$ runs over a finite set parametrising the eigenvalues. We also repeat the eigenvalues according to their multiplicities. Then
\begin{equation}\label{Eq: Compact Meromorphic 1}
    \Tr^{H^\bullet(Y,E_1)}\big((-1)^Fg^*(T^*)^n\big) = \sum_{q=0}^{\dim Y}\sum_j (-1)^q \mu_{j,q}(g^*)\mu_{j,q}(T^*)^n.
\end{equation}
Using the Taylor series of $x \mapsto \log(1 + x)$, for $\sigma \in \C$ with $\re(\sigma) > 1$ we have
\begin{equation}\label{Eq: Compact Meromorphic 2}
    \sum_{n \in \Z \setminus \{0\}} \frac{e^{-\abs{n}\sigma}\mu_{j,q}(T^*)^n}{\abs{n}} = \log\big((1 - e^{-\sigma}\mu_{j,q}(T^*))(1 - e^{-\sigma}\mu_{j,q}(T^*)^{-1})\big)^{-1}.
\end{equation}
Hence we see that substituting \eqref{Eq: Compact Meromorphic 1} and \eqref{Eq: Compact Meromorphic 2} into \eqref{Eq: Torsion Suspension Compact} gives
\begin{equation}\label{Eq: Compact Meromorphic 3}
\begin{aligned}
    T_g(\nabla^E,&\sigma) = \exp\Big(-\frac{\chi_{(g,0)}(\nabla^{E_1})\sqrt{\sigma}}{2}\Big)\\
    &\qquad\quad \cdot \prod_{q=0}^{\dim Y}\Big(\prod_j \big((1 - e^{-\sqrt\sigma}\mu_{j,q}(T^*))(1 - e^{-\sqrt\sigma}\mu_{j,q}(T^*)^{-1})\big)^{-\mu_{j,q}(g^*)/2}\Big)^{(-1)^q}.
\end{aligned}
\end{equation}
For raising to a complex power to be well-defined we choose the branch cut of the non-positive real numbers, and note that the base of the complex powers in \eqref{Eq: Compact Meromorphic 3} do not lie in this branch cut for all $j$ and $q$ when $\re(\sigma) > 0$. Therefore we see that \eqref{Eq: Compact Meromorphic 3} defines a meromorphic extension of $T_g(\nabla^E,\sigma)$ to the set of $\sigma$ with $\re(\sigma) > 0$. 

As there exists a \v Svarc--Milnor function for the action of $G$ on $M$, and condition (II) of \thref{Lem: Alternative Convergence Torsion} holds as $M$ is compact, it follows that the equivariant analytic torsion for the suspension is well-defined by \thref{Lem: Alternative Convergence Torsion}. Hence $T_g(\nabla^E,\sigma)$ has a meromorphic extension which is holomorphic at $\sigma = 0$. To calculate $T_g(\nabla^E)$ first note that
\begin{equation}\label{Eq: Torsion Suspension Compact Proof 2}
    \exp\Big(\sum_{n \in \Z \setminus \{0\}} \frac{\Tr^{H^\bullet(Y,E_1)}\big((-1)^Fg^*(T^*)^n\big)}{2\abs{n}}\Big)
\end{equation}
converges. This is because following the working above, and using the fact that $\abs{\mu_{j,q}(T^*)} = 1$, we see that \eqref{Eq: Torsion Suspension Compact Proof 2} equals
\begin{equation}\label{Eq: Torsion Suspension Compact Proof 3}
    \prod_{q=0}^{\dim Y}\Big(\prod_j \abs{1 - \mu_{j,q}(T^*)}^{-\mu_{j,q}(g^*)}\Big)^{(-1)^q}.
\end{equation}
This converges as $\mu_{j,q}(T^*)$ does not equal 1 for all $j$ and $q$ because $H^\bullet(M,E) = 0$ (see Proposition 3.2 in \cite{Analytic.Torsion.Fried.Survey}). Finally, from Abel's limit theorem, see Section 2.5 in \cite{ComplexAnalysisAhlfors}, 
\[
    \lim_{\sigma \to 0^+} T_g(\nabla^E,\sigma) = \exp\Big(\sum_{n \in \Z \setminus \{0\}} \frac{\Tr^{H^\bullet(Y,E_1)}\big((-1)^Fg^*(T^*)^n\big)}{2\abs{n}}\Big)
\]
where the limit can be taken along the positive real line. As the meromorphic extension of $T_g(\nabla^E,\sigma)$ is holomorphic at 0, it continuous on region where we take the limit, and thus we obtain \eqref{Eq: Torsion Suspension Compact 2}.
\end{proof}

\subsection{The Equivariant Fried Conjecture}\label{Subsec: Equivariant Fried Conjecture}

The goal of the remaining subsections is prove the results stated in Subsection \ref{Subsec: Results}. We start with the proof of \thref{Thm: Most General Fried Suspension}, which follows almost immediately from \thref{Prop: Torsion Suspension Fix Points} and \thref{Thm: Equivariant Ruelle Suspension}.

\begin{proof}[Proof of \thref{Thm: Most General Fried Suspension}]
As the Ruelle $\zeta$-function converges absolutely for $\sigma$ with $\re(\sigma) > 0$ large enough by assumption, using \thref{Thm: Equivariant Ruelle Suspension} we find that
\begin{equation}\label{Eq: Ruelle Suspension Sums and Integrals}
\begin{aligned}
    R^g_{\vphi, \nabla^E}(\sigma) &= \exp\Big(\sum_{n \in \Z \setminus \{0\}}\frac{e^{-|n|\sigma}}{|n|}\sum_{y \in Y^{g^{-1}T^n}}\sgn\det\big(1 - D_y(g^{-1}T^n\big)\tr(gT_{E_1}^{-n}|_{(E_1)_y})\\
    &\qquad \qquad \qquad \qquad \qquad\qquad\cdot \frac{1}{p(y)}\int_{G/Z}\int_0^{p(y)}\psi_M[hy,s]ds\;d(hZ)\Big).
\end{aligned}
\end{equation}
If \eqref{Eq: Most General Fried Condition} holds, then \eqref{Eq: Ruelle Suspension Sums and Integrals} equals $\exp\big(\sigma\chi_{(g,0)}(\nabla^{E_1}))\big)T_g(\nabla^{E},\sigma^2)^2$ by \thref{Prop: Torsion Suspension Fix Points}. As
\begin{equation}\label{Eq: Most General Fried Proof}
    R^g_{\vphi,\nabla^E}(\sigma) = \exp\big(\sigma\chi_{(g,0)}(\nabla^{E_1}))\big)T_g(\nabla^E,\sigma^2)^2
\end{equation}
for $\re(\sigma) > 0$ large enough, it follows that right hand side of \eqref{Eq: Most General Fried Proof} has a meromorphic extension which is holomorphic at 0 if, and only if, the right side does as well. Thus if $T_g(\nabla^E)$ is well-defined then the equivariant Fried conjecture holds. 
\end{proof}

Thus, it now remains to look at cases where the conditions of \thref{Thm: Most General Fried Suspension} are satisfied.

\subsection{Compact Groups}\label{Subsec: Compact Group Fried}

In the case that $G$ is compact, we have \thref{Cor: Equivariant Fried Compact}.

\begin{proof}[Proof of \thref{Cor: Equivariant Fried Compact}]
As $G$ is compact, we may normalise the respective measures so that $\vol(G) = \vol(Z) = \vol(G/Z) = 1$. We then choose the cutoff functions $\psi_M$, $\psi_G$, and $\psi_Y$ to all be the constant function 1. Thus,
\[
    \int_{G/Z}\int_0^{p(y)}\frac{\psi_M[hy,s]}{p(y)}\; ds\; d(hZ) = 1 = \psi^g(y),
\]
for all $y \in Y^{g^{-1}T^n}$ and $n \in \Z \setminus \{0\}$. Therefore, for $\re(\sigma) > 0$ large enough, $R^g_{\vphi,\nabla^E}(\sigma) = \exp\big(\sigma\chi_{(g,0)}(\nabla^{E_1}))\big)T_g(\nabla^E,\sigma^2)^2$ by \thref{Thm: Most General Fried Suspension}. As $M$ is compact, and as $\ker \Delta_E = 0$, the equivariant analytic torsion is well-defined by \thref{Prop: Suspension Torsion Compact}, and hence the equivariant Fried conjecture holds.
\end{proof}

\begin{remark}
\thref{Cor: Equivariant Fried Compact} can also be deduced immediately from \thref{Prop: Suspension Torsion Compact} and \thref{Prop: Ruelle Suspension Compact}.
\end{remark}

\subsection{Compact Centraliser}\label{Subsec: Compact Z Fried}

Let us now consider the case when the centraliser of $g$ is compact.

\begin{proof}[Proof of \thref{Cor: Equivariant Fried Compact Z}]
From \thref{Prop: Convergence of Ruelle for Z Compact}, the equivariant Ruelle $\zeta$-function converges absolutely for $\sigma$ with $\re(\sigma) > c$, and is given by \eqref{Eq: Ruelle Suspension Z Compact}.

As $Z$ is compact we may normalise the Haar measure on $G$ such that $\vol(Z) = 1$. Then it follows that $\psi_G \equiv 1$ satisfies \eqref{Eq: Cutoff Z for G}, and so it follows that
\[
    \psi^g(y) = \int_G \psi_G(x)\psi_Y(xgy) dx = 1
\]
due to the fact that $\psi_Y$ is a cutoff function for the action of $G$ on $Y$. Thus by \thref{Thm: Most General Fried Suspension}, $R^g_{\vphi,\nabla^E}(\sigma) = \exp\big(\sigma\chi_{(g,0)}(\nabla^{E_1}))\big)T_g(\nabla^E,\sigma^2)^2$.
\end{proof}

In the rest of this subsection we present a special case of when \thref{Cor: Equivariant Fried Compact Z} holds, which also provides the proof for \thref{Ex: Fried Compact Z}.

Let $G$ be a real, connected, semisimple Lie group. Suppose that $K < G$ is a maximal compact subgroup, and that $S < K$ is a maximal torus. Let $g \in S$ and suppose $g$ generates a dense subgroup of $S$. Let $N$ be a compact Riemannian manifold equipped with an isometric $S$-action, and let $T \colon N \to N$ be a $S$-equivariant isometry. If for all $\ell \in \Z \setminus \{0\}$, $N^{g^{-1}T^\ell}$ is finite, then taking the suspension of $N$ with suspension flow defines a compact example of the equivariant Fried conjecture. 

We extend this to a non-compact example by the following procedure. Form the fibred product $Y = G \times_S N$ where $S$ acts on $G$ via right multiplication, and define $\widetilde T \colon Y \to Y$ by $\widetilde T[x,n] = [x,Tn]$. Then $\widetilde T$ is well-defined as $T$ is $S$-equivariant. Moreover, $\widetilde T$ is $G$-equivariant with respect to the action $h \cdot [x,n] = [hx,n]$ on $Y$. As $TY = G \times_S (TN \oplus \mathfrak g/\mathfrak s)$, the $S$-invariant Riemannian metric on $N$ and an $\Ad(K)$-invariant inner product on $\mathfrak g$ induce a $G$-invariant Riemannian metric on $G$, for which $\widetilde T$ is an isometry.

For the rest of this subsection for each element $w$ of the Weyl group $N_K(S)/S$, we fix a representative in $N_K(S)$ which we also denote by $w$.

\begin{lemma}\thlabel{Lem: Example Compact Z Fixed Point}
Suppose that $\mathrm{rank }(K) = \mathrm{rank}(G)$. Then for all $\ell \in \Z \setminus \{0\}$,
\begin{equation}\label{Eq: Compact Z Example Fixed Point}
\begin{aligned}
    Y^{g^{-1}\widetilde T^\ell} &= \coprod_{w \in N_K(S)/S}\big\{[w, n] : n \in N^{wg^{-1}w^{-1}T^\ell}\big\}.
\end{aligned}
\end{equation}
Moreover, there is a fixed compact subset of $Y$ containing \eqref{Eq: Compact Z Example Fixed Point} for all $\ell \in \Z \setminus \{0\}$.
\end{lemma}

\begin{proof}
Let $[x, n] \in Y^{g^{-1}\widetilde T^\ell}$. Then $[x,n] = [g^{-1}x,T^\ell n]$ so there exists $s \in S$ such that $g^{-1}x = xs^{-1}$ and $T^\ell n = sn$. Hence $x^{-1}gx = s \in S$, and as the powers of $g$ are dense in $S$, this implies that $x \in N_G(S)$. As $\mathrm{rank }(K) = \mathrm{rank}(G)$ we have $N_G(S) = N_K(S)$. Using the fact that $x^{-1}gx^{-1} = s$, we find that $T^\ell n = x^{-1}gxn$, which implies that $n \in N^{xg^{-1}x^{-1}T^\ell}$. Altogether, we find that
\begin{equation}\label{Eq: Compact Z Example Proof 1}
    Y^{g^{-1}\widetilde T^\ell} = \bigcup_{x \in N_K(S)} Y_{x,\ell},
\end{equation}
where 
\[
    Y_{x,\ell} \coloneqq \big\{[x,n] : n \in N^{xg^{-1}x^{-1}T^\ell}\big\}.
\]
We claim that for all $s \in S$, $x \in N_K(S)$, and $\ell \in \Z \setminus \{0\}$, that we have $Y_{xs,\ell} = Y_{x,\ell}$. First note that, as $g \in S$, we have
\[
    N^{(xs)g^{-1}(xs)^{-1}T^\ell} = N^{xg^{-1}x^{-1}T^\ell}, 
\]
and that $N^{xg^{-1}x^{-1}T^\ell}$ is $S$-invariant because $xg^{-1}x^{-1} \in S$. Now elements of $Y_{xs,\ell}$ are of the form $[xs,n] = [x,sn]$ for $n \in N^{xg^{-1}x^{-1}T^\ell}$. Thus it follows that $Y_{xs,\ell} \subseteq Y_{x,\ell}$ for all $s \in S$, $x \in N_K(S)$, and $\ell \in \Z \setminus \{0\}$. The reverse inclusion follows now as for all $x \in N_K(S)$ and $s \in S$ that $xs \in N_K(S)$ and so $Y_{x,\ell} = Y_{(xs)s^{-1},\ell} \subseteq Y_{xs,\ell}$ by the previous working. Hence, it follows that \eqref{Eq: Compact Z Example Proof 1} equals
\begin{equation}\label{Eq: Compact Z Example Proof 2}
    \bigcup_{x \in N_{K}(S)/S}Y_{x,\ell},
\end{equation}
where we pick a representative of every coset. Finally note that if $x,x' \in N_K(S)$ and $Y_{x,\ell} = Y_{x',\ell}$, then it follows from the definition of the set $Y_{x,\ell}$ that $xS = x'S$, and so define the same element in the Weyl group $N_K(S)/S$. Thus the union in \eqref{Eq: Compact Z Example Proof 2} is disjoint, and equals \eqref{Eq: Compact Z Example Fixed Point}.

For the final statement, let $W \subseteq N_K(S)$ be a set of representatives for the Weyl group. As the Weyl group is finite, so is $W$, and hence $W$ is compact. It follows from \eqref{Eq: Compact Z Example Fixed Point} that for all $l \in \Z \setminus \{0\}$ the set $Y^{g^{-1}\widetilde T^l}$ is contained in the image of $W \times N$ under the canonical projection $G \times N \to Y$ which is a compact subset of $Y$. 
\end{proof}

Suppose that $N^{wg^{-1}w^{-1}\widetilde T^l}$ is finite for all $l \in \Z \setminus \{0\}$ and $w \in N_K(S)/S$, and suppose further that there exist $C,c > 0$ such that $\abs{N^{wg^{-1}w^{-1}\widetilde T^l}} \le Ce^{c\abs{l}}$ for all $l \in \Z \setminus \{0\}$ and $w \in N_K(S)/S$. 

Let $M$ be the suspension of $Y$ with respect to $\widetilde T$, and consider the equivariant suspension flow on $M$. Then by \thref{Prop: Convergence of Ruelle for Z Compact}, the equivariant Ruelle $\zeta$-function converges on the set of $\sigma$ with $\re(\sigma) > c$. 

\begin{example}
Suppose that $N = K/S$ and that $T \colon N \to N$ is left multiplication by an element $T \in S$. Suppose that for all $l \in \Z$ the elements $g$, and $g^{-1}x^{-1}T^lx$ for all $x \in N_K(S)/S$, generate dense subgroups of $S$. Further suppose that $g \neq T^k$ for all $k \in \Z$. Then for all $l \in \Z$, as $g^{-1}x^{-1}T^lx$ has dense powers,
\[
    N^{g^{-1}xT^lx^{-1}} = N^S = N_K(S)/S,
\]
which is finite. Moreover,
\[
    N^{xg^{-1}x^{-1}T^l} = xN^{g^{-1}x^{-1}T^l} = N_K(S)/S,
\]
where the second equality follows as $g^{-1}x^{-1}T^lx$ has dense powers. Thus, it follows that $Y^{g^{-1}T^l}$ is contained in the image of the finite set $W \times N_K(S)/S$ under the projection map of $G \times (K/S) \to Y$ for all $l \in \Z \setminus \{0\}$. Hence for all $l \in \Z\setminus\{0\}$ it follows that $\abs{Y^{g^{-1}T^l}}$ is bounded by $\abs{N_K(S)/S}^2$. Therefore the equivariant Ruelle $\zeta$-function for suspension flow converges absolutely on the set of $\sigma$ with $\re(\sigma) > 0$ by \thref{Prop: Convergence of Ruelle for Z Compact}.
\end{example}

Thus, if the equivariant analytic torsion for the suspension of $Y$ is well-defined, then the equivariant Fried conjecture for suspension flow holds by \thref{Cor: Equivariant Fried Compact Z}.

\subsection{The Identity Element of a Discrete Group}\label{Subsec: Identity Discrete Group}

We now consider the case when $g = e$ is the identity element of a non-compact group $G$ which we assume to be discrete. Under the assumption that the sets $Y^{T^n}$ are either empty or discrete for all $n \in \Z \setminus \{0\}$ there is a compact subgroup $H \le G$ such that $G/H$ is discrete. Hence the assumption that $G$ itself is discrete is relatively mild.

\begin{proposition}\thlabel{Prop: Ruelle g=e}
Suppose that
\begin{enumerate}[label=\emph{(\Roman*)}]
    \item there exist $C,c > 0$ such that for all $n \in \Z \setminus \{0\}$, we have $\abs{K_{T^n}} \le Ce^{c\abs{n}}$, where $K_{T^n}$ is the set defined in \eqref{Eq: y in support finite}; and
    \item the set of primitive periods
    \begin{equation}\label{Eq: Set Primitive Periods}
        \Lambda = \{p(y) : n \in \Z \setminus \{0\}, y \in Y^{T^n}\}
    \end{equation}
    is finite.
\end{enumerate}
Then the equivariant Ruelle dynamical $\zeta$-function converges absolutely at $\sigma \in \C$ with $\re(\sigma) > c$, and is given by
\begin{equation}\label{Eq: Ruelle g=e}
\begin{aligned}
    &R^e_{\vphi,\nabla^E}(\sigma) \\
    &= \exp\left(\sum_{n \in \Z \setminus\{0\}}\frac{e^{-\abs{n}\sigma}}{\abs{n}}\sum_{y \in Y^{T^n}}\sgn\det(1 - D_yT^n)\Tr(T^{-n}_E|_{(E_1)_y})\frac{1}{p(y)}\sum_{j=0}^{p(y)-1}\psi_Y(T^jy)\right).
\end{aligned}
\end{equation}
\end{proposition}

\begin{proof}
For ease of notation, let
\[
    a(y,n) = \sgn\det(1 - D_yT^n)\Tr(T^{-n}_E|_{(E_1)_y}).
\]
By assumption (II) there exist $b_1,b_2 \in \R$ such that $0 < b_1 < b_2$ and for all $n \in \Z \setminus \{0\}$, $y \in Y^{T^n}$, we have $b_1 \le p(y) \le b_2$. 

Let $(\psi_j)_{j=1}^\infty$ be a sequence of cutoff functions for the action of $\Z$ on $\R$ such that for all $j$,
\begin{enumerate}
    \item $\norm{\psi_j}_\infty \le 1$,
    \item $\supp \psi_j \subseteq [-1,2]$; and
    \item $\psi_j$ converges pointwise to $1_{[0,1]}$, the indicator function on $[0,1]$.
\end{enumerate}
Then for all $n \in \Z \setminus \{0\}$, $y \in Y^{T^n}$, $s \in [0,p(y)]$, $k \in \Z$, and $j \in \N$ we have that if $\psi_j(s - k) \neq 0$ then $s- k \in [-1,2]$. But this further implies that $-k \in [-1-b_2,2]$, and so $k \in [-2,b_2 + 1]$. Therefore 
\[
    \sum_{k \in \Z}\psi_Y(T^ky)\psi_j(s-k)
\]
has at most $b_2 + 4$ non-zero terms. Hence
\begin{equation}\label{Eq: Ruelle g=e proof inequality}
\begin{aligned}
      \sum_{y \in K_{T^n}}\abs{a(y,n)\frac{1}{p(y)}\int_0^{p(y)}\sum_{k \in \Z}\psi_Y(T^ky)\psi_j(s-k)ds} & \le Ce^{c\abs{n}}\frac{b_2}{b_1}(b_2 + 4)\norm{\psi_Y}_\infty.
\end{aligned}
\end{equation}
Set $F(n) = Ce^{c\abs{n}}\frac{b_2}{b_1}(b_2 + 4)\norm{\psi_Y}_\infty$. Then for $\sigma \in \C$ with $\re(\sigma) > c$ we see that
\[
    \sum_{n \in \Z \setminus \{0\}}\frac{e^{-\abs{n}\sigma}}{\abs{n}}F(n)
\]
converges absolutely.

Note that inequality \eqref{Eq: Ruelle g=e proof inequality} still holds for $\psi_j$ replaced by $1_{[0,1]}$. Hence by the dominated convergence theorem,
\begin{align*}
    \lim_{j \to \infty}\sum_{n \in \Z \setminus \{0\}}&\frac{e^{-\abs{n}\sigma}}{\abs{n}}\sum_{y \in Y^{T^n}}a(y,n)\frac{1}{p(y)}\int_0^{p(y)}\sum_{k \in \Z}\psi_Y(T^ky)\psi_j(s-k)ds\\
    &= \sum_{n \in \Z \setminus \{0\}}\frac{e^{-\abs{n}\sigma}}{\abs{n}}\sum_{y \in Y^{T^n}}a(y,n)\frac{1}{p(y)}\int_0^{p(y)}\sum_{k \in \Z}\psi_Y(T^ky)1_{[0,1]}(s-k)ds.
\end{align*}
Now, using \thref{Prop: Convergence of Ruelle for G/Z Compact} we find that $R^e_{\vphi,\nabla^E}$ converges absolutely on the set of $\sigma$ with $\re(\sigma) > c$, and
\begin{align*}
    R^e_{\vphi,\nabla^E}(\sigma) &= \lim_{j \to \infty}\exp\left(\sum_{n \in \Z \setminus \{0\}}\frac{e^{-\abs{n}\sigma}}{\abs{n}}\sum_{y \in Y^{T^n}}a(y,n)\frac{1}{p(y)}\int_0^{p(y)}\sum_{k \in \Z}\psi_Y(T^ky)\psi_j(s-k)ds\right)\\
    &= \exp\left(\sum_{n \in \Z \setminus \{0\}}\frac{e^{-\abs{n}\sigma}}{\abs{n}}\sum_{y \in Y^{T^n}}a(y,n)\frac{1}{p(y)}\int_0^{p(y)}\sum_{k \in \Z}\psi_Y(T^ky)1_{[0,1]}(s-k)ds\right),
\end{align*}
where the first equality follows as $R^e_{\vphi,\nabla^E}$ is independent of the choice of cutoff function, see Proposition 3.6 in \cite{hochs2023ruelle}, and the second follows as the exponential map is continuous. Finally, note that 
\begin{align*}
    \int_0^{p(y)}\sum_{k \in \Z}\psi_Y(T^ky)1_{[0,1]}(s-k)ds &=  \sum_{j=0}^{p(y) - 1}\psi_Y(T^jy)ds.
\end{align*}
Therefore, we obtain \eqref{Eq: Ruelle g=e}.
\end{proof}

\begin{proof}[Proof of \thref{Cor: Equivariant Fried Identity}]
From \thref{Prop: Ruelle g=e}, we have seen that the Ruelle $\zeta$-function converges absolutely for $\sigma$ with $\re(\sigma) > 0$ large enough, and on this domain is given by \eqref{Eq: Ruelle g=e}. Set
\[
    q \coloneqq \prod_{p(y) \in \Lambda}p(y),
\]
where $\Lambda$ is the set of all primitive periods defined in \eqref{Eq: Set Primitive Periods}. Then $q$ is finite as $\Lambda$ is finite by assumption. Define $\widetilde \psi_Y \in C^\infty_c(Y)$ by
\[
    \widetilde \psi_Y(y) = \frac{1}{q}\sum_{j=0}^{q-1}\psi_Y(T^jy).
\]
As $\psi_Y$ is cutoff function for the action of $G$ on $Y$, so is $\widetilde \psi_Y$. Moreover, for all $n \in \Z \setminus \{0\}$, and $y \in Y^{T^n}$, we have $\{T^jy : j=0,\dots,q\} = \{T^jy : j = 0,\dots,p(y)\}$. Thus we find that 
\begin{equation}\label{Eq: g=e proof 1}
    \widetilde \psi_Y(y) = \frac{q/p(y)}{q}\sum_{j=0}^{p(y) - 1}\psi_Y(T^jy) = \frac{1}{p(y)}\sum_{j=0}^{p(y) - 1}\psi_Y(T^jy).
\end{equation}
On the other hand, as $G$ is a discrete group the delta function $\delta_e$ at $e$ is a cutoff function for the action of $G$ on itself by right multiplication. Hence $\psi^e \in C^\infty(Y)$ defined by
\[
    \psi^e(y) = \sum_{x \in G} \delta_e(x)\widetilde \psi_Y(xy),
\]
satisfies 
\begin{equation}\label{Eq: g=e proof 2}
   \psi^e(y) = \widetilde \psi_Y(y),
\end{equation}
for all $y \in Y$. Therefore, \eqref{Eq: g=e proof 1} and \eqref{Eq: g=e proof 2} together with \thref{Prop: Ruelle g=e} imply that $R^e_{\vphi,\nabla^E}(\sigma) = \exp\big(\sigma\chi_{(e,0)}(\nabla^{E_1}))\big)T_e(\nabla^E,\sigma^2)^2$ for $\sigma \in \C$ with $\re(\sigma) > c$ large enough.
\end{proof}

In the rest of this subsection we show that the assumptions of \thref{Cor: Equivariant Fried Identity} are satisfied in the setting of \thref{Ex: Fried Identity Example}. From now on, we use the notation and assumptions from \thref{Ex: Fried Identity Example}. 

Let $\tilde E_1 \to X$ be a Hermitian, flat, equivariant vector bundle over $X$ with Laplacian denoted by $\Delta_X$. Lift this to a flat Hermitian equivariant bundle $E_1 = \Gamma \times \tilde E_1$ over $Y$ with Laplacian $\Delta_{E_1}$.

Let $\psi_Y$ be the cutoff function defined by $\psi_Y \equiv 1$ on $\{e\} \times X$, and 0 otherwise. Then as $Y^{T^n} = \Gamma \times X^{T^n} = \Gamma \times X^T$, it follows that the number of distinct primitive periods is finite. Moreover, we see that $K_{T^n}$ as defined in \eqref{Eq: y in support finite} equals $\{e\} \times X^T$ which is finite and independent of $n \in \Z \setminus \{0\}$. Therefore, the equivariant Ruelle $\zeta$-function for suspension flow converges absolutely on the set of $\sigma$ with $\re(\sigma) > 0$ by \thref{Prop: Ruelle g=e}.

Therefore by \thref{Cor: Equivariant Fried Identity}, the equivariant Fried conjecture is true if the equivariant analytic torsion for the suspension of $Y$ is well-defined. Thus in the rest of this subsection, we show that the equivariant analytic torsion is well-defined.

We first note that as $e$ is a central element of $G$, the projection $P^{E_1}$ is $e$-trace class. Thus it remains to check positivity of the relevant Novikov--Shubin numbers.

\begin{proposition}\thlabel{Prop: g=e Novikov-Shubin invariants}
The delocalised Novikov--Shubin numbers $\alpha_e^p$ for $\Delta_{E_1}$ are infinite for all $p$.
\end{proposition}

\begin{proof}
We claim that
\begin{equation}\label{Eq: NS for example g=e}
    \Tr_e(e^{-t\Delta_{E_1}}-P^{E_1}) = \Tr(e^{-t\Delta_X} - P^X),
\end{equation}
where $P^X$ is the projection onto the kernel of $\Delta_X$. The result then follows, as the right hand side of \eqref{Eq: NS for example g=e} is $\mathcal O(t^{-\alpha})$ for all $\alpha > 0$ as $X$ is compact, and so the Novikov--Shubin numbers are infinite. 

First, we note $\psi_Y P^{E_1}\psi_Y$ corresponds to the operator $P^X$ via the isomorphism $L^2(E_1|_{\{e\} \times X}) \cong L^2(\widetilde E_1)$, and so 
\begin{align*}
    \Tr_e(e^{-t\Delta_{E_1}}-P^{E_1}) &= \Tr\big(\psi_Y(e^{-t\Delta_{E_1}}-P^{E_1})\psi_Y\big)\\
    &= \Tr(\psi_Ye^{-t\Delta_{E_1}}\psi_Y) - \Tr(P^X).
\end{align*}
Now, $\psi_Ye^{-t\Delta_{E_1}}\psi_Y = e^{-t\Delta_{E_1}}|_{\{e\} \times X}$, and we claim that $e^{-t\Delta_{E_1}}|_{\{e\} \times X} = e^{-t\Delta_X}$. To prove this we note that $\Delta_{E_1} = \bigoplus_{n \in \Z} \Delta_X$ on the domain $\bigoplus_{n \in \Z} W^2(E_1|_{\{n\} \times X})$. Now as $e^{-t\Delta_{E_1}}$ is diagonal with respect this decomposition, and as functional calculus commutes with direct sums, the result follows.
\end{proof}

\begin{corollary}\thlabel{Prop: g=e Novikov-Shubin invariants 2}
The delocalised Novikov--Shubin numbers for the suspension of $Y$ are infinite. 
\end{corollary}

\begin{proof}
Replace $X$ in \thref{Prop: g=e Novikov-Shubin invariants} with the suspension of $X$ by $T$.  
\end{proof}

From \thref{Prop: g=e Novikov-Shubin invariants} and \thref{Prop: g=e Novikov-Shubin invariants 2}, it follows that the equivariant analytic torsion is well-defined in this case. Therefore by \thref{Cor: Equivariant Fried Identity} the equivariant Fried conjecture holds.

\appendix
\section{Relation to Classical Suspension Flow}\label{App: Relation Classical}

Here we relate our results to those presented in Section 3 of the overview paper on the Fried conjecture \cite{Analytic.Torsion.Fried.Survey}.

First recall the definition of the superdeterminant.

\begin{definition}
Suppose $V^\bullet$ is a $\Z$-graded vector space, with $V^k$ finite dimensional for all $k \in \Z$, and only finitely many $V^k$ are non-trivial. If $A$ is linear map acting on $V^\bullet$, which preserves the degree, then the \textbf{superdeterminant} of $A$ is
\[
    \Sdet^{V^\bullet}(A) = \prod_{q \in \Z}\det\big(A|_{V^q}\big)^{(-1)^q}.
\]
\end{definition}

For the Ruelle dynamical $\zeta$-function we have the following result, which is Theorem 3.4 in \cite{Analytic.Torsion.Fried.Survey}.

\begin{corollary}\thlabel{Cor: Classical Suspension Ruelle}
Assume $H^\bullet(M,E) = 0$. For all $\sigma \in \C$ with large enough real part, the Ruelle dynamical $\zeta$-function for equivariant suspension flow at the identity is given by
\begin{equation}\label{Eq: Ruelle Suspension Identity}
     R^e_{\vphi, \nabla^E}(\sigma) = \big(\Sdet^{H^\bullet(Y,E_1)}(1 - e^{-\sigma}T^*)\Sdet^{H^\bullet(Y,E_1)}(1 - e^{-\sigma}(T^{-1})^*)\big)^{-1},
\end{equation}
where $T^*, (T^{-1})^*$ are the maps on $H^\bullet(Y,E_1)$ induced by $T$ and $T^{-1}$, respectively. 
\end{corollary}

\begin{proof}
From \thref{Prop: Ruelle Suspension Compact}, when $\re(\sigma) > 0$ is large enough, we have
\[
    R^e_{\vphi,\nabla^E}(\sigma) = \exp\Big(\sum_{n \in \Z \setminus \{0\}}\frac{e^{-\abs{n}\sigma}}{\abs{n}}\Tr^{H^\bullet(Y,E_1)}\big((-1)^F(T^*)^n\big)\Big).
\]
Let $\mu_{j,q}$ be the eigenvalues of $T^*$ acting on $H^q(Y,E_1)$, counted with multiplicity. Using the fact that $T^*$ is an isometry and so $\abs{\mu_{j,q}} = 1$ for all $j$ and $q$, together with the Taylor series for the function $x \mapsto \log(1 + x)$, we see that
\begin{equation}\label{Eq: Classical Suspension Ruelle Proof 1}
    R^e_{\vphi,\nabla^E}(\sigma) = \prod_{q=0}^{\dim Y}\Big(\prod_j \big((1 - e^{-\sigma}\mu_{j,q})(1 - e^{-\sigma}\mu_{j,q}^{-1})\big)^{-1}\Big)^{(-1)^q}.
\end{equation}
The result now follows as \eqref{Eq: Classical Suspension Ruelle Proof 1} equals \eqref{Eq: Ruelle Suspension Identity}. \qedhere
\end{proof}

Considering now the torsion side, we have the following which is Proposition 3.3 in \cite{Analytic.Torsion.Fried.Survey}.

\begin{corollary}\thlabel{Cor: Torsion Suspension Identity}
Suppose that $G$ is a compact group and $E$ is an acyclic vector bundle. Then
\begin{equation}\label{Eq: Classical Torsion Suspension 2}
    T_e(\nabla^{E}) = \abs{\Sdet^{H^\bullet(Y,E_1)}(1 - T^*)}^{-1}.
\end{equation}
\end{corollary}

\begin{proof}
Again, let $\mu_{j,q}$ be the eigenvalues of $T^*$ acting on $H^q(Y,E_1)$. From \thref{Prop: Suspension Torsion Compact},
\begin{equation}\label{Eq: Torsion Suspension Identity Proof}
    T_e(\nabla^{E}) = \exp\left(\sum_{n \in \Z \setminus\{0\}} \frac{\Tr^{H^\bullet(Y,E_1)}\big((-1)^F(T^*)^n\big)}{2|n|}\right).
\end{equation}
So, following the proof of \thref{Prop: Suspension Torsion Compact} we find that \eqref{Eq: Torsion Suspension Compact Proof 3} becomes
\begin{equation}
    T_e(\nabla^E) = \prod_{q=0}^{\dim Y}\left(\prod_{j}\abs{1-\mu_{j,q}}^{-1}\right)^{(-1)^q}.
\end{equation}
which equals \eqref{Eq: Classical Torsion Suspension 2}.
\end{proof}

Moreover, note that if $\sigma \in \R$ then from \thref{Cor: Classical Suspension Ruelle}
\[
    R^e_{\vphi,\nabla^E}(\sigma) = \abs{\Sdet^{H^\bullet(Y,E_1)}(1 - e^{-\sigma}T^*)}^{-2},
\]
which immediately implies that $R^e_{\vphi,\nabla^E}(0) = T_e(\nabla^E)^2$. Hence we recover Corollary 3.5 in \cite{Analytic.Torsion.Fried.Survey}. 

Finally we note that the ease of our proof compared to the proof presented in Corollary 3.5 of \cite{Analytic.Torsion.Fried.Survey} is due to the assumption that $T$ is an isometry, and so (3.23) in \cite{Analytic.Torsion.Fried.Survey} is immediate.

\printbibliography[title=References]
\addcontentsline{toc}{section}{References}

\end{document}